\RequirePackage{fix-cm}
\documentclass[smallextended]{svjour3}
\smartqed

\usepackage[utf8]{inputenc}
\usepackage[T1]{fontenc}
\usepackage{amsmath,amssymb,amsfonts}
\usepackage{graphicx}
\usepackage{booktabs}
\usepackage{xcolor}
\usepackage{bm}
\usepackage{algorithm}
\usepackage{algorithmicx}
\usepackage{algpseudocode}
\usepackage{float}
\usepackage{placeins}
\usepackage{listings}
\usepackage{hyperref}
\usepackage{adjustbox}
\hypersetup{hidelinks}

\graphicspath{{figures/}}

\newcommand{\RR}{\mathbb{R}}
\newcommand{\NN}{\mathbb{N}}
\newcommand{\ZZ}{\mathbb{Z}}
\newcommand{\TT}{\mathcal{T}}
\newcommand{\PP}{\mathcal{P}}
\newcommand{\HH}{\mathcal{H}}
\newcommand{\bs}[1]{\boldsymbol{#1}}
\newcommand{\norm}[1]{\left\lVert#1\right\rVert}
\newcommand{\abs}[1]{\left\lvert#1\right\rvert}

\allowdisplaybreaks

\begin{document}

\title{Scattered data interpolation on the torus by compactly supported multinode Shepard operators}
\titlerunning{Compactly supported multinode Shepard interpolation on the torus}

\author{Francesco Dell'Accio\thanks{Corresponding author: Francesco Dell'Accio.} \and Filomena Di Tommaso \and Rossana Lammirato \and Francesco Larosa}
\authorrunning{Dell'Accio et al.}

\institute{Francesco Dell'Accio \at
Department of Mathematics and Computer Science, University of Calabria, via P. Bucci cubo 30 A, 87036 Rende (CS), Italy\\
\email{francesco.dellaccio@unical.it}
\and
Filomena Di Tommaso \at
Department of Mathematics and Computer Science, University of Calabria, via P. Bucci cubo 30 A, 87036 Rende (CS), Italy\\
\email{filomena.ditommaso@unical.it}
\and
Rossana Lammirato \at
Department of Mathematics and Computer Science, University of Calabria, via P. Bucci cubo 30 A, 87036 Rende (CS), Italy\\
\email{lammiratorossana02@gmail.com}
\and
Francesco Larosa \at
Department of Mathematics and Computer Science, University of Calabria, via P. Bucci cubo 30 A, 87036 Rende (CS), Italy\\
\email{francesco.larosa@unical.it}
}

\date{Received: date / Accepted: date}

\maketitle

\begin{abstract}
We introduce a compactly supported multinode Shepard operator for the interpolation of scattered data on the torus embedded in $\RR^3$. The method combines local polynomial interpolation of total degree $d\in \NN$ with compactly supported Shepard-type weights, so that the approximation at each evaluation point depends only on neighbouring stencils of nodes. The torus is treated as an algebraic surface defined by a quartic polynomial, and Gr\"obner bases are used to construct reduced polynomial spaces on the surface by removing the redundancy induced by the defining equation. This yields local Vandermonde systems adapted to the toroidal geometry. We discuss the metric structure of the torus, show the local equivalence between the periodic parameter distance and the Euclidean distance inherited from $\RR^3$, and use this equivalence to motivate the compact support construction. We establish a uniform error estimate in terms of the maximal support radius and the local Lebesgue constants. Under uniform locality and stability assumptions, the method converges with order $d+1$ with respect to the fill distance. Numerical experiments on analytical test functions confirm polynomial reproduction and exhibit an error decay consistent with the theoretical analysis. The approach is also tested on Computational Fluid Dynamics data mapped onto the torus, including the interpolation of the velocity components, the reconstruction of the velocity magnitude from the interpolated components, and the reconstruction of a tangent velocity field through an orthonormal lifting of the interpolated components.
\end{abstract}

\section{Introduction}

Scattered data interpolation is a fundamental problem in numerical
analysis and scientific computing. It consists of reconstructing an
unknown function from values prescribed at irregularly distributed
points, a situation that naturally arises in geophysics, computer
graphics, fluid dynamics, medical imaging, and numerical simulations
on complex geometries \cite{Wendland}. Meshfree methods are particularly attractive in
this setting, since they do not require an underlying triangulation or
structured grid and can therefore accommodate irregular sampling and
non-Euclidean geometries.

Among meshfree interpolation techniques, Shepard-type methods have
received considerable attention because of their simplicity,
flexibility, and robustness with respect to scattered node
distributions \cite{DiTommaso,Shepard}. The classical Shepard operator constructs the
interpolant as a normalized weighted average of the data, with weights
depending on inverse powers of the distances from the evaluation point.
Although interpolatory, the classical construction reproduces only
constants and consequently has limited approximation order.
Multinode Shepard operators overcome this limitation by replacing
individual data values with local polynomial interpolants constructed
on suitably selected stencils of nodes. The normalized multinode
weights retain the partition-of-unity structure, while the local
polynomials provide higher-order reproduction and approximation.

The present work is conceived as a direct continuation of our previous
construction of multinode Shepard operators on the sphere
\cite{DellAccioSphere}. In that setting, the local
approximation space is the restriction to $\mathbb S^2$ of ambient
trivariate polynomials of total degree at most $d$,
\[
  \mathcal H_d(\mathbb S^2)
  =
  \{p|_{\mathbb S^2}:p\in\mathbb P_d(\mathbb R^3)\}.
\]
This restriction space is represented through the decomposition into
spherical harmonics. Minimal unisolvent stencils are extracted from
local candidate sets by partial row pivoting in a $PA=LU$
factorization, yielding Leja-type interpolation stencils. The
corresponding local interpolants are then blended through multinode
Shepard functions based on the spherical geodesic distance. The
resulting operator is interpolatory, reproduces
$\mathcal H_d(\mathbb S^2)$, and achieves approximation order $d+1$
under suitable geometric assumptions.

The extension from the sphere to the torus retains this interpolatory
multinode framework, but it is not formal. The underlying principle is
again to use restrictions of ambient polynomials,
\[
  \mathcal H_d(\mathbb T)
  =
  \{p|_{\mathbb T}:p\in\mathbb P_d(\mathbb R^3)\},
\]
and to construct minimal local interpolants on unisolvent stencils of
scattered nodes. The algebraic representation of this space, however,
is substantially different. On the sphere, the quadratic relation
$x^2+y^2+z^2-1=0$ is naturally handled through harmonic
decomposition. The torus is instead defined by a quartic polynomial
$F$, and distinct ambient polynomials may have the same restriction
whenever they differ by a multiple of $F$. Thus,
\[
  \mathcal H_d(\mathbb T)
  \simeq
  \frac{\mathbb P_d(\mathbb R^3)}
       {\mathbb P_d(\mathbb R^3)\cap\langle F\rangle}.
\]
We use Gröbner bases and normal forms to construct an explicit
non-redundant monomial basis of this quotient space. In this sense, the
Gröbner-reduced toroidal basis plays the role that the spherical
harmonic representation played in the spherical construction.

A second difficulty concerns the metric entering the multinode
weights. On $\mathbb S^2$, the geodesic distance has the explicit
formula
\[
  d_{\mathbb S^2}(x,y)=\arccos(x^Ty).
\]
No equally convenient expression is available for geodesic distances
on the torus. The toroidal parametrization instead induces a flat
periodic distance on the parameter domain, while the embedded surface
inherits the Euclidean distance from $\mathbb R^3$. We prove that these
two distances are locally equivalent. This makes it possible to use
the ambient Euclidean distance in a genuinely local construction,
provided that the supports are sufficiently small.

The third development with respect to the spherical method is the
introduction of compact support. The multinode Shepard operator on the
sphere is global: every interpolation stencil contributes, with
different intensity, at every evaluation point. Here the inverse
distance factors are truncated by stencil-dependent support radii, and
only the locally active stencils enter the approximation. The support
radii are chosen in terms of the stencil diameters and an upper bound
for the fill distance, which guarantees coverage of the torus while
preserving interpolation and polynomial reproduction. The resulting
compactly supported multinode Shepard operator is therefore both
interpolatory and genuinely local.

Recent work has established a general theory of stable local
polynomial reproductions on Riemannian and algebraic manifolds and has
applied it to coordinate-free moving least squares approximation
\cite{hangelbroek2026}. That framework provides existence,
locality, stability, and regularity results for reproducing shape
functions constructed from generally overdetermined local point sets.
The present construction is complementary in nature. We consider an
interpolatory multinode Shepard operator built from minimal square
stencils, whose cardinality is exactly
$\dim\mathcal H_d(\mathbb T)$. The polynomial redundancy is removed
algebraically before the numerical computation, rather than detected
through a numerical rank criterion, and the local interpolants are
combined by compactly supported multinode products of distances.
Thus, the distinctive features of the present method are exact
algebraic reduction, minimal interpolation stencils, multinode
Shepard blending, and nodal interpolation.

The main contributions of this work are the following. First, we
construct the restriction spaces $\mathcal H_d(\mathbb T)$ by means of
the quotient defined by the quartic torus equation and obtain explicit
Gröbner-reduced bases and dimensions. Second, we develop a compactly
supported multinode Shepard operator using minimal unisolvent stencils
selected by a Leja-type $PA=LU$ procedure. Third, we justify the use of
the ambient Euclidean distance through its local equivalence with the
periodic parameter distance and prove that the chosen support radii
provide complete coverage. Fourth, using a smooth normal extension to
a tubular neighbourhood and a local ambient Taylor argument, we
derive the estimate
\[
  \|f-\widetilde       {\mathcal{M}}_{\mu,h}f\|_{L^\infty(\mathbb T)}
  \le
  C\,\Lambda_h\,\rho_h^{d+1}
  \|f\|_{C^{d+1}(\mathbb T)},
\]
which yields order $d+1$ with respect to the fill distance under
uniform locality and stability assumptions. Finally, we provide a
numerical investigation of reproduction, degree enrichment,
conditioning, Lebesgue factors, and sensitivity to data
perturbations, and we apply the method to Computational Fluid Dynamics
data mapped onto the torus, including the reconstruction of a tangent
velocity field.

The paper is organized as follows. Section~2 develops the
Gröbner-based construction of polynomial restriction spaces on the
torus. Section~3 introduces the compactly supported multinode Shepard
operator, discusses the relevant distances, and establishes the error
and convergence estimates. Section~4 presents the analytical
experiments and the numerical stability diagnostics. Section~5
describes the application to CFD data and the tangent-field
reconstruction. Section~6 summarizes the stencil-selection and
implementation procedures, and the final section contains concluding
remarks.

\section{Gr\"obner-based polynomial spaces on the torus}
\label{sec:prelim}

The construction of suitable local polynomial spaces on the torus requires particular attention. Indeed, starting from degree $4$, the algebraic relation defining the
surface introduces dependencies among trivariate polynomials: distinct
polynomials of total degree at most $d$ may determine the same function
when restricted to the torus. Consequently, the direct use of standard polynomial bases leads to redundant degrees of freedom and potentially singular interpolation
systems. To obtain a representation adapted to the surface, we describe
polynomial restrictions through the quotient by the defining ideal of the
torus. Gr\"obner bases and normal forms then provide an explicit reduced polynomial basis~\cite{Cox}.
More precisely, for $R>r>0$, let us consider the standard torus embedded in $\mathbb{R}^3$,
\begin{equation*}
\TT
=
\left\{
(x,y,z)\in\RR^3:
\left(\sqrt{x^2+y^2}-R\right)^2+z^2=r^2
\right\}.
\end{equation*}
A smooth parametrization of $\TT$ is
\begin{equation}
\label{eq:param}
\Phi(\theta,\phi)
=
\bigl(
(R+r\cos\theta)\cos\phi,\,
(R+r\cos\theta)\sin\phi,\,
r\sin\theta
\bigr),
\qquad
(\theta,\phi)\in[0,2\pi)^2.
\end{equation}
The periodicity relations
\[
(\theta,\phi)\sim(\theta+2\pi,\phi),
\qquad
(\theta,\phi)\sim(\theta,\phi+2\pi)
\]
identify the torus with the quotient space
$\RR^2/(2\pi\ZZ)^2$.
Equivalently, the torus can be represented as the real algebraic surface
\[
\TT
=
\left\{
(x,y,z)\in\RR^3:F(x,y,z)=0
\right\}, 
\]
where
\begin{equation}
\label{eq:torus_F}
F(x,y,z)
=
\left(x^2+y^2+z^2+R^2-r^2\right)^2
-
4R^2(x^2+y^2).
\end{equation}

Let $\PP_d(\RR^3)$ denote the space of trivariate polynomials of
total degree at most $d$. The polynomial space relevant to interpolation on $\TT$ is the restriction space
\begin{equation*}
\HH_d(\TT)
=
\left\{
p|_{\TT}:p\in\PP_d(\RR^3)
\right\}.
\end{equation*}
The representation of an element of $\HH_d(\TT)$ is not unique. In fact, if two polynomials differ by a multiple of $F$,
then they have the same restriction to $\TT$. Moreover, the polynomial
$F$ is irreducible and its real zero set contains nonsingular points;
therefore the vanishing ideal of the real torus is generated by $F$~\cite{BochnakCosteRoy}. Consequently,
\begin{equation*}
\HH_d(\TT)
\simeq
\frac{\PP_d(\RR^3)}
{\PP_d(\RR^3)\cap\langle F\rangle}.
\end{equation*}

To make the quotient representation explicit, we briefly recall the algebraic notion underlying the  construction. Let
$S=\RR[x,y,z]$ denote the polynomial ring in the variables $x,y,z$ with
real coefficients. We equip $S$ with the lexicographic monomial order
induced by
\[
x\succ y\succ z.
\]
For every nonzero polynomial $p\in S$, we denote by
$\operatorname{LM}(p)$ its leading monomial with respect to this order. A finite set
\[
G=\{g_1,\ldots,g_s\}\subset I
\]
is called a Gr\"obner basis of an ideal $I\subset S$ with respect to the
chosen monomial order if
\begin{equation*}
\left\langle
\operatorname{LM}(g_1),\ldots,
\operatorname{LM}(g_s)
\right\rangle
=
\left\langle
\operatorname{LM}(p):
p\in I\setminus\{0\}
\right\rangle .
\end{equation*}
Equivalently, the leading monomial of every nonzero polynomial in $I$ is
divisible by the leading monomial of at least one element of $G$. This
property makes reduction modulo $I$ constructive: every residue class in
$S/I$ admits a unique normal-form representative expressed as a linear
combination of monomials not divisible by any
$\operatorname{LM}(g_j)$; see \cite{Cox}.

With respect to the lexicographic order fixed above, the leading monomial of the defining polynomial $F$ in \eqref{eq:torus_F} is
\[
\operatorname{LM}(F)=x^4.
\]
Since $\langle F\rangle$ is a principal ideal, the singleton $\{F\}$ is a
Gr\"obner basis of $\langle F\rangle$. It follows that every polynomial
has a unique normal form modulo $F$, expressed as a linear combination of
monomials not divisible by $x^4$.

\begin{proposition}
\label{prop:torus_basis}
For every $d\geq0$, the restrictions to $\TT$ of the monomials
\begin{equation}
\label{eq:basis_torus}
\mathcal B_d
=
\left\{
x^iy^jz^k:
0\leq i\leq3,\quad i+j+k\leq d
\right\}
\end{equation}
form a basis of $\HH_d(\TT)$. Consequently, if $m_d:=\dim\HH_d(\TT)$,
we have
\begin{equation}
\label{eq:dim_torus}
m_d
=
\begin{cases}
\displaystyle \binom{d+3}{3},
& 0\leq d\leq3,\\[2ex]
\displaystyle
\binom{d+3}{3}-\binom{d-1}{3}
=
2(d^2+1),
& d\geq4.
\end{cases}
\end{equation}
\end{proposition}
\begin{proof}
It is well known that
\[
\dim \PP_d(\RR^3)=\binom{d+3}{3}.
\]
Let
\[
\mathcal R_d:=\operatorname{span}\mathcal B_d
\subset \PP_d(\RR^3).
\]
Since $\operatorname{LM}(F)=x^4$, the multivariate division algorithm with
respect to the Gr\"obner basis $\{F\}$ implies that every polynomial
$p\in\PP_d(\RR^3)$ admits a unique decomposition
\[
p=Fg+r,
\]
where $g\in\PP_{d-4}(\RR^3)$ when $d\geq4$, and
$r\in\mathcal R_d$ contains no monomial divisible by $x^4$. For $d<4$,
the quotient term is absent and $r=p$.

The decomposition is direct. Indeed, if
\[
p\in F\PP_{d-4}(\RR^3)\cap\mathcal R_d,
\]
then $p=Fg$ for some $g\in\PP_{d-4}(\RR^3)$, while $p$ is also a remainder
with respect to division by $F$. Since $\{F\}$ is a Gr\"obner basis, every element of $\langle F\rangle$ has zero normal form modulo $F$. It follows that
$p=0$. Hence, for $d\geq4$,
\[
\PP_d(\RR^3)
=
F\PP_{d-4}(\RR^3)\oplus\mathcal R_d.
\]

Because $F$ vanishes identically on $\TT$, the restriction of every
$p\in\PP_d(\RR^3)$ coincides with the restriction of its unique remainder
$r\in\mathcal R_d$. Therefore, the restrictions to $\TT$ of the monomials
in $\mathcal B_d$ span $\HH_d(\TT)$. Moreover, if an element of
$\mathcal R_d$ represents the zero class in
\[
\frac{\PP_d(\RR^3)}
{\PP_d(\RR^3)\cap\langle F\rangle},
\]
then it belongs to
$F\PP_{d-4}(\RR^3)\cap\mathcal R_d$ and is therefore zero. Thus the
restrictions of the monomials in $\mathcal B_d$ are linearly independent
and form a basis of $\HH_d(\TT)$.

For $0\leq d\leq3$, no nonzero multiple of the quartic polynomial $F$
belongs to $\PP_d(\RR^3)$. Hence
\[
m_d=\dim\HH_d(\TT)=\dim\PP_d(\RR^3)
=\binom{d+3}{3}.
\]

For $d\geq4$, consider the linear map
\[
M_F:\PP_{d-4}(\RR^3)\longrightarrow\PP_d(\RR^3),
\qquad
g\longmapsto Fg.
\]
Since $S=\RR[x,y,z]$ is an integral domain and $F\neq0$, the map $M_F$ is
injective. Therefore,
\[
\dim\bigl(F\PP_{d-4}(\RR^3)\bigr)
=
\dim\PP_{d-4}(\RR^3)
=
\binom{d-1}{3}.
\]
Taking dimensions in the direct-sum decomposition gives
\[
m_d
=
\dim\PP_d(\RR^3)-\dim\PP_{d-4}(\RR^3)
=
\binom{d+3}{3}-\binom{d-1}{3}.
\]
A direct simplification yields
\[
\binom{d+3}{3}-\binom{d-1}{3}
=
2(d^2+1),
\]
which proves \eqref{eq:dim_torus}. $\blacksquare$
\end{proof}

For instance, the dimensions corresponding to the degrees considered in
the numerical experiments are
\[
m_1=4,\quad m_2=10,\quad m_3=20,\quad
m_4=34,\quad m_5=52,\quad m_6=74.
\]
The construction above can be viewed within the general framework of
Buchberger's algorithm, which provides a constructive procedure for
computing a Gr\"obner basis from a finite set of generators of a
polynomial ideal. Starting from the given generators, the algorithm forms
their $S$-polynomials, reduces them with respect to the current generating
set, and adjoins every nonzero remainder. The procedure terminates when all the relevant $S$-polynomials reduce to zero; by Buchberger's
criterion, the resulting set is then a Gr\"obner basis. In the present
case, the ideal $\langle F\rangle$ is principal, and hence the singleton
$\{F\}$ is already a Gr\"obner basis, so that no nontrivial iteration of
the algorithm is required.
For completeness, Appendix~\ref{app:buchberger} reports the Wolfram Mathematica routine used to generate the reduced monomial basis $\mathcal{B}_d$.

We are now ready to formulate the interpolation problem on the torus $  \mathcal{T}$. Let $\Xi =
\{\bs{x}_1,\ldots,\bs{x}_{m_d}\}
\subset\TT$ be a set of pairwise distinct points, and let $\mathcal B_d
=
\{\beta_1,\ldots,\beta_{m_d}\}$ be a fixed ordering of the basis in \eqref{eq:basis_torus}. Given data
$f_i=f(\bs{x}_i)$, $i=1,\ldots,m_d$, we seek an interpolant
\begin{equation*}
P[f](\bs{x})
=
\sum_{k=1}^{m_d}c_k\beta_k(\bs{x})
\end{equation*}
satisfying
\begin{equation*}
P[f](\bs{x}_i)
=
f_i,
\qquad
i=1,\ldots,m_d.
\end{equation*}
The interpolation conditions lead to the Vandermonde
system
\begin{equation*}
V\bs{c}=\bs{f},
\qquad
V_{ik}
=
\beta_k(\bs{x}_i),
\end{equation*}
where
\[
\bs{c}
=
(c_1,\ldots,c_{m_d})^T,
\qquad
\bs{f}
=
(f_1,\ldots,f_{m_d})^T.
\]
The set $\Xi$ is said to be unisolvent for $\HH_d(\TT)$ if the matrix
$V$ is nonsingular. In this case, the interpolation problem admits a
unique solution for every data vector $\bs{f}$.

\section{The multinode Shepard framework}
\label{sec:multinode-framework}

The multinode Shepard construction combines local polynomial interpolants
through a partition of unity. More precisely, let
$X=\{\bs{x}_1,\ldots,\bs{x}_n\}\subset\TT$ be a set of pairwise distinct
interpolation nodes, and let $f_i=f(\bs{x}_i)$, $i=1,\ldots,n$, be the
corresponding data values. We consider a family
$\Sigma=\{\sigma_1,\ldots,\sigma_L\}$ of local stencils of $X$ such that
\[
X=\bigcup_{j=1}^{L}\sigma_j.
\]
Each stencil has the form
\[
\sigma_j
=
\{\bs{x}_{j_1},\ldots,\bs{x}_{j_{m_d}}\}
\subset X
\]
and is assumed to be unisolvent for interpolation on $\HH_d(\TT)$.
Therefore, for every $j=1,\ldots,L$, there exists a unique local
interpolant $P_j[f]\in\HH_d(\TT)$ satisfying
\[
P_j[f](\bs{x}_{j_\ell})
=
f(\bs{x}_{j_\ell}),
\qquad
\ell=1,\ldots,m_d.
\]
The corresponding multinode Shepard approximation is defined by
\[
\mathcal M_{\mu}[f](\bs{x})
=
\begin{cases}
\displaystyle
\sum_{j=1}^{L}
W_{\mu,j}(\bs{x})P_j[f](\bs{x}),
& \bs{x}\in\TT\setminus X,\\[3ex]
f(\bs{x}_i),
& \bs{x}=\bs{x}_i,\quad i=1,\ldots,n,
\end{cases}
\]
where the nonnegative weight functions satisfy
\[
\sum_{j=1}^{L}W_{\mu,j}(\bs{x})=1,
\qquad
\bs{x}\in\TT\setminus X.
\]
The algebraic construction developed in the previous section determines
the local polynomial interpolants $P_j[f]$. The definition of the weights,
on the other hand, requires a suitable notion of distance on the torus.

More precisely, for each unisolvent stencil $\sigma_j$, we introduce the
multinode distance product
\[
D_j(\bs{x})
=
\prod_{\ell=1}^{m_d}
d(\bs{x},\bs{x}_{j_\ell}),
\]
where $d(\cdot,\cdot)$ is a distance adapted to the underlying domain.
The associated multinode Shepard functions, for $\bs{x}\notin X$, are
defined as
\begin{equation}
\label{eq:Wmuj}
W_{\mu,j}(\bs{x})
=
\frac{D_j(\bs{x})^{-\mu}}
{\displaystyle\sum_{k=1}^{L}D_k(\bs{x})^{-\mu}},
\qquad
\mu>0.
\end{equation}
Thus, the choice of the distance is an essential part of the construction,
since it determines the influence of each local interpolation stencil.

In Euclidean domains, the natural choice is the Euclidean distance
\cite{dellaccio2019rate}. On the sphere, it is natural to replace it with
the geodesic distance, which is explicitly available and reflects the
intrinsic geometry of the surface \cite{DellAccioSphere}. On the torus,
the situation is more delicate. The periodic parameter domain is the
quotient
\[
\mathbb T^2=\RR^2/(2\pi\ZZ)^2,
\]
obtained by identifying parameter pairs that differ by integer multiples
of $2\pi$ in either coordinate. The flat periodic distance is the quotient
distance induced by the Euclidean norm on $\RR^2$; equivalently, it
measures the shortest Euclidean distance among all periodic
representatives of two parameter points. Here, the adjective ``flat''
refers to the Euclidean geometry of the parameter domain and not to the
geometry of the torus embedded in $\RR^3$.
The interpolation nodes, however, are represented on the embedded surface
$\TT\subset\RR^3$, where proximity can be measured using the Euclidean
distance inherited from the ambient space. The periodic parameter
distance and the ambient Euclidean distance may behave differently at a
global scale.

To overcome this difficulty, we adopt a local construction based on
compactly supported multinode weights. The following section shows that,
in a sufficiently small neighborhood of every point of the torus, the
periodic distance in the parameter domain and the Euclidean distance
between the corresponding embedded points are equivalent. This result
provides the theoretical justification for using the Euclidean distance
in the compactly supported weights, provided that their supports are
sufficiently small.

\subsection{Local comparison of distances}
\label{sec:toruspoly}
 By the periodicity of the
parametrization in \eqref{eq:param}, the map $\Phi$ induces a
well-defined smooth parametrization
\[
\Phi:\mathbb T^2\longrightarrow\TT,
\]
which we denote by the same symbol.
The flat periodic distance described above is defined by
\begin{equation*}
d_{\mathrm{per}}(p,q)
=
\min_{k\in(2\pi\ZZ)^2}
\norm{\widetilde p-\widetilde q+k}_2,
\qquad
p,q\in\mathbb T^2,
\end{equation*}
where $\widetilde p,\widetilde q\in\RR^2$ are any representatives of the equivalence classes $p$ and $q$, respectively. On the embedded torus, we consider instead
the Euclidean distance inherited from the ambient space:
\begin{equation*}
d_E(P,Q)
=
\norm{P-Q}_2,
\qquad
P,Q\in\TT.
\end{equation*}
The two distances can be compared through the parametrization $\Phi$.
Although they may assign substantially different values to points that
are far apart in the parameter domain, they are locally equivalent.

\begin{proposition}
\label{prop:local-equivalence}
For every $p\in\mathbb T^2$, there exist a neighborhood
$U\subset\mathbb T^2$ of $p$ and constants $c_1,c_2>0$ such that
\begin{equation*}
c_1d_{\mathrm{per}}(p,q)
\leq
d_E\bigl(\Phi(p),\Phi(q)\bigr)
\leq
c_2d_{\mathrm{per}}(p,q),
\qquad
q\in U.
\end{equation*}
\end{proposition}

\begin{proof}
Let $p\in\mathbb T^2$, and let $\widetilde p\in\RR^2$ be a
representative of $p$. Since the quotient projection identifies points
that differ by elements of $(2\pi\ZZ)^2$, we can choose a sufficiently
small neighborhood $U$ of $p$ such that every $q\in U$ admits a unique
representative $\widetilde q$ close to $\widetilde p$. In this
neighborhood, the periodic distance is given by
\[
d_{\mathrm{per}}(p,q)
=
\norm{\widetilde p-\widetilde q}_2.
\]
Since the parametrization $\Phi$ in~\eqref{eq:param} is smooth and regular, its differential $D\Phi$ has rank $2$ at every point. Therefore, after possibly restricting to a smaller
neighborhood $V$ of $\widetilde p$, there exist constants $a,b>0$ such
that
\begin{equation}\label{upperbound}
a\norm{v}_2
\leq
\norm{D\Phi(\xi)v}_2
\leq
b\norm{v}_2,
\qquad
\forall\,\xi\in V,\quad
\forall\,v\in\RR^2.
\end{equation}
For $\widetilde q\in V$, we have
\[
\Phi(\widetilde q)-\Phi(\widetilde p)
=
\int_0^1
D\Phi\bigl(
\widetilde p+t(\widetilde q-\widetilde p)
\bigr)
(\widetilde q-\widetilde p)\,dt.
\]
Using~\eqref{upperbound}, we obtain
\[
d_E\bigl(\Phi(p),\Phi(q)\bigr)
=
\norm{\Phi(\widetilde q)-\Phi(\widetilde p)}_2
\leq
b\norm{\widetilde q-\widetilde p}_2
=
b\,d_{\mathrm{per}}(p,q).
\]
Moreover, from the smoothness and regularity of $\Phi$, it follows that $\Phi$ is locally
an embedding. Thus, after possibly reducing $V$, the inverse map
\[
\Phi^{-1}:\Phi(V)\longrightarrow V
\]
is smooth. Hence, $\Phi^{-1}$ is locally Lipschitz: there exists a
constant $C>0$ such that
\[
\norm{\widetilde q-\widetilde p}_2
\leq
C\norm{\Phi(\widetilde q)-\Phi(\widetilde p)}_2.
\]
Therefore,
\[
d_{\mathrm{per}}(p,q)
\leq
C\,d_E\bigl(\Phi(p),\Phi(q)\bigr),
\]
or, equivalently,
\[
\frac{1}{C}d_{\mathrm{per}}(p,q)
\leq
d_E\bigl(\Phi(p),\Phi(q)\bigr).
\]
Setting $c_1=1/C$ and $c_2=b$, we obtain
\[
c_1d_{\mathrm{per}}(p,q)
\leq
d_E\bigl(\Phi(p),\Phi(q)\bigr)
\leq
c_2d_{\mathrm{per}}(p,q),
\qquad
\forall\,q\in U.
\]
This proves the local equivalence of the two distances. $\blacksquare$
\end{proof}
Note that the local equivalence does not mean that the two distances assign similar numerical values to arbitrary pairs of points. For example, consider $p=(0,0)$, $q=(\pi,0)$. Their periodic distance is
\[
d_{\mathrm{per}}(p,q)=\pi,
\]
whereas $\Phi(p)=(R+r,0,0)$, $\Phi(q)=(R-r,0,0)$, and hence
\[
d_E\bigl(\Phi(p),\Phi(q)\bigr)=2r.
\]
Thus, points separated by half a turn in the periodic parameter domain
may be considerably closer when viewed in the ambient space. This is the shortcut effect produced by the embedding of the torus in $\RR^3$.

\subsection{Compactly supported multinode Shepard operator}
\label{sec:csms}

The local equivalence established in
Proposition~\ref{prop:local-equivalence} motivates the use of the
Euclidean distance in a localized version of the multinode Shepard
construction. More precisely, the global inverse-distance
weights~\eqref{eq:Wmuj} are replaced by compactly supported weights, so
that only interpolation stencils lying sufficiently close to the
evaluation point contribute to the approximation.

In what follows, the weight functions are evaluated at points
$\bs{x}\in\TT\setminus X$, unless otherwise stated. For each stencil $\sigma_j$, let $R_j>0$ be a support radius and define
the unnormalized compact multinode function
\begin{equation}
\label{eq:compact-unnormalized-weight}
\omega_{\mu,j}(\bs{x})
=
\prod_{\ell=1}^{m_d}
\left(
\frac{1}{d_E(\bs{x},\bs{x}_{j_\ell})}
-
\frac{1}{R_j}
\right)_+^{\mu},
\qquad
\mu>0,
\end{equation}
where $(t)_+:=\max\{t,0\}$ is the positive part function.
The function $\omega_{\mu,j}(\bs{x})$ is positive if and only if
\[
d_E(\bs{x},\bs{x}_{j_\ell})<R_j,
\qquad
\ell=1,\ldots,m_d.
\]
Accordingly, the support of the unnormalized compact multinode
function~\eqref{eq:compact-unnormalized-weight} is
\begin{equation*}
C_j
=
\bigcap_{\ell=1}^{m_d}
\left\{
\bs{\xi}\in\TT:
d_E(\bs{\xi},\bs{x}_{j_\ell})\leq R_j
\right\}.
\end{equation*}
Notice that $C_j$ is independent of the parameter $\mu$. For each evaluation point $\bs{x}\in\TT$, we introduce the active index
set
\begin{equation*}
J_{\bs{x}}
=
\left\{
j\in\{1,\ldots,L\}:
\bs{x}\in\overset{\circ}{C}_j
\right\},
\end{equation*}
where $\overset{\circ}{C}_j$ denotes the interior of $C_j$ in $\TT$.
The practical choice of the support radii introduced below guarantees that $J_{\bs{x}}\neq\emptyset$ for every $\bs{x}\in\TT$; see Lemma~\ref{lem:coverage}. The compact
support multinode Shepard functions are then defined by
\begin{equation}
\label{eq:compact-weights}
\widetilde W_{\mu,j}(\bs{x})
=
\begin{cases}
\displaystyle
\frac{\omega_{\mu,j}(\bs{x})}
{\displaystyle\sum_{k\in J_{\bs{x}}}
\omega_{\mu,k}(\bs{x})},
& j\in J_{\bs{x}},\\[3ex]
0,
& j\notin J_{\bs{x}},
\end{cases}
\qquad
j=1,\ldots,L.
\end{equation}
As the multinode Shepard functions~\eqref{eq:Wmuj}, they are
nonnegative and satisfy the partition of unity property, i.e.,
\begin{equation*}
\widetilde W_{\mu,j}(\bs{x})\geq0, \qquad
\sum_{j=1}^{L}
\widetilde W_{\mu,j}(\bs{x})
=
\sum_{j\in J_{\bs{x}}}
\widetilde W_{\mu,j}(\bs{x})
=
1.
\end{equation*}
The compactly supported multinode Shepard operator on the
torus is defined by
\begin{equation}
\label{eq:csms-operator}
\widetilde{\mathcal M}_{\mu}[f](\bs{x})
=
\begin{cases}
\displaystyle
\sum_{j\in J_{\bs{x}}}
\widetilde W_{\mu,j}(\bs{x})P_j[f](\bs{x}),
& \bs{x}\in\TT\setminus X,\\[3ex]
f(\bs{x}_i),
& \bs{x}=\bs{x}_i,\quad i=1,\ldots,n.
\end{cases}
\end{equation}
Since the interpolation nodes are pairwise distinct, the operator~\eqref{eq:csms-operator} is well defined and directly enforces the
interpolation conditions. The radius $R_j$ must be large enough so that
$\overset{\circ}{C}_j$ contains all nodes of $\sigma_j$. A practical
choice is
\begin{equation}
\label{eq:support-radius}
R_j
=
\operatorname{diam}_E(\sigma_j)+H,
\qquad
\operatorname{diam}_E(\sigma_j)
=
\max_{\bs{x},\bs{y}\in\sigma_j}
d_E(\bs{x},\bs{y}),
\end{equation}
where $H>0$ is a strictly upper bound of the fill distance
\[
h_{X,\TT}
=
\sup_{\bs{x}\in\TT}
\min_{\bs{x}_i\in X}
d_E(\bs{x},\bs{x}_i).
\]

\begin{lemma}
\label{lem:coverage}
Assume that the family $\Sigma=\{\sigma_1,\ldots,\sigma_L\}$ covers the node set $X$ and that the radii are chosen according to~\eqref{eq:support-radius}, with $H>h_{X,\TT}$. Then
\[
J_{\bs{x}}\neq\emptyset,
\qquad
\bs{x}\in\TT.
\]
Consequently, the denominator in~\eqref{eq:compact-weights} is strictly positive at every evaluation point in $\TT\setminus X$.
\end{lemma}
\begin{proof}
Fix $\bs{x}\in\TT$. Since $H>h_{X,\TT}$, there exists a node $\bs{x}_i\in X$ such that
\[
d_E(\bs{x},\bs{x}_i)<H.
\]
Because the stencils in $\Sigma$ cover $X$, the node $\bs{x}_i$ belongs to at least one stencil $\sigma_j$. Hence, for every $\bs{x}_{j_\ell}\in\sigma_j$, the triangle inequality gives
\[
d_E(\bs{x},\bs{x}_{j_\ell})
\leq d_E(\bs{x},\bs{x}_i)+d_E(\bs{x}_i,\bs{x}_{j_\ell})
<H+\operatorname{diam}_E(\sigma_j)
=R_j.
\]
Thus $\bs{x}\in\overset{\circ}{C}_j$, and therefore $j\in J_{\bs{x}}$.
\end{proof}

Finally, the partition-of-unity property implies that the operator
reproduces the polynomial restriction space $\HH_d(\TT)$. Indeed, if
$p\in\HH_d(\TT)$, then the uniqueness of the local interpolation
problems yields
\[
P_j[p]=p,
\qquad
j=1,\ldots,L.
\]
Therefore, for every $\bs{x}\in\TT\setminus X$,
\[
\widetilde{\mathcal M}_{\mu}[p](\bs{x})
=
\sum_{j\in J_{\bs{x}}}
\widetilde W_{\mu,j}(\bs{x})p(\bs{x})
=
p(\bs{x})
\sum_{j\in J_{\bs{x}}}
\widetilde W_{\mu,j}(\bs{x})
=
p(\bs{x}).
\]
At the interpolation nodes, the same identity follows directly from~\eqref{eq:csms-operator}. Hence,
\[
\widetilde{\mathcal M}_{\mu}[p]=p,
\qquad
p\in\HH_d(\TT).
\]

\subsection{Error estimate and convergence}
\label{sec:error-estimate}

We now derive an error estimate for a family of compactly supported multinode Shepard operators of fixed polynomial degree $d$. Let $\{X_h\}_{h>0}$ be a family of node sets on $\TT$, with
\[
h:=h_{X_h,\TT}\longrightarrow0,
\]
and let $\Sigma_h=\{\sigma_{1,h},\ldots,\sigma_{L_h,h}\}$ be a family of unisolvent local stencils covering $X_h$. We denote by $P_{j,h}$ the corresponding local interpolation operator, by $R_{j,h}$ the support radius, by $J_{\bs{x},h}$ the active index set, and by $\widetilde W_{\mu,j,h}$ the normalized compact weights. The resulting compactly supported multinode Shepard operator is denoted by $\widetilde{\mathcal M}_{\mu,h}$.

For each stencil $\sigma_{j,h}=\{\bs{x}_{j_1,h},\ldots,\bs{x}_{j_{m_d},h}\}$, let $\ell_{j_1,h},\ldots,\ell_{j_{m_d},h}\in\HH_d(\TT)$ be the fundamental polynomials satisfying
\[
\ell_{j_k,h}(\bs{x}_{j_\ell,h})=\delta_{k\ell}=\begin{cases}
    1, & k=\ell,\\
    0, & k\neq \ell.
\end{cases}
\]
We introduce the maximal support radius
\begin{equation*}
\rho_h:=\max_{1\leq j\leq L_h}R_{j,h}
\end{equation*}
and the local stability factor
\begin{equation}
\label{eq:Lambda-h}
\Lambda_h
:=
\sup_{\bs{x}\in\TT}
\max_{j\in J_{\bs{x},h}}
\sum_{k=1}^{m_d}
\abs{\ell_{j_k,h}(\bs{x})}.
\end{equation}
The non-emptiness of $J_{\bs{x},h}$ follows from Lemma~\ref{lem:coverage} whenever the radii are chosen as in~\eqref{eq:support-radius}.

Since $\TT\subset\RR^3$ is a compact embedded smooth submanifold,
the tubular neighbourhood theorem and compactness imply that there
exists $\delta_0>0$ such that every $\bs{y}\in U_{\delta_0}(\TT)$ has
a unique nearest point $\pi(\bs{y})\in\TT$, where
\[
U_\delta(\TT)
:=
\left\{
\bs{y}\in\RR^3:
\operatorname{dist}(\bs{y},\TT)<\delta
\right\}.
\]
The resulting nearest-point projection
\[
\pi:U_{\delta_0}(\TT)\longrightarrow\TT
\]
is smooth. Indeed, the existence of a nearest point follows from
compactness; the displacement from a nearest point is normal to
$\TT$, and uniqueness follows, for $\delta_0$ sufficiently small,
from the injectivity of the normal map. See
\cite[Theorem~6.24 and Proposition~6.25]{Lee}, and
cf.\ \cite[Problem~6-5]{Lee}.

We equip $C^k(\TT)$ with the norm induced by the periodic
parametrization $\Phi:\mathbb T^2\to\TT$. More precisely, identifying
$f\circ\Phi$ with its $2\pi$-periodic lift to $\RR^2$, we set
\[
\norm{f}_{C^k(\TT)}
:=
\max_{|\alpha|\leq k}
\norm{
\partial^\alpha(f\circ\Phi)
}_{L^\infty([0,2\pi]^2)}.
\]
This norm is equivalent to every standard $C^k$-norm on the compact
manifold $\TT$.

Given $f\in C^{d+1}(\TT)$, we define its normal extension by
\[
Ef(\bs{y})
:=
f\bigl(\pi(\bs{y})\bigr),
\qquad
\bs{y}\in U_{\delta_0}(\TT).
\]
For every $0<\delta_1<\delta_0$, one has
\[
\overline{U_{\delta_1}(\TT)}
\Subset
U_{\delta_0}(\TT).
\]
Since $\Phi^{-1}\circ\pi$ is smooth, all its derivatives up to order
$d+1$ are uniformly bounded on
$\overline{U_{\delta_1}(\TT)}$. Writing
$Ef=(f\circ\Phi)\circ(\Phi^{-1}\circ\pi)$, the multivariate chain
rule therefore yields
\begin{equation}
\label{eq:extension-bound}
\norm{Ef}_{C^{d+1}(U_{\delta_1}(\TT))}
\leq
C_E\norm{f}_{C^{d+1}(\TT)},
\end{equation}
where $C_E>0$ is independent of $f$ and depends only on $\TT$, $d$,
$\delta_1$, and the fixed parametrization $\Phi$.

\begin{lemma}
\label{lem:ambient-taylor}
Let $0<\delta_1<\delta_0$. For $f\in C^{d+1}(\TT)$ and $\bs{x}\in\TT$, define the ambient Taylor polynomial
\begin{equation*}
Q_{d,\bs{x}}(\bs{y})
:=
\sum_{|\alpha|\leq d}
\frac{D^\alpha Ef(\bs{x})}{\alpha!}
(\bs{y}-\bs{x})^\alpha,
\qquad
\bs{y}\in\RR^3.
\end{equation*}
Then $Q_{d,\bs{x}}|_{\TT}\in\HH_d(\TT)$ and there exists a constant $C_{\TT,d,\delta_1}>0$ such that
\begin{equation}
\label{eq:ambient-taylor-remainder}
\abs{f(\bs{y})-Q_{d,\bs{x}}(\bs{y})}
\leq
C_{\TT,d,\delta_1}
\norm{f}_{C^{d+1}(\TT)}
\norm{\bs{y}-\bs{x}}_2^{d+1}
\end{equation}
for all $\bs{x},\bs{y}\in\TT$ satisfying $\norm{\bs{y}-\bs{x}}_2<\delta_1$.
\end{lemma}
\begin{proof}
The polynomial $Q_{d,\bs{x}}$ has total degree at most $d$ in the ambient variables; hence its restriction to $\TT$ belongs to $\HH_d(\TT)$. Let
\[
\bs{z}(t)=\bs{x}+t(\bs{y}-\bs{x}),
\qquad 0\leq t\leq1.
\]
Since $\bs{x}\in\TT$,
\[
\operatorname{dist}(\bs{z}(t),\TT)
\leq
\norm{\bs{z}(t)-\bs{x}}_2
=
t\norm{\bs{y}-\bs{x}}_2
<\delta_1
\]
for $0\leq t<1$, while $\bs{z}(1)=\bs{y}\in\TT$. Thus the whole segment $[\bs{x},\bs{y}]$ lies in $U_{\delta_1}(\TT)$. The ordinary multivariate Taylor theorem applied to $Ef$ along this segment, together with~\eqref{eq:extension-bound}, yields~\eqref{eq:ambient-taylor-remainder}.
\end{proof}

\begin{remark}
\label{rem:explicit-tubular-radius}
For the standard torus parametrized by~\eqref{eq:param}, a unit normal
field is given by
\[
\boldsymbol{\nu}(\theta,\varphi)
=
\bigl(
\cos\theta\cos\varphi,\,
\cos\theta\sin\varphi,\,
\sin\theta
\bigr).
\]
The corresponding normal map therefore takes the explicit form
\[
\begin{aligned}
\Psi(\theta,\varphi,t)
&=
\Phi(\theta,\varphi)
+t\boldsymbol{\nu}(\theta,\varphi)\\
&=
\bigl(
(R+(r+t)\cos\theta)\cos\varphi,\,
(R+(r+t)\cos\theta)\sin\varphi,\,
(r+t)\sin\theta
\bigr).
\end{aligned}
\]
Consequently, any
\[
0<\delta_0<\min\{r,R-r\}
\]
is an admissible tubular radius. Indeed, for every
$|t|<\delta_0$ one has
\[
0<r+t<R.
\]
Thus, in every meridional half-plane, the pair $(r+t,\theta)$ gives
unique polar coordinates with respect to the centre of the generating
circle, while
\[
R+(r+t)\cos\theta
\geq R-(r+t)>0
\]
ensures that the revolution angle $\varphi$ is uniquely determined.
Hence the normal map is regular and one-to-one on
$\mathbb T\times(-\delta_0,\delta_0)$.
 
The two bounds have a direct geometric interpretation. At $t=-r$, the
normal fibres collapse onto the core circle
\[
\bigl\{
(R\cos\varphi,R\sin\varphi,0):
\varphi\in[0,2\pi)
\bigr\},
\]
whereas at $t=R-r$ the normal fibres issuing from the inner equator
collapse at the origin. A more conservative choice, such as
\[
0<\delta_0<
\min\left\{r,\frac{R-r}{2}\right\},
\]
is therefore also admissible, although it is not sharp.
 
Finally, the preceding Taylor argument does not require the tubular
neighbourhood to be globally convex. The active-support condition ensures
that the segment joining the evaluation point to each relevant interpolation
node remains inside the tubular neighbourhood. More generally, Taylor
remainder estimates on non-convex domains may be obtained under the
Whitney-type path condition used by Farwig~\cite{Farwig}.
\end{remark}

\begin{theorem}
\label{thm:csms-error}
Fix $0<\delta_1<\delta_0$ and assume that $\rho_h<\delta_1$. Then, for every $\mu>0$ and every $f\in C^{d+1}(\TT)$,
\begin{equation}
\label{eq:csms-error-bound}
\norm{f-\widetilde{\mathcal M}_{\mu,h}[f]}_{L^\infty(\TT)}
\leq
C_{\TT,d,\delta_1}
\Lambda_h\rho_h^{d+1}
\norm{f}_{C^{d+1}(\TT)}.
\end{equation}
The constant is independent of $h$, $\mu$, and $f$.
\end{theorem}
\begin{proof}
At the interpolation nodes the error is zero by definition. Let therefore $\bs{x}\in\TT\setminus X_h$ and $j\in J_{\bs{x},h}$. Since $Q_{d,\bs{x}}|_{\TT}\in\HH_d(\TT)$, the local interpolation operator reproduces it exactly:
\[
P_{j,h}[Q_{d,\bs{x}}](\bs{x})
=Q_{d,\bs{x}}(\bs{x})
=f(\bs{x}).
\]
Consequently,
\[
\abs{f(\bs{x})-P_{j,h}[f](\bs{x})}
=
\abs{P_{j,h}[Q_{d,\bs{x}}-f](\bs{x})}
\leq
\sum_{k=1}^{m_d}
\abs{\ell_{j_k,h}(\bs{x})}
\abs{Q_{d,\bs{x}}(\bs{x}_{j_k,h})-f(\bs{x}_{j_k,h})}.
\]
Because $j\in J_{\bs{x},h}$,
\[
\norm{\bs{x}_{j_k,h}-\bs{x}}_2
<R_{j,h}
\leq\rho_h
<\delta_1,
\]
for every $k=1,\ldots,m_d$. Lemma~\ref{lem:ambient-taylor} therefore gives
\[
\abs{f(\bs{x})-P_{j,h}[f](\bs{x})}
\leq
C_{\TT,d,\delta_1}
\Lambda_h R_{j,h}^{d+1}
\norm{f}_{C^{d+1}(\TT)}.
\]
Using non-negativity of the compact weights and their partition-of-unity property, we obtain
\[
\begin{aligned}
\abs{f(\bs{x})-\widetilde{\mathcal M}_{\mu,h}[f](\bs{x})}
&\leq
\sum_{j\in J_{\bs{x},h}}
\widetilde W_{\mu,j,h}(\bs{x})
\abs{f(\bs{x})-P_{j,h}[f](\bs{x})}\\
&\leq
C_{\TT,d,\delta_1}
\Lambda_h\rho_h^{d+1}
\norm{f}_{C^{d+1}(\TT)}.
\end{aligned}
\]
Taking the supremum over $\bs{x}\in\TT$ proves~\eqref{eq:csms-error-bound}.
\end{proof}

\begin{corollary}
\label{cor:fill-distance-convergence}
Assume that $h=h_{X_h,\TT}\to0$ and that there exist constants $C_H,C_\sigma,\Lambda>0$, independent of $h$, such that
\begin{equation}
\label{eq:convergence-assumptions}
h<H_h\leq C_Hh,
\qquad
\max_{1\leq j\leq L_h}\operatorname{diam}_E(\sigma_{j,h})
\leq C_\sigma h,
\qquad
\Lambda_h\leq\Lambda.
\end{equation}
If the support radii are chosen as
\[
R_{j,h}=\operatorname{diam}_E(\sigma_{j,h})+H_h,
\]
then, for every fixed $d$ and every $f\in C^{d+1}(\TT)$,
\begin{equation}
\label{eq:fill-distance-order}
\norm{f-\widetilde{\mathcal M}_{\mu,h}[f]}_{L^\infty(\TT)}
\leq
C h^{d+1}\norm{f}_{C^{d+1}(\TT)},
\end{equation}
for all sufficiently small $h$, where $C$ is independent of $h$, $\mu$, and $f$.
\end{corollary}
\begin{proof}
By~\eqref{eq:convergence-assumptions},
\[
\rho_h
=\max_jR_{j,h}
\leq(C_\sigma+C_H)h.
\]
Hence $\rho_h<\delta_1$ for all sufficiently small $h$. The conclusion follows from Theorem~\ref{thm:csms-error} and the uniform bound $\Lambda_h\leq\Lambda$.
\end{proof}

\begin{remark}
The estimate~\eqref{eq:csms-error-bound} requires no lower threshold on the exponent $\mu$: its role in the proof is only through non-negativity and the partition-of-unity property of the normalized compact weights. The essential assumptions for the order in~\eqref{eq:fill-distance-order} are instead the uniform locality of the stencils and the uniform boundedness of $\Lambda_h$. In particular, nonsingularity of each local Vandermonde matrix alone does not imply the stability condition $\sup_h\Lambda_h<\infty$. The result is an $h$-convergence statement for fixed degree $d$; degree enrichment at fixed node set is a distinct issue.
\end{remark}

\section{Numerical experiments on analytical data}\label{sec:numerics}

This section presents the numerical experiments carried out to assess the performance of the proposed compactly supported multinode Shepard operator on the torus. The experiments are designed to validate both the theoretical properties of the method and its practical approximation capabilities on scattered data.

The numerical investigation is organized into three complementary
parts. The first two examine the approximation behaviour under node
refinement and degree enrichment, respectively. The third provides an
a posteriori analysis of the highest degrees considered, focusing on
the conditioning of the local systems, complete polynomial
reproduction, sampled Lebesgue factors, and sensitivity to data
perturbations.

Throughout these experiments, we consider the following collection of
15 analytical test functions defined on the torus, including both
polynomial and smooth non-polynomial examples \cite{DellAccioSphere}.

\begin{align*}
f_1(x,y,z)&=\frac{1+2x+3y+4z}{6},\\
f_2(x,y,z)&=\frac{-1+2x-3y+4x^2-xy+9y^2+3z^2-yz}{10},\\
f_3(x,y,z)&=\frac{9x^3-2x^2y+3xy^2-4y^3+2z^3-xyz}{10},\\
f_4(x,y,z)&=\frac12+(x+y)^4+z^4,\\
f_5(x,y,z)&=\sin(x)\sin(y)\sin(z),\\
f_6(x,y,z)&=\frac{1+\tanh(-9x-9y+9z)}{9},\\
f_7(x,y,z)&=1+x^8+e^{2y^3}+e^{2z^2}+10xyz,\\
f_8(x,y,z)&=\sin(x+y)+\sin(xz),\\
f_9(x,y,z)&=\frac34\exp\!\left(-\frac{(9x-2)^2+(9y-2)^2+(9z-2)^2}{4}\right)\\
&\quad+\frac34\exp\!\left(-\frac{(9x+1)^2}{49}-\frac{9y+1}{10}-\frac{9z+1}{10}\right)\\
&\quad+\frac12\exp\!\left(-\frac{(9x-7)^2+(9y-3)^2+(9z-5)^2}{4}\right)\\
&\quad-\frac15\exp\!\left(-(9x-4)^2-(9y-7)^2-(9z-5)^2\right),\\
f_{10}(x,y,z)&=x^2+y^2+z^2+\frac14\sin(8x)+\frac14\cos(8y)+\frac14\sin(16z),\\
f_{11}(x,y,z)&=-5\sin(1+10z),\\
f_{12}(x,y,z)&=\frac{e^x+2e^{y+z}}{10},\\
f_{13}(\theta,\phi)&=1+0.2\cos(3\theta-2\phi),\\
f_{14}(x,y,z)&=\exp(0.5x+0.25y-0.1z),\\
f_{15}(x,y,z)&=\sin(3x-2y+z).
\end{align*}
These tests are used to verify the polynomial reproduction property of the proposed operator and to study its approximation behaviour as the polynomial degree and the number of interpolation nodes increase. 
To provide a qualitative overview of the analytical benchmark, the test
functions are visualized on the torus before presenting the interpolation
results. Each function is displayed using two complementary representations:
the toroidal surface is radially deformed according to the function values,
highlighting the geometric variation, while a color map on the original torus
provides a direct visualization of the corresponding scalar field; the two representations are reported in Figures~\ref{fig:test_functions_deformation} and~\ref{fig:test_functions_colormap}, respectively. These two
representations facilitate the interpretation of the interpolation results for
functions exhibiting different levels of smoothness, oscillation, and
localization.

For all analytical tests, the interpolation nodes are generated from a
two-dimensional Halton sequence with bases $2$ and $3$. More precisely,
for $k=1,\ldots,N$, we set
\[
(u_k,v_k)
=
\bigl(\phi_2(k),\phi_3(k)\bigr)\in[0,1)^2,
\]
where $\phi_b$ denotes the radical-inverse function in base $b$, and map
these parameter points onto the torus according to
\[
\theta_k=2\pi u_k,
\qquad
\varphi_k=2\pi v_k,
\qquad
\boldsymbol{x}_k=\Phi(\theta_k,\varphi_k).
\]
The radical-inverse construction of the Halton sequence follows
\cite{Tien1997}. In the implementation, the term corresponding to $k=0$
is discarded, so that exactly $N$ interpolation nodes are retained.
The resulting nested node sets are low-discrepancy with respect to the
flat measure in the periodic parameter domain.
 
Notice that the surface element induced by the toroidal parametrization is
\[
dA
=
\left\|
\frac{\partial\Phi}{\partial\theta}
\times
\frac{\partial\Phi}{\partial\varphi}
\right\|
\,d\theta\,d\varphi
=
r(R+r\cos\theta)\,d\theta\,d\varphi.
\]
Since this Jacobian is not constant, the mapped Halton nodes are not, in
general, exactly equidistributed with respect to the surface-area measure.
Accordingly, throughout this section, the distribution of the Halton nodes
is understood in the parameter-space sense.
The accuracy of the proposed interpolation operator is assessed on an
independent validation set obtained by mapping a $100\times100$
tensor-product grid, uniform in the periodic parameter domain, onto the
toroidal surface.

At each validation point
$\bs{\xi}_i$, the pointwise interpolation error is computed as
\[
e_i=
\left|
f(\bs{\xi}_i)
-\widetilde{\mathcal M}_{\mu}[f](\bs{\xi}_i)
\right|.
\]
The interpolation accuracy is then quantified by means of the following three
error indicators:
\begin{equation*}
E_{\max}=\max_i e_i,
\qquad
E_{\operatorname{mean}}=\frac{1}{n_e}\sum_{i=1}^{n_e}e_i,
\qquad
E_{\operatorname{RMS}}=\sqrt{\frac{1}{n_e}\sum_{i=1}^{n_e}e_i^2},
\end{equation*}
corresponding to the maximum, mean, and root mean square interpolation errors,
respectively. In the experiments we set $R=1$, $r=1/2$ and $\mu=4$. 

\begin{figure}[htbp]
\centering

\begin{minipage}[t]{0.19\textwidth}
    \centering
    \includegraphics[width=\linewidth]{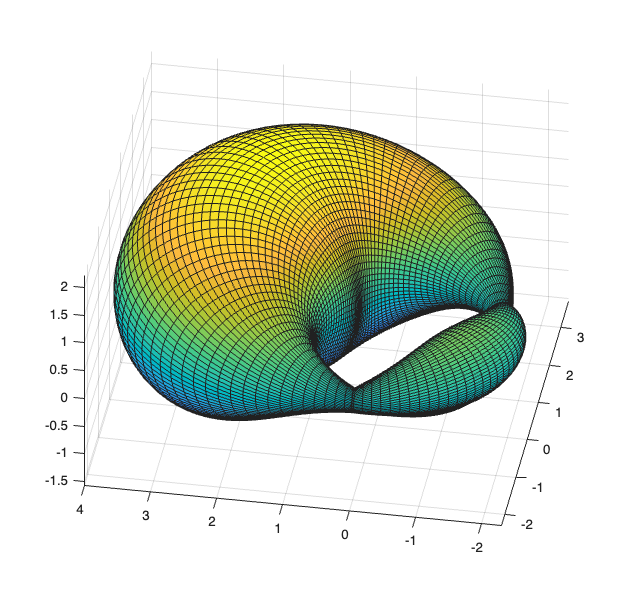}
    \par\smallskip
    {\scriptsize $f_1$}
\end{minipage}\hfill
\begin{minipage}[t]{0.19\textwidth}
    \centering
    \includegraphics[width=\linewidth]{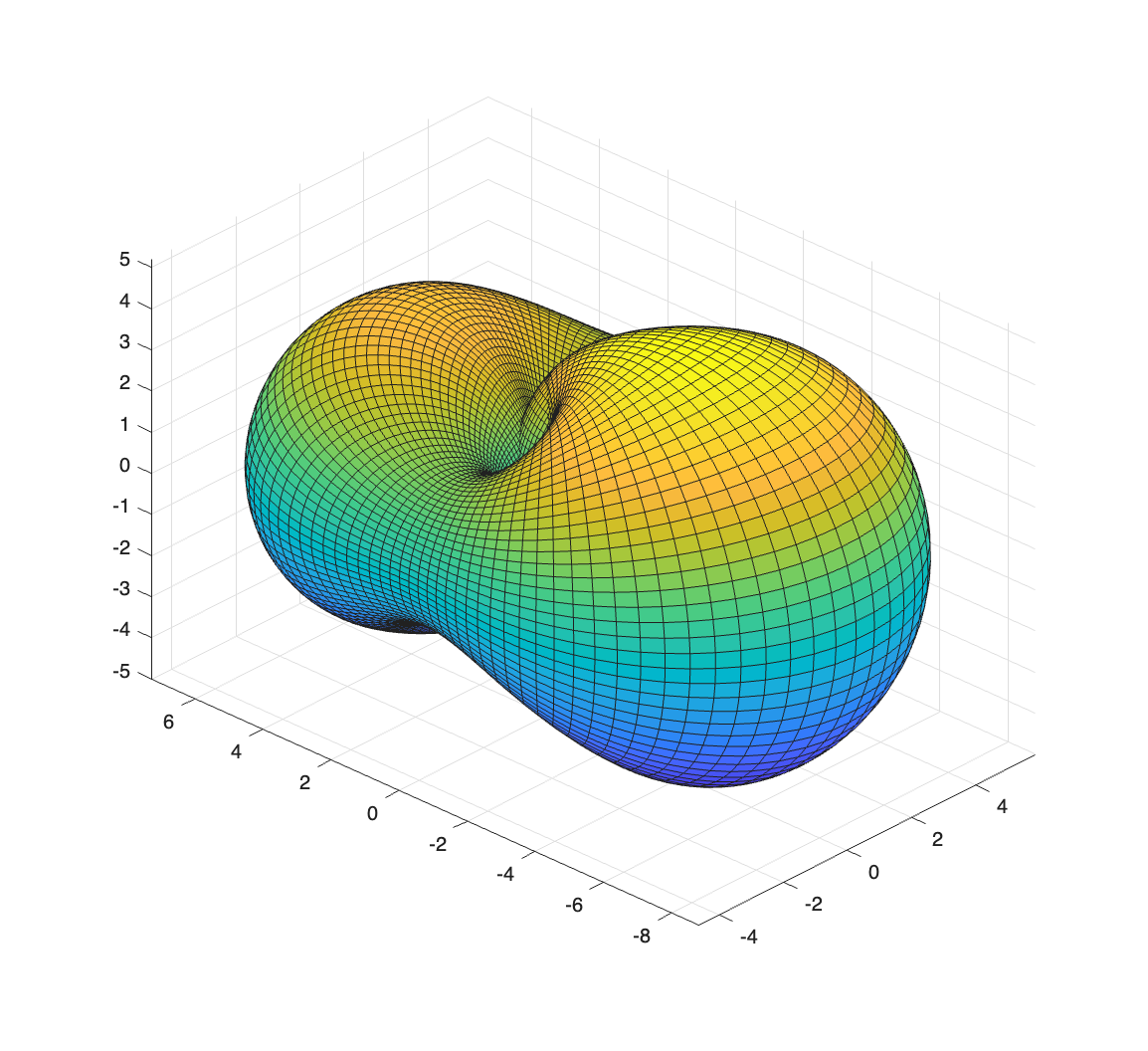}
    \par\smallskip
    {\scriptsize $f_2$}
\end{minipage}\hfill
\begin{minipage}[t]{0.19\textwidth}
    \centering
    \includegraphics[width=\linewidth]{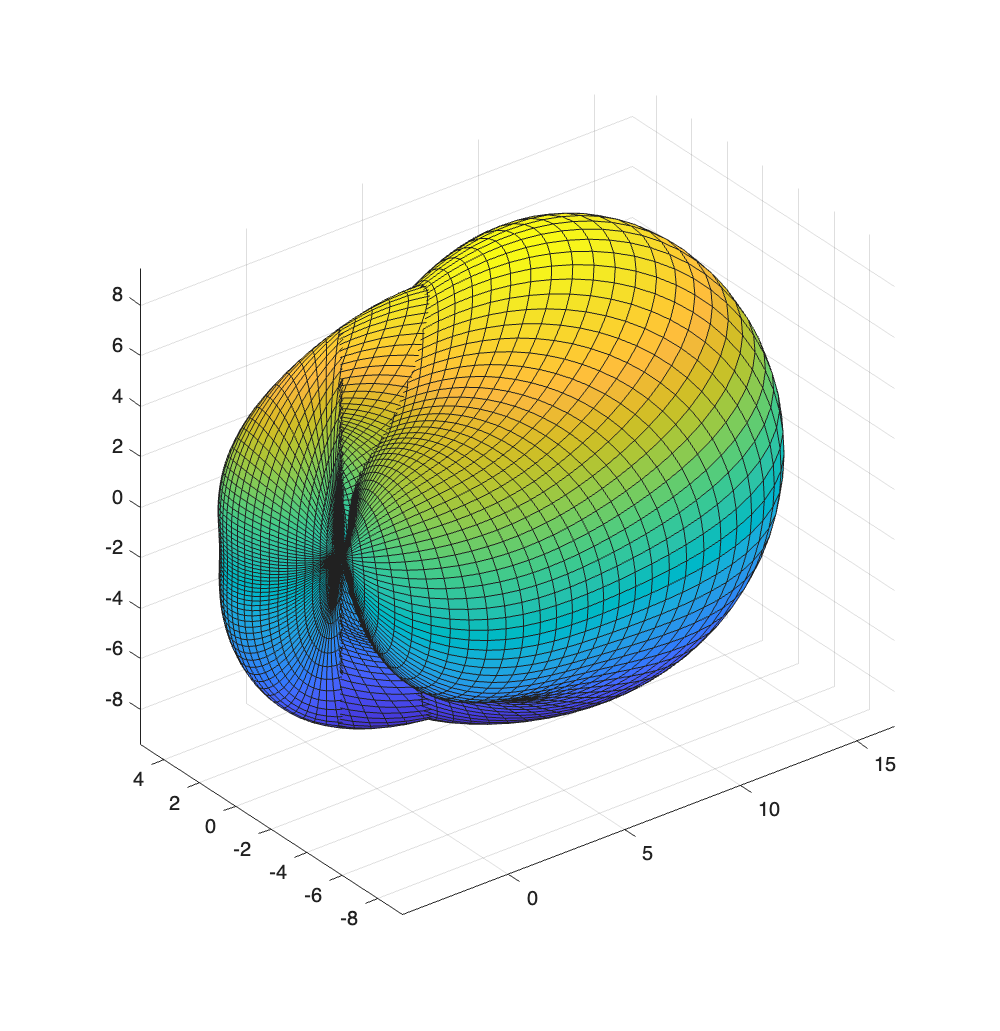}
    \par\smallskip
    {\scriptsize $f_3$}
\end{minipage}\hfill
\begin{minipage}[t]{0.19\textwidth}
    \centering
    \includegraphics[width=\linewidth]{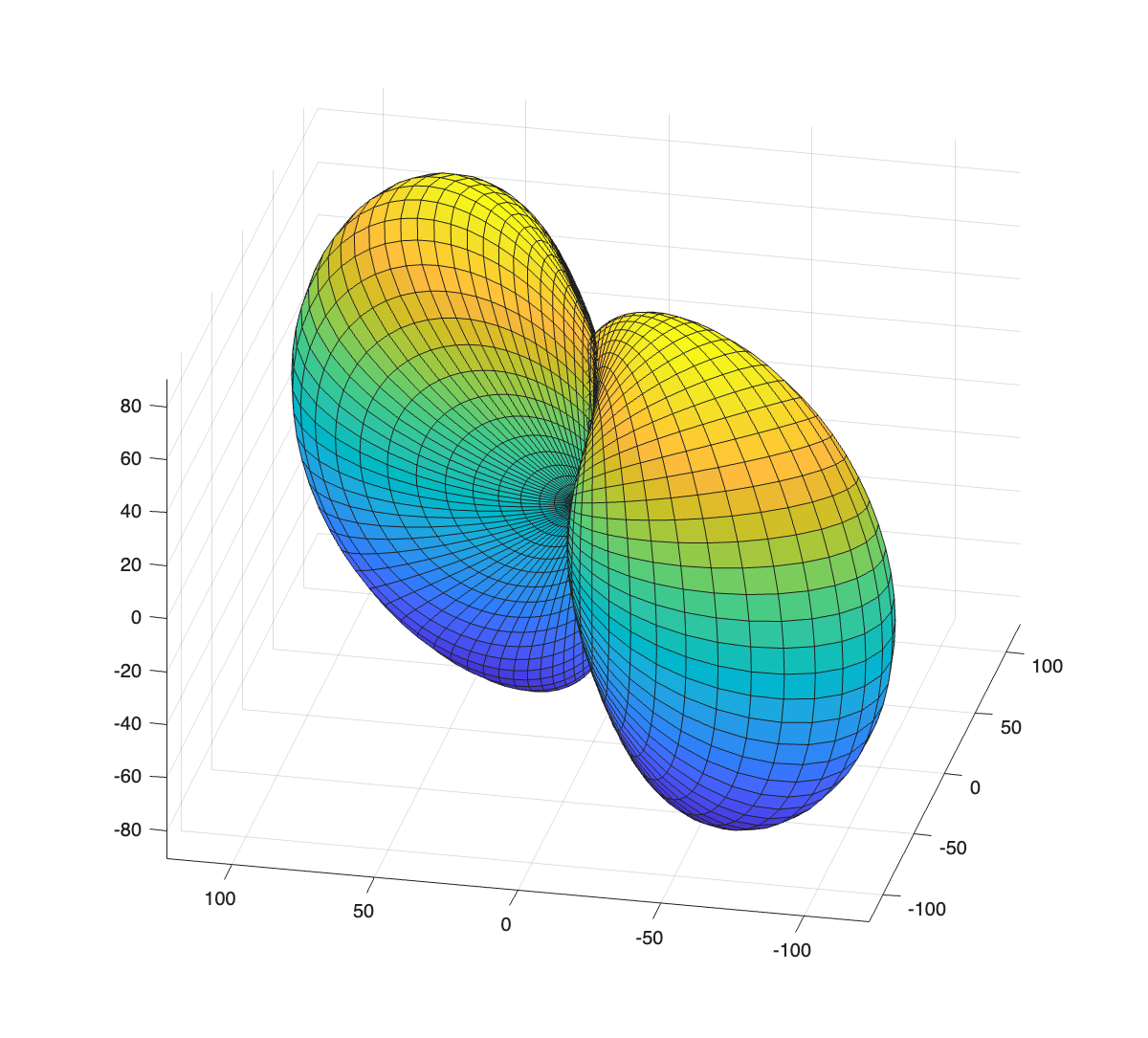}
    \par\smallskip
    {\scriptsize $f_4$}
\end{minipage}\hfill
\begin{minipage}[t]{0.19\textwidth}
    \centering
    \includegraphics[width=\linewidth]{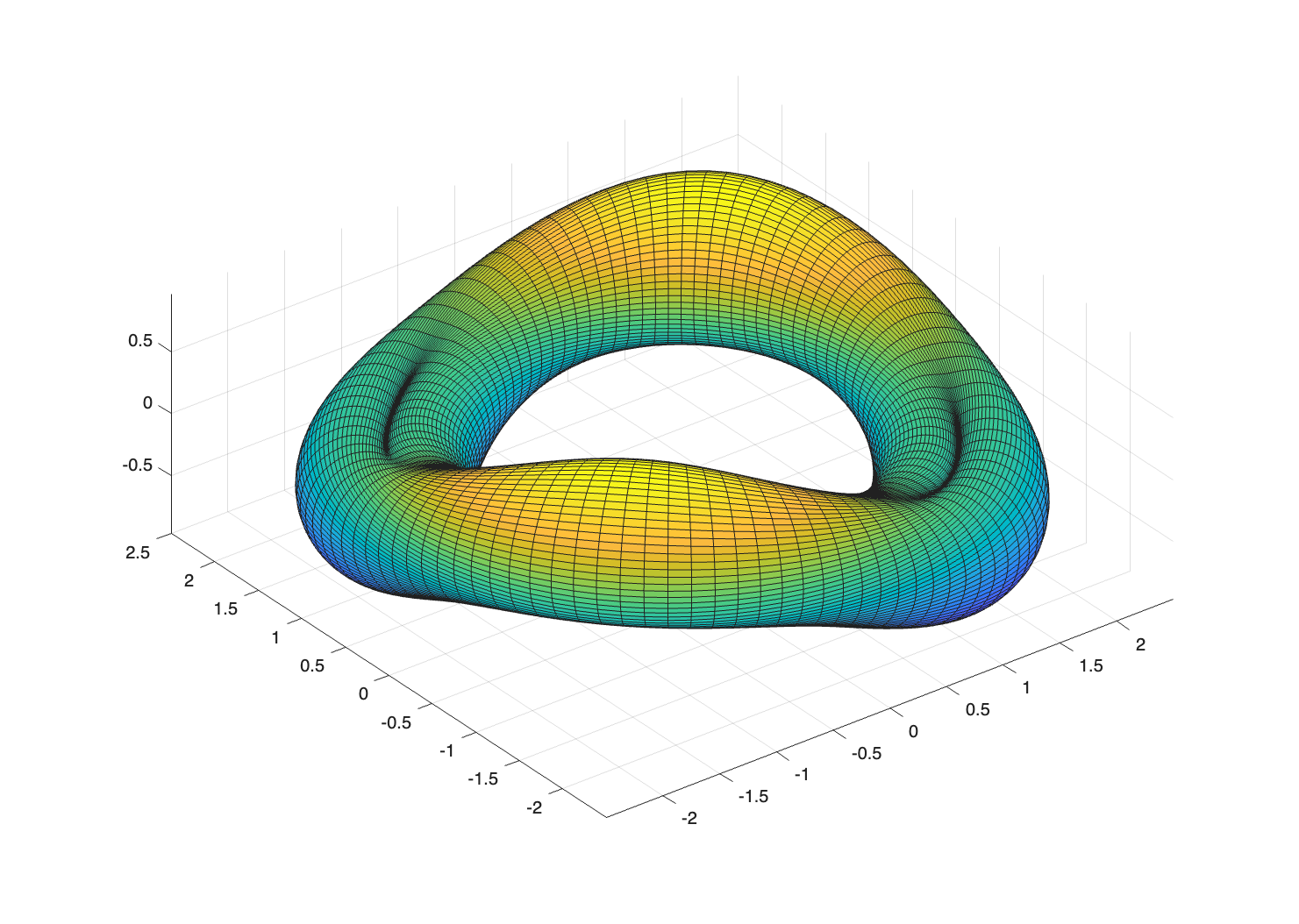}
    \par\smallskip
    {\scriptsize $f_5$}
\end{minipage}

\medskip

\begin{minipage}[t]{0.19\textwidth}
    \centering
    \includegraphics[width=\linewidth]{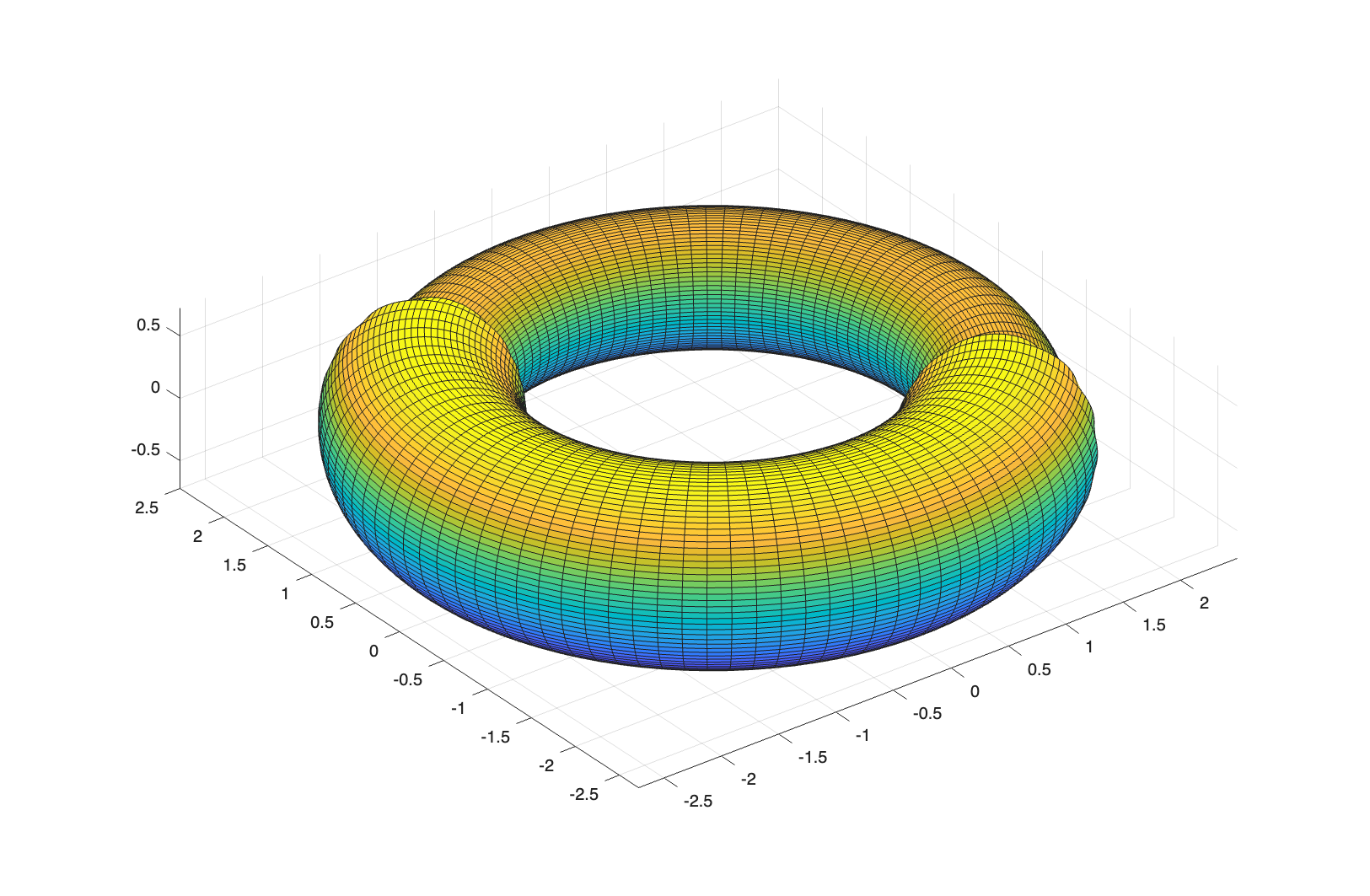}
    \par\smallskip
    {\scriptsize $f_6$}
\end{minipage}\hfill
\begin{minipage}[t]{0.19\textwidth}
    \centering
    \includegraphics[width=\linewidth]{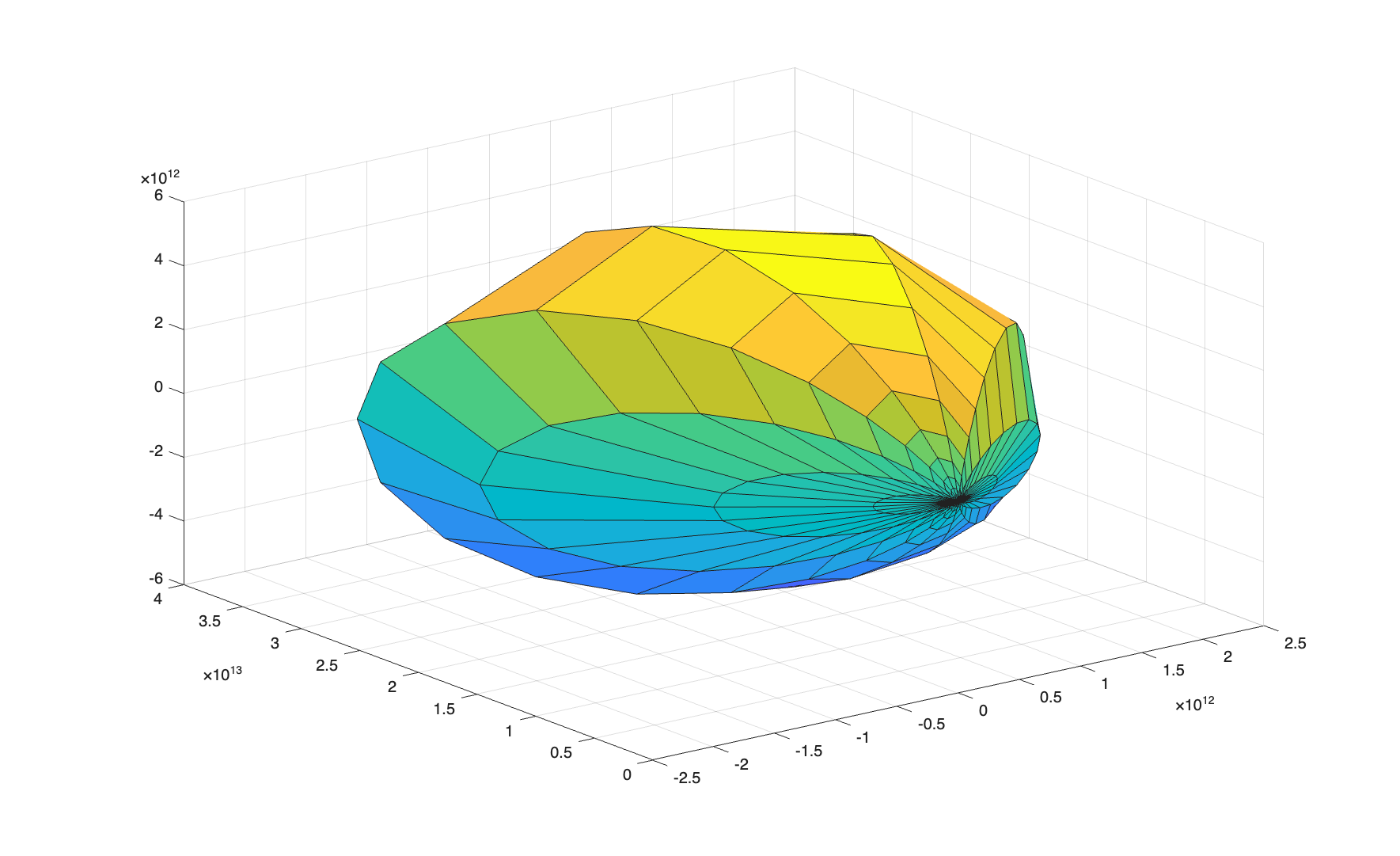}
    \par\smallskip
    {\scriptsize $f_7$}
\end{minipage}\hfill
\begin{minipage}[t]{0.19\textwidth}
    \centering
    \includegraphics[width=\linewidth]{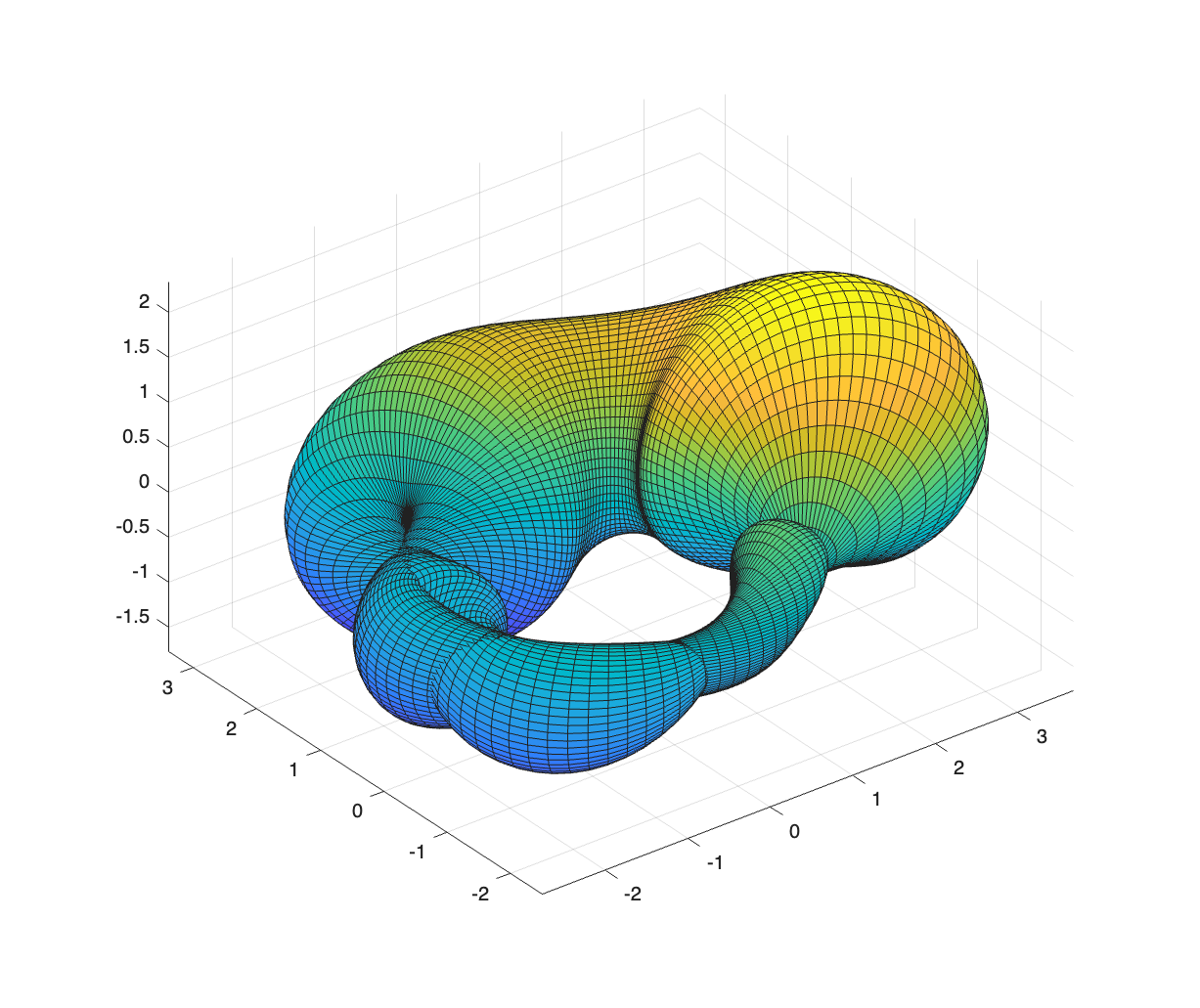}
    \par\smallskip
    {\scriptsize $f_8$}
\end{minipage}\hfill
\begin{minipage}[t]{0.19\textwidth}
    \centering
    \includegraphics[width=\linewidth]{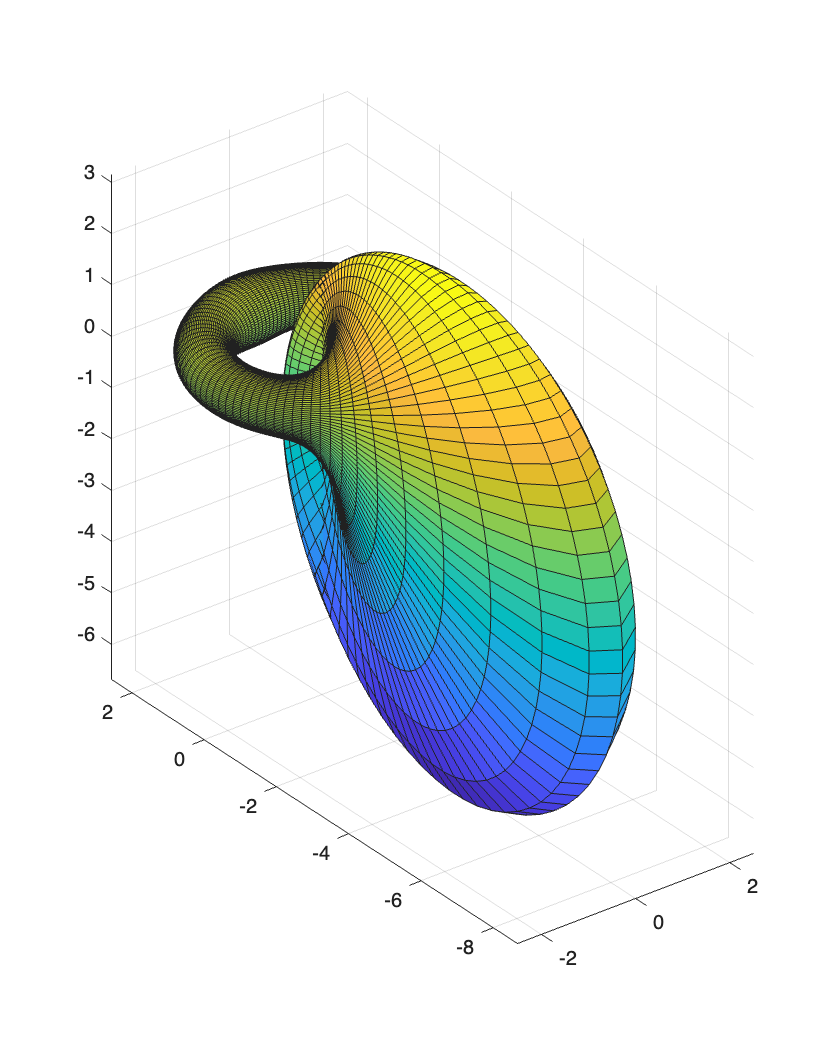}
    \par\smallskip
    {\scriptsize $f_9$}
\end{minipage}\hfill
\begin{minipage}[t]{0.19\textwidth}
    \centering
    \includegraphics[width=\linewidth]{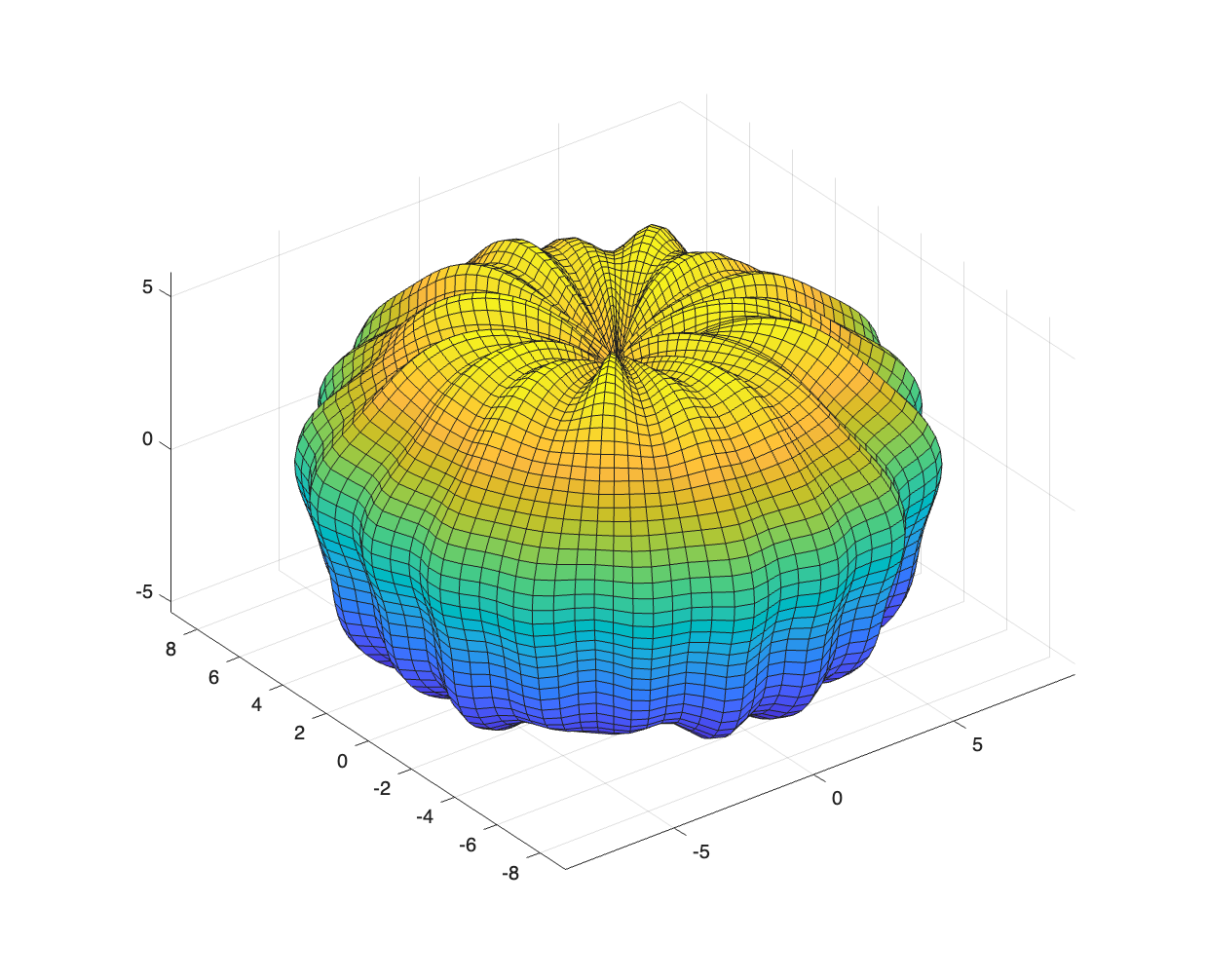}
    \par\smallskip
    {\scriptsize $f_{10}$}
\end{minipage}

\medskip

\begin{minipage}[t]{0.19\textwidth}
    \centering
    \includegraphics[width=\linewidth]{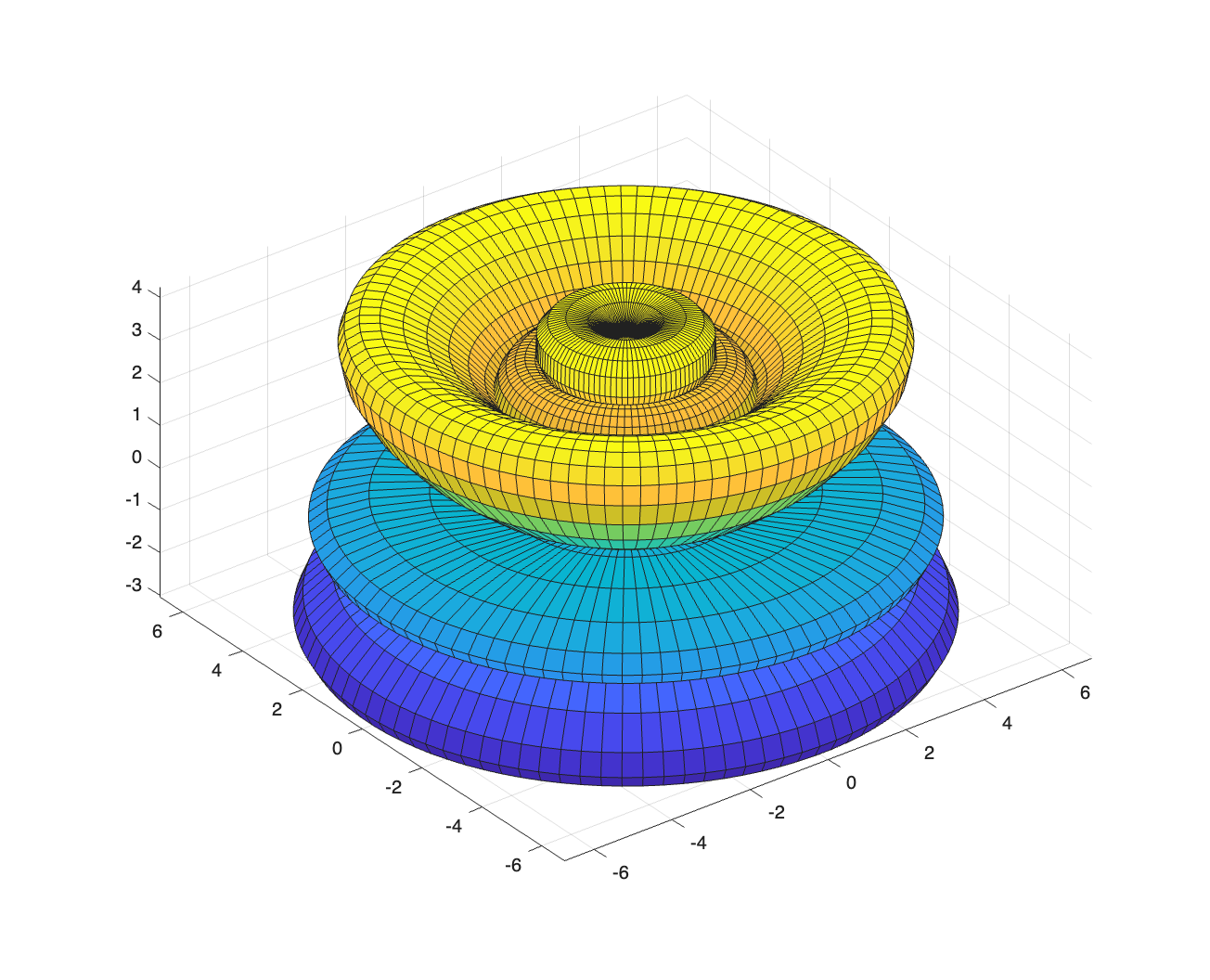}
    \par\smallskip
    {\scriptsize $f_{11}$}
\end{minipage}\hfill
\begin{minipage}[t]{0.19\textwidth}
    \centering
    \includegraphics[width=\linewidth]{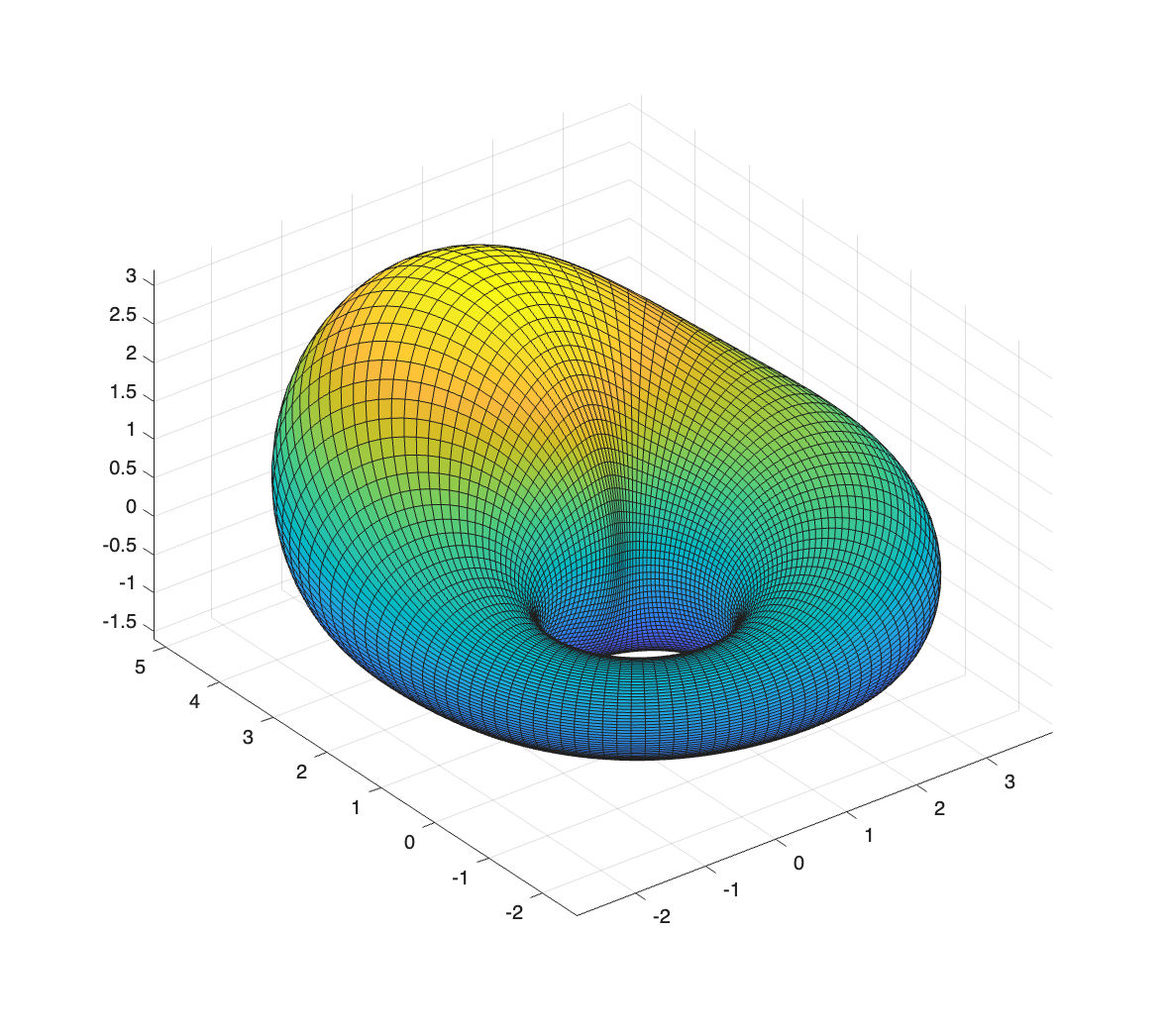}
    \par\smallskip
    {\scriptsize $f_{12}$}
\end{minipage}\hfill
\begin{minipage}[t]{0.19\textwidth}
    \centering
    \includegraphics[width=\linewidth]{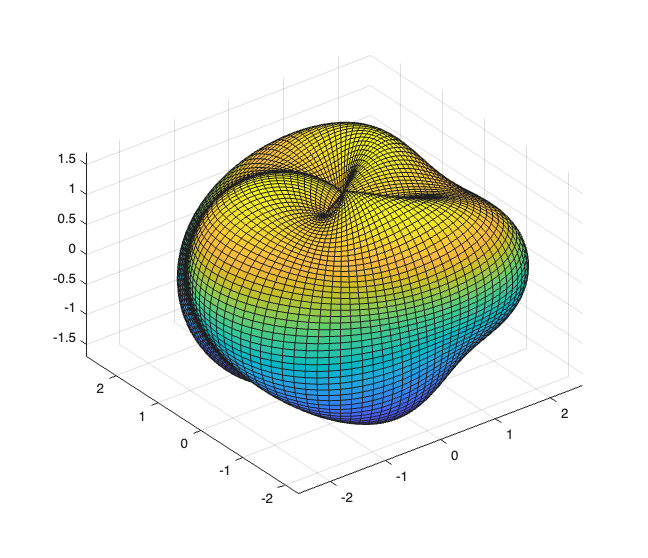}
    \par\smallskip
    {\scriptsize $f_{13}$}
\end{minipage}\hfill
\begin{minipage}[t]{0.19\textwidth}
    \centering
    \includegraphics[width=\linewidth]{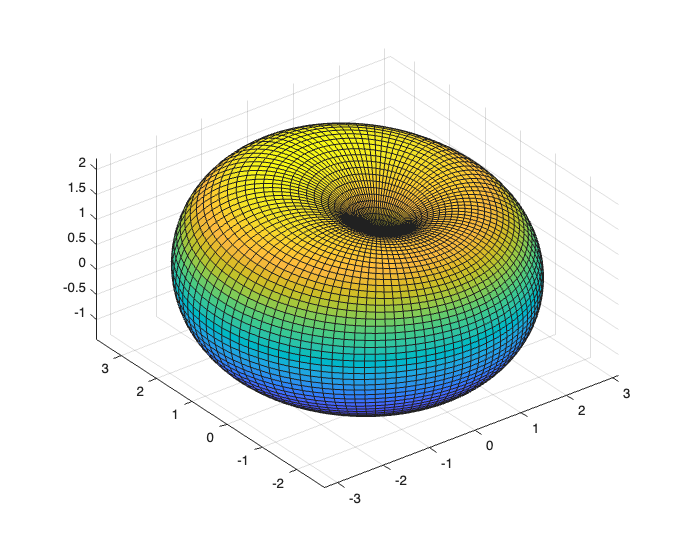}
    \par\smallskip
    {\scriptsize $f_{14}$}
\end{minipage}\hfill
\begin{minipage}[t]{0.19\textwidth}
    \centering
    \includegraphics[width=\linewidth]{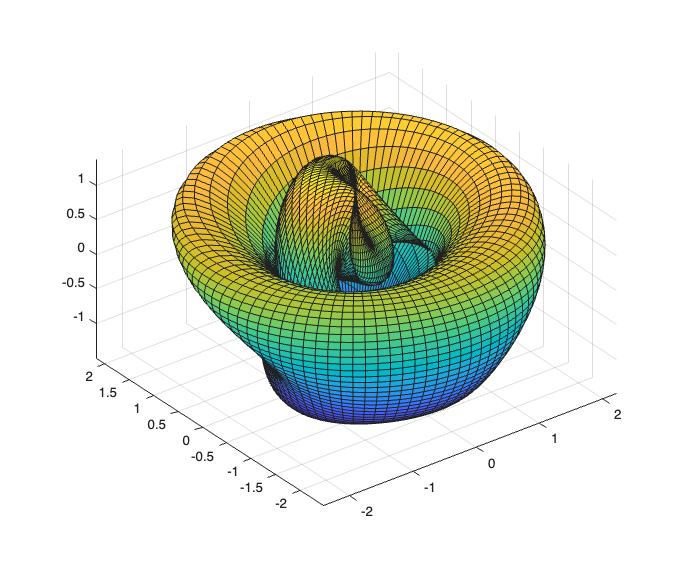}
    \par\smallskip
    {\scriptsize $f_{15}$}
\end{minipage}

\caption{Geometric visualization of the analytical test functions
$f_1,\ldots,f_{15}$ by deformation of the toroidal surface. The displacement
of the surface is proportional to the corresponding function value, thereby
highlighting the global shape, localized features, and oscillatory behaviour
of the test functions.}
\label{fig:test_functions_deformation}
\end{figure}

\begin{figure}[htbp]
\centering

\begin{minipage}[t]{0.19\textwidth}
    \centering
    \includegraphics[width=\linewidth]{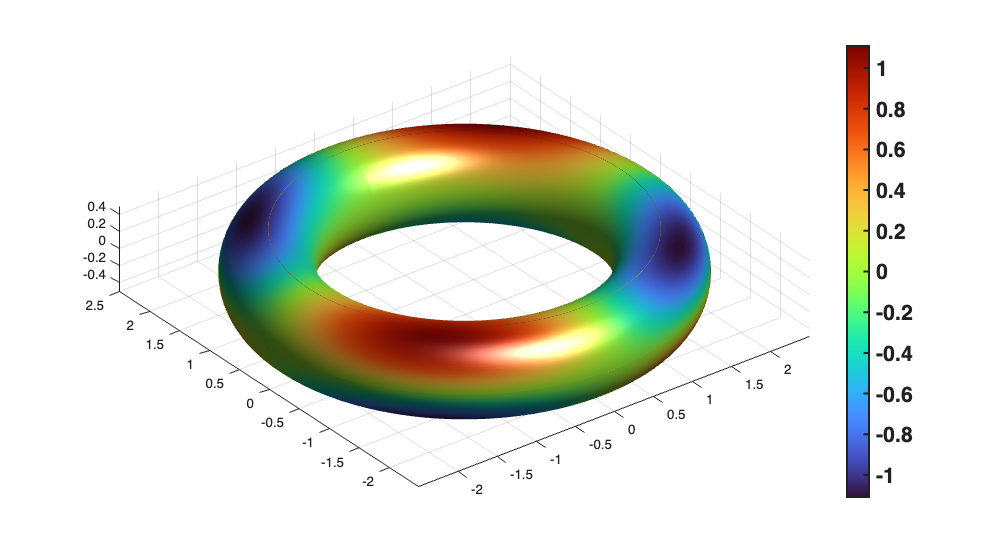}
    \par\smallskip
    {\scriptsize $f_1$}
\end{minipage}\hfill
\begin{minipage}[t]{0.19\textwidth}
    \centering
    \includegraphics[width=\linewidth]{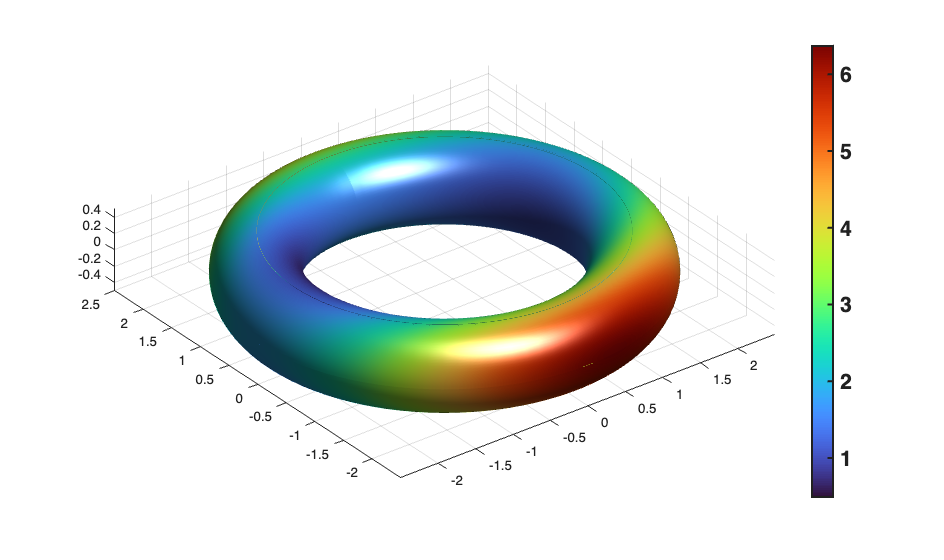}
    \par\smallskip
    {\scriptsize $f_2$}
\end{minipage}\hfill
\begin{minipage}[t]{0.19\textwidth}
    \centering
    \includegraphics[width=\linewidth]{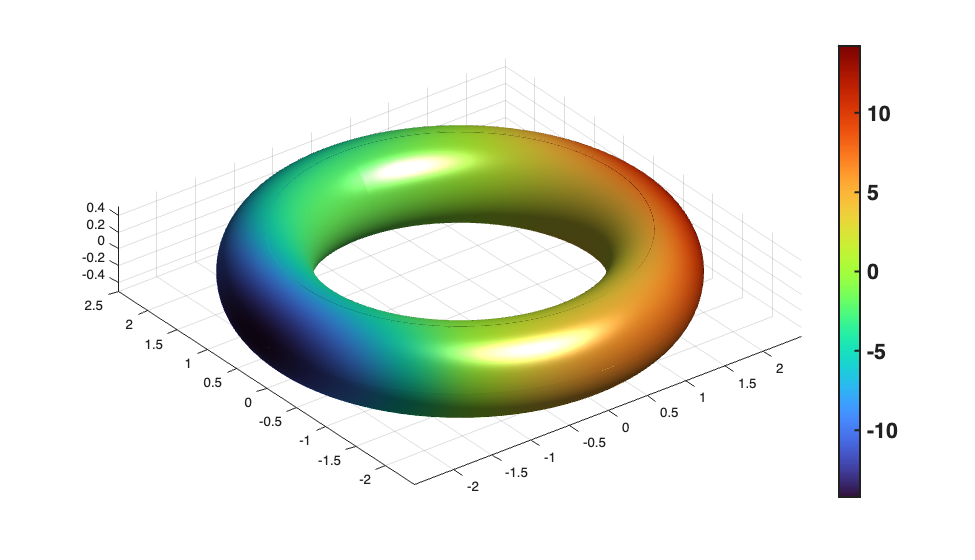}
    \par\smallskip
    {\scriptsize $f_3$}
\end{minipage}\hfill
\begin{minipage}[t]{0.19\textwidth}
    \centering
    \includegraphics[width=\linewidth]{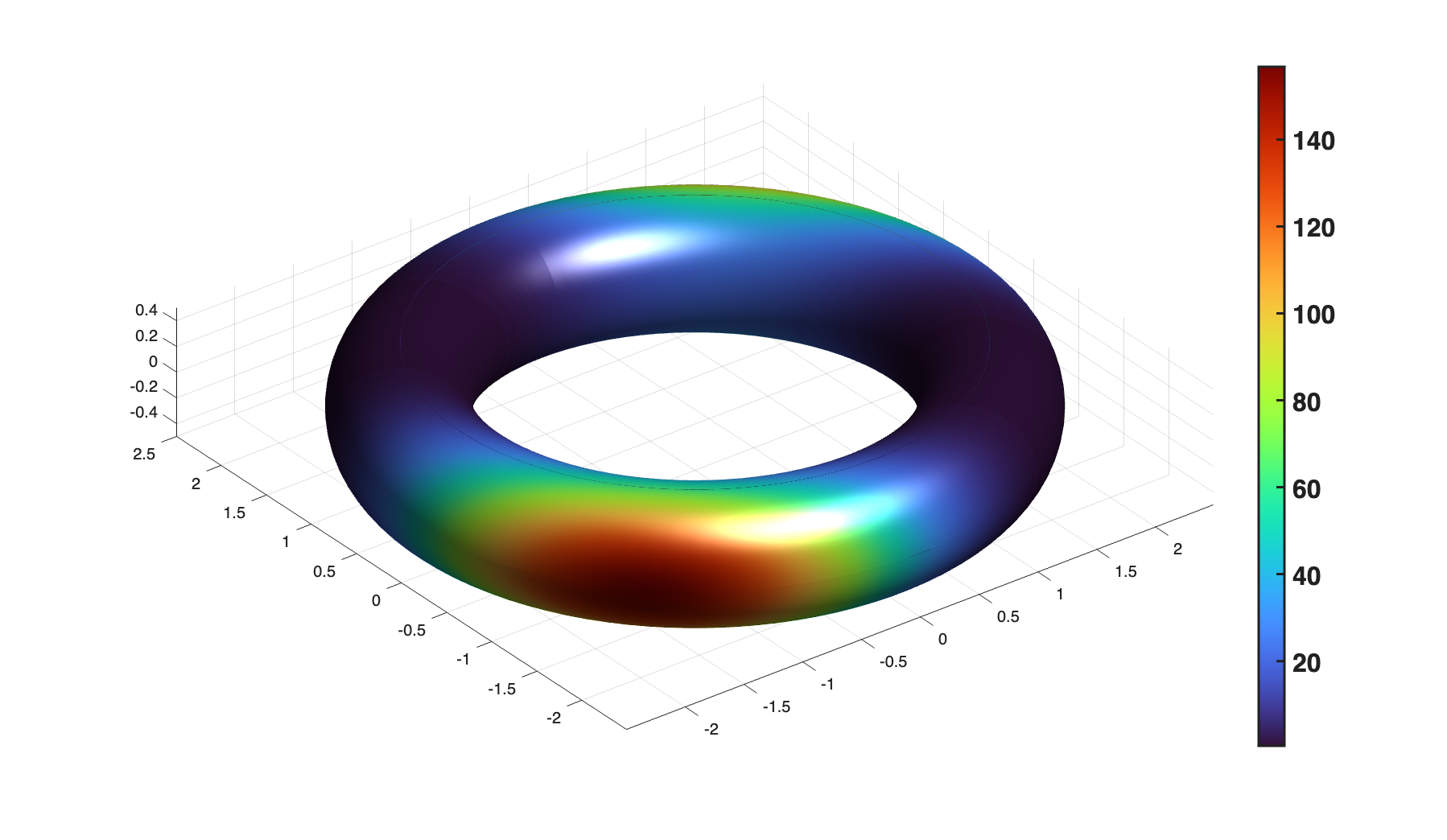}
    \par\smallskip
    {\scriptsize $f_4$}
\end{minipage}\hfill
\begin{minipage}[t]{0.19\textwidth}
    \centering
    \includegraphics[width=\linewidth]{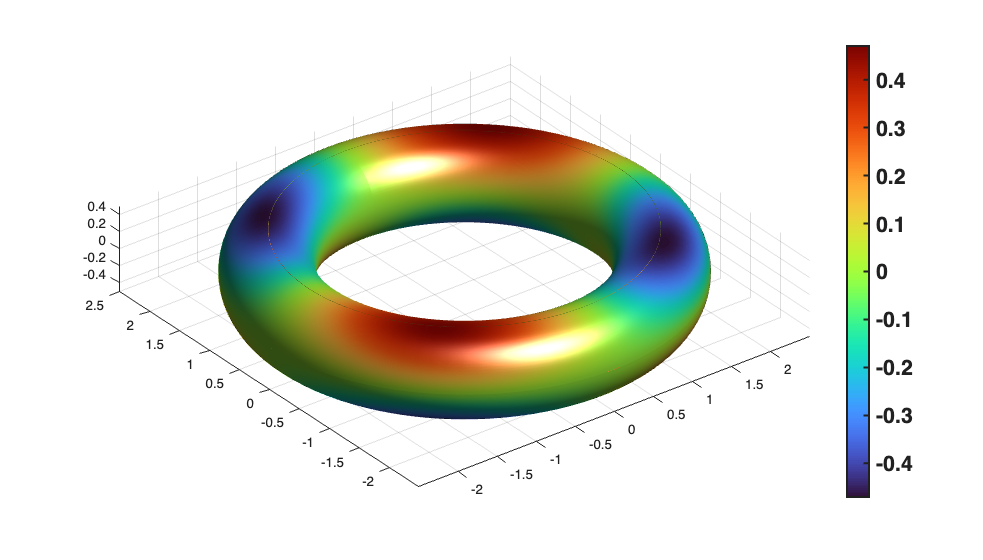}
    \par\smallskip
    {\scriptsize $f_5$}
\end{minipage}

\medskip

\begin{minipage}[t]{0.19\textwidth}
    \centering
    \includegraphics[width=\linewidth]{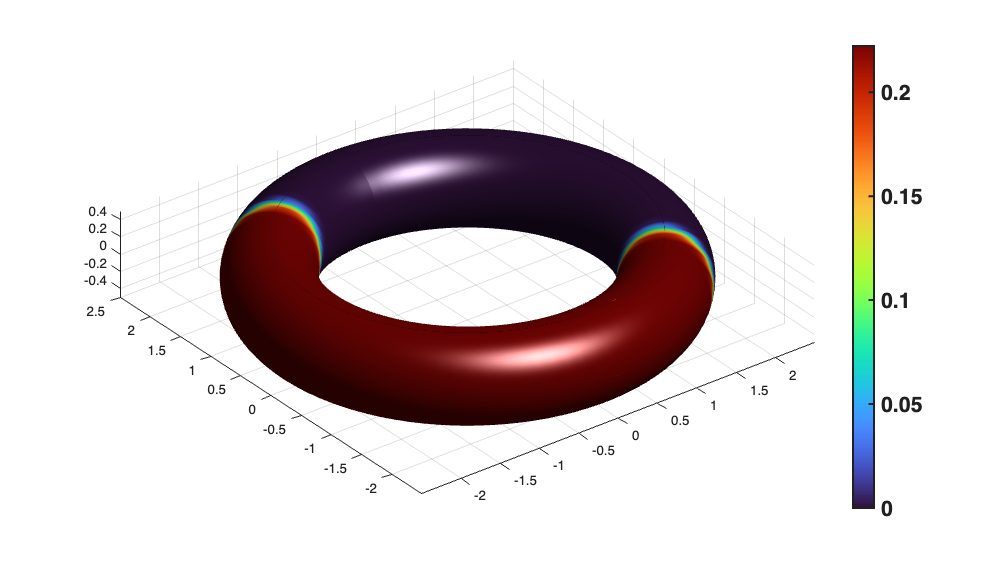}
    \par\smallskip
    {\scriptsize $f_6$}
\end{minipage}\hfill
\begin{minipage}[t]{0.19\textwidth}
    \centering
    \includegraphics[width=\linewidth]{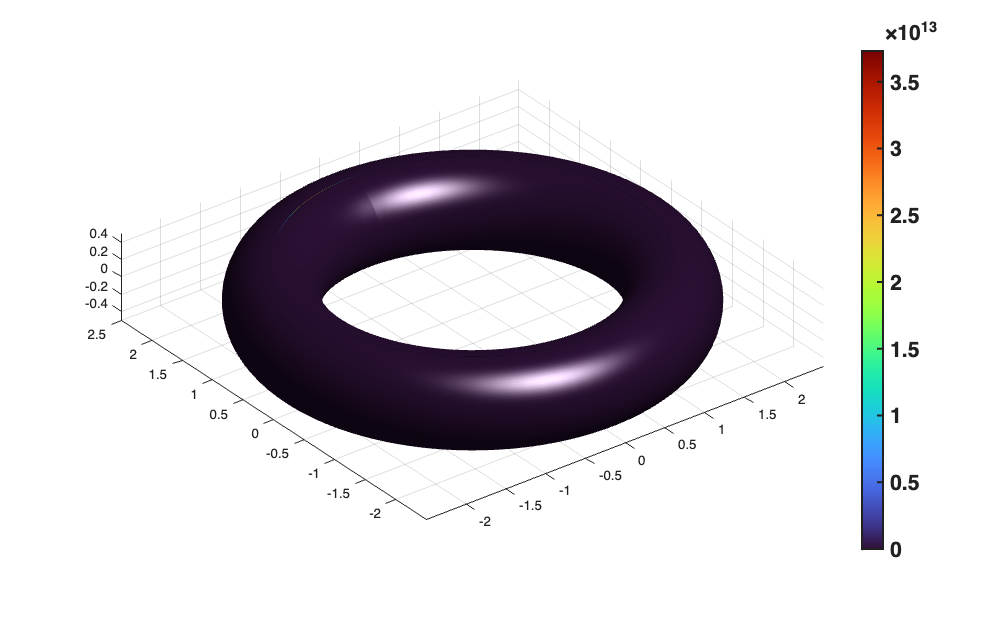}
    \par\smallskip
    {\scriptsize $f_7$}
\end{minipage}\hfill
\begin{minipage}[t]{0.19\textwidth}
    \centering
    \includegraphics[width=\linewidth]{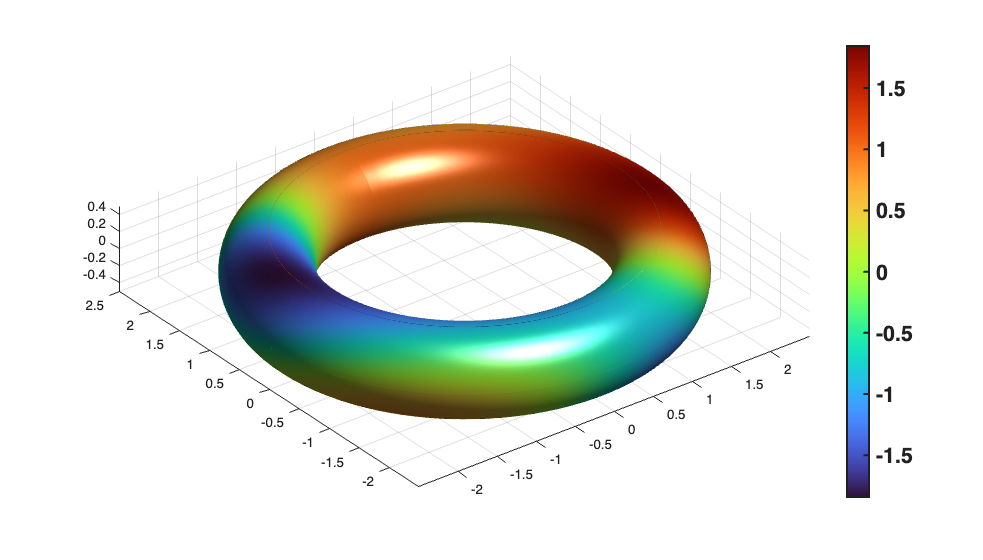}
    \par\smallskip
    {\scriptsize $f_8$}
\end{minipage}\hfill
\begin{minipage}[t]{0.19\textwidth}
    \centering
    \includegraphics[width=\linewidth]{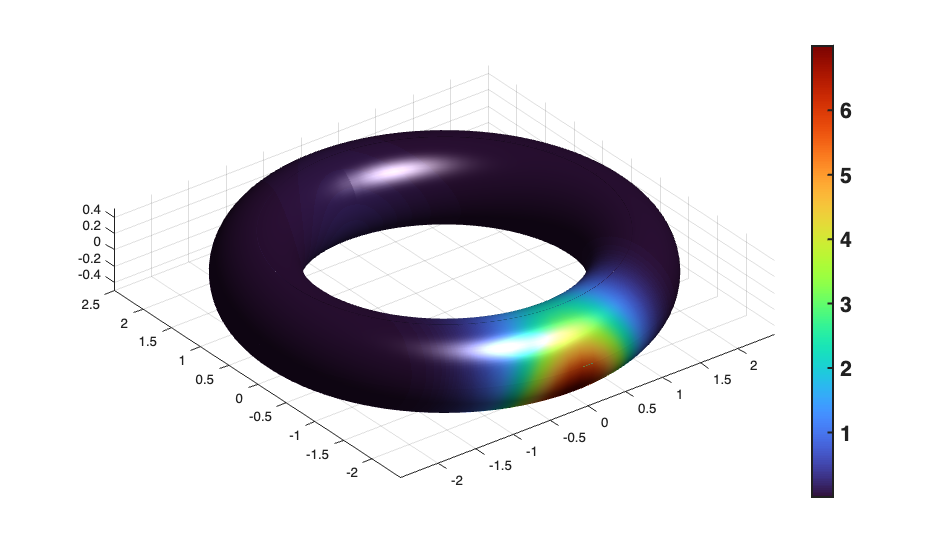}
    \par\smallskip
    {\scriptsize $f_9$}
\end{minipage}\hfill
\begin{minipage}[t]{0.19\textwidth}
    \centering
    \includegraphics[width=\linewidth]{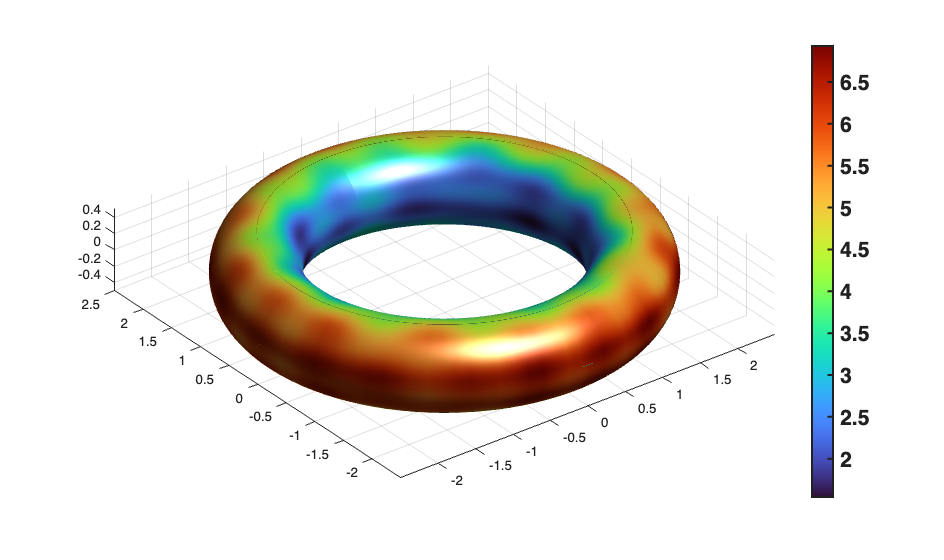}
    \par\smallskip
    {\scriptsize $f_{10}$}
\end{minipage}

\medskip

\begin{minipage}[t]{0.19\textwidth}
    \centering
    \includegraphics[width=\linewidth]{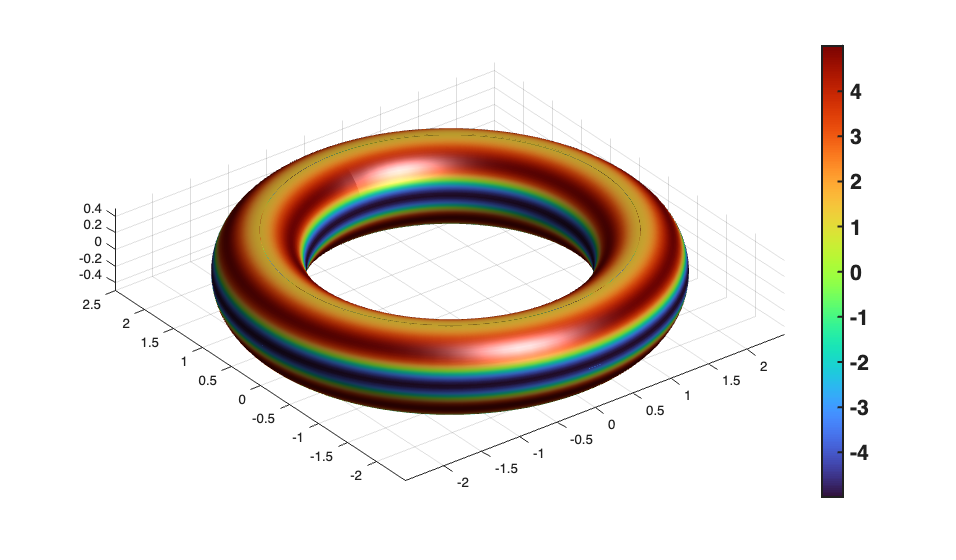}
    \par\smallskip
    {\scriptsize $f_{11}$}
\end{minipage}\hfill
\begin{minipage}[t]{0.19\textwidth}
    \centering
    \includegraphics[width=\linewidth]{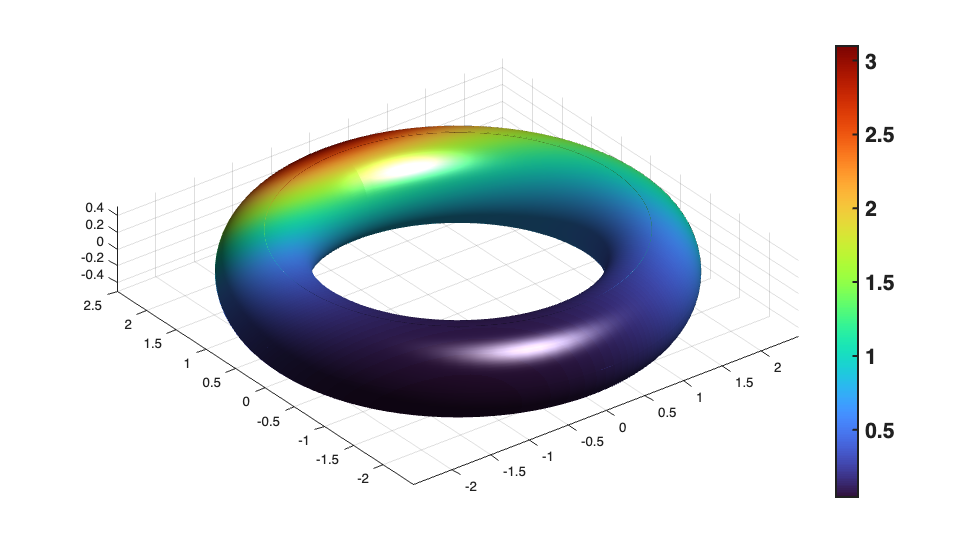}
    \par\smallskip
    {\scriptsize $f_{12}$}
\end{minipage}\hfill
\begin{minipage}[t]{0.19\textwidth}
    \centering
    \includegraphics[width=\linewidth]{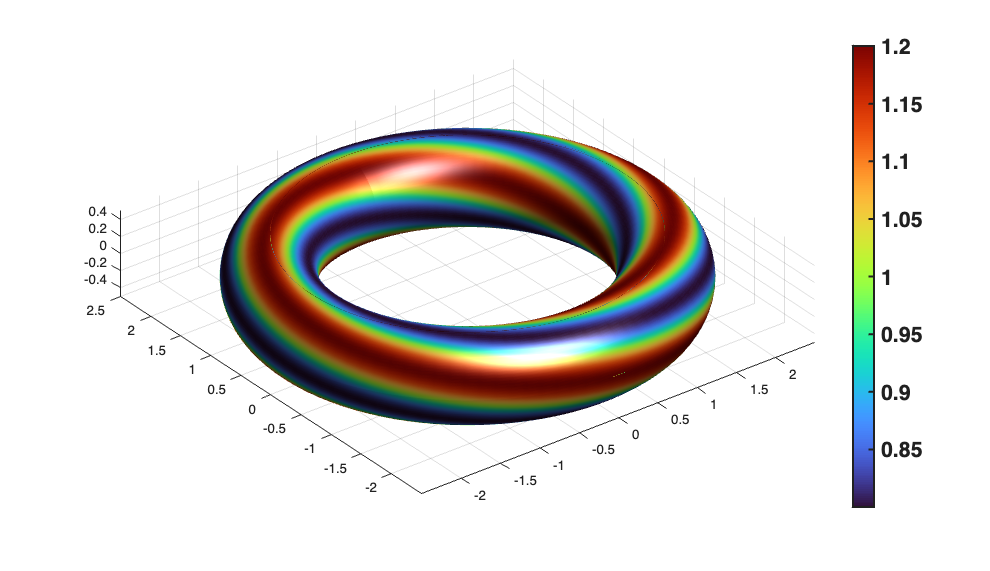}
    \par\smallskip
    {\scriptsize $f_{13}$}
\end{minipage}\hfill
\begin{minipage}[t]{0.19\textwidth}
    \centering
    \includegraphics[width=\linewidth]{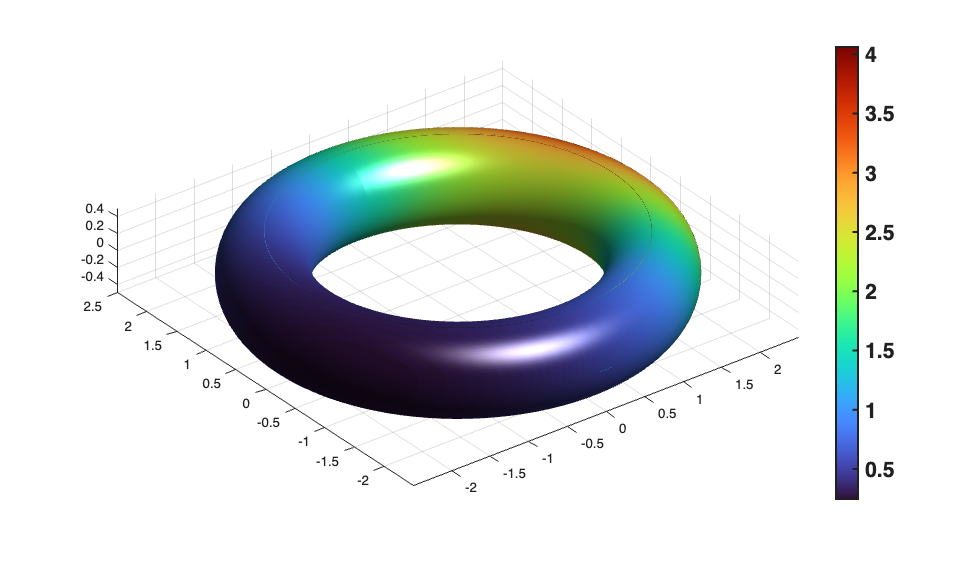}
    \par\smallskip
    {\scriptsize $f_{14}$}
\end{minipage}\hfill
\begin{minipage}[t]{0.19\textwidth}
    \centering
    \includegraphics[width=\linewidth]{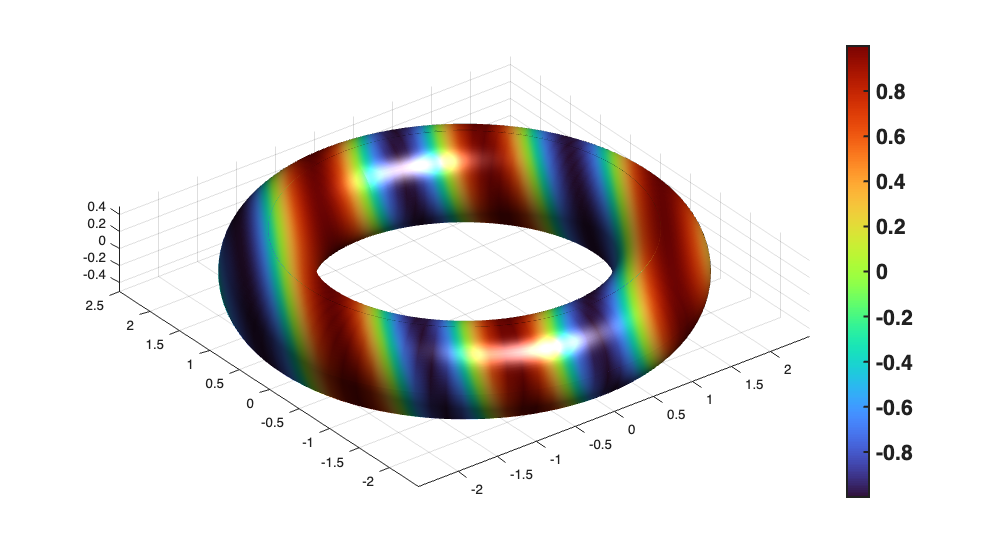}
    \par\smallskip
    {\scriptsize $f_{15}$}
\end{minipage}

\caption{Color-map visualization of the analytical test functions
$f_1,\ldots,f_{15}$ on the toroidal surface. The color scale encodes the
function values while preserving the original geometry of the torus, thereby
providing a complementary representation to the surface deformations shown in
Figure~\ref{fig:test_functions_deformation}.}
\label{fig:test_functions_colormap}
\end{figure}

\begin{figure}[htbp]
\centering
\begin{minipage}{0.48\textwidth}
\centering
\includegraphics[width=\textwidth]{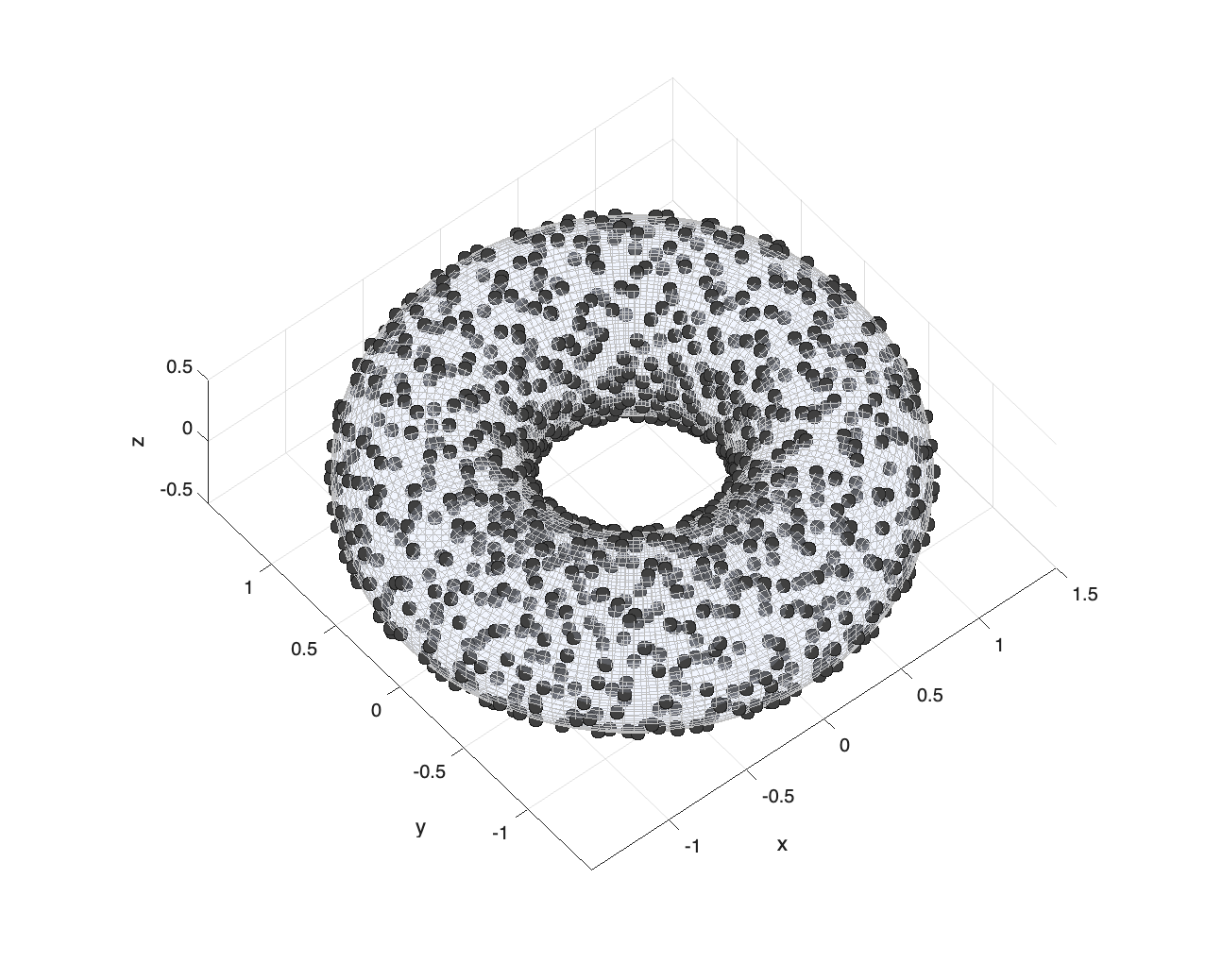}
\end{minipage}\hfill
\begin{minipage}{0.48\textwidth}
\centering
\includegraphics[width=\textwidth]{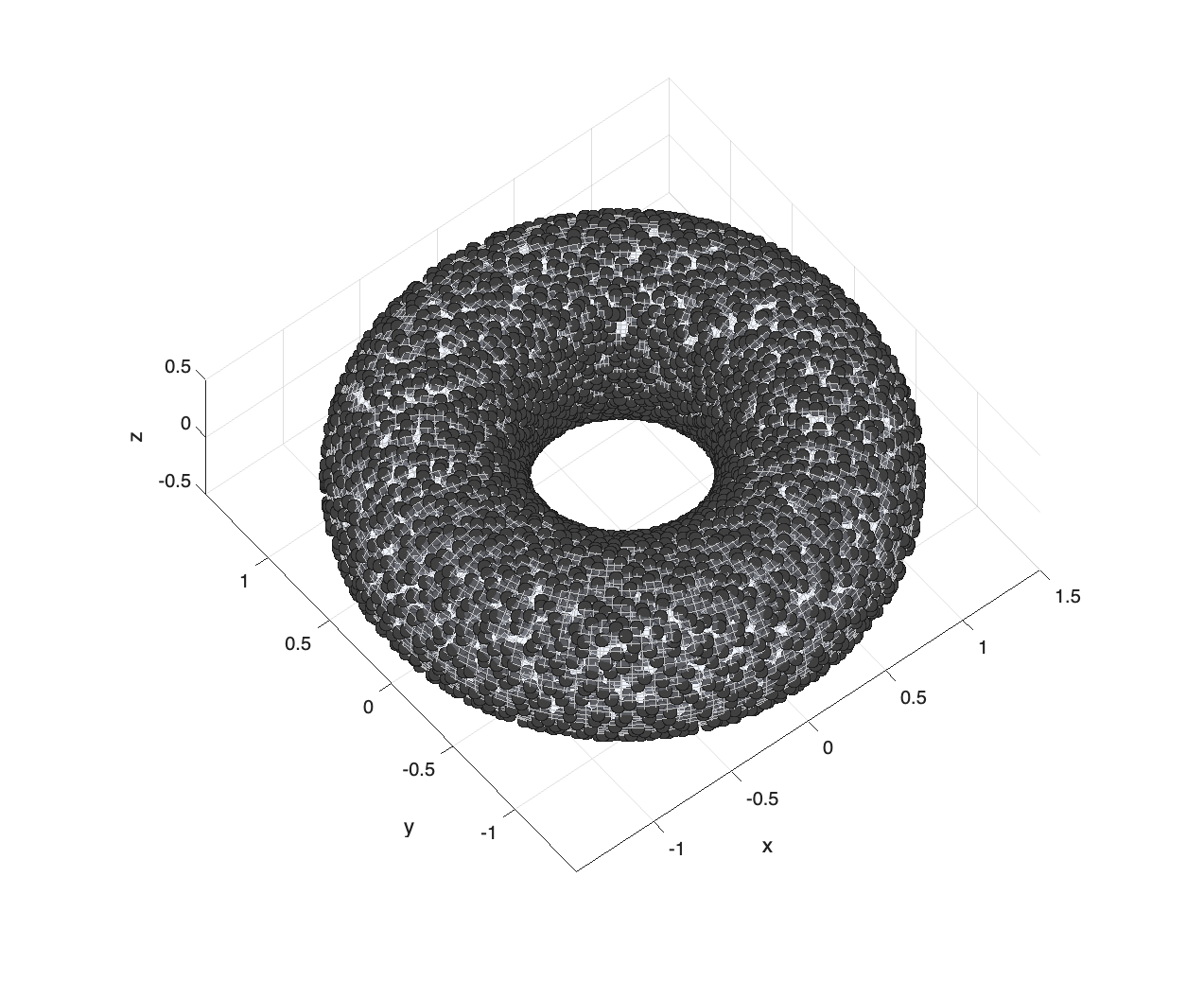}
\end{minipage}
\caption{Halton nodes with bases $2$ and $3$, generated in the periodic parameter domain and mapped onto the toroidal surface for $N=1000$ (left) and $N=4000$ (right).}
\label{fig:halton_nodes}
\end{figure}

\subsection{Varying the number of interpolation nodes}

In the first experiment, we investigate the influence of the node density on
the approximation accuracy of the compactly supported multinode Shepard operator. To this end, the
degree $d$ of the local polynomial space $\HH_d(\TT)$ is fixed at $d=4$, while $4$
Halton node sets of increasing cardinality,
$N=1000,\;4000,\;16000,\;64000,$
are considered.
Figure~\ref{fig:halton_nodes} shows the Halton nodes mapped onto the toroidal
surface for $N=1000$ and $N=4000$. As the number of nodes increases, the
sampling of the torus becomes progressively denser while preserving the
scattered, low-discrepancy character inherited from the periodic parameter domain.

The results, reported in Figure~\ref{fig:error_nodes}, show a systematic reduction of the interpolation errors as the number of nodes increases. This experiment records error decay under node refinement; a direct numerical measurement of the asymptotic order in Corollary~\ref{cor:fill-distance-convergence} would additionally require the corresponding fill distances and experimental orders of convergence. For the polynomial test functions, the errors reach
values close to machine precision whenever the fixed local polynomial space
contains the function being approximated, in agreement with the polynomial
reproduction property established in Section~\ref{sec:multinode-framework}. For the smooth
non-polynomial functions, the errors also decrease as the node distribution is
refined, although the observed rate depends on the oscillatory and localized
features of each test function.
For the test functions $f_{13}$--$f_{15}$, the errors are reported in Table~\ref{tab:toroidal_errors}. The mean error of $f_{13}$ decreases from $5.68e-4$ to $7.37e-10$, and similar reductions are observed for the Cartesian functions restricted to the torus. These results show that the method provides accurate approximations
both for intrinsically parametrized toroidal functions and for
Cartesian functions restricted from the ambient space.

\begin{figure}[htbp]
\centering
\begin{minipage}{0.3\textwidth}
    \centering
    \includegraphics[width=\textwidth]{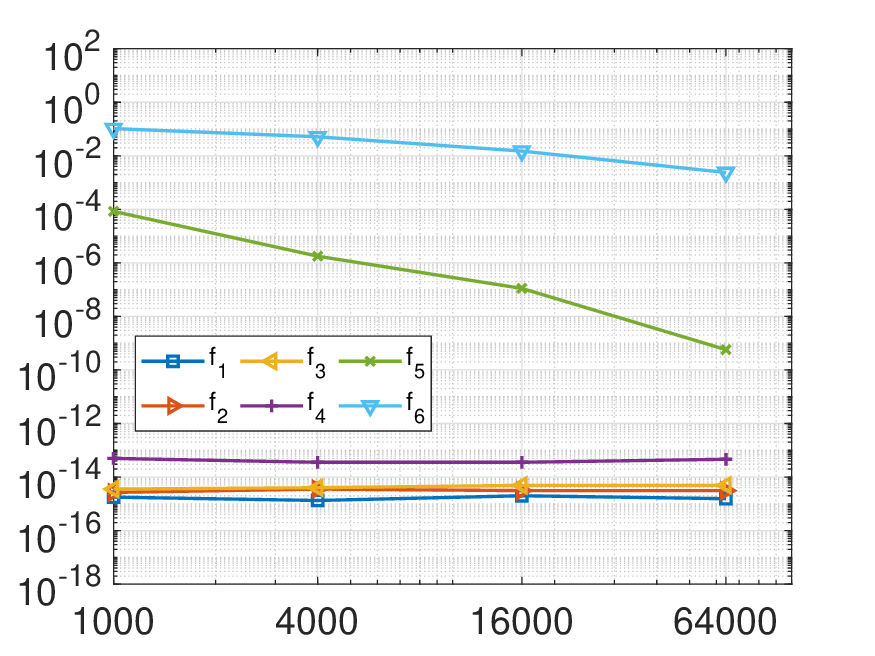}
    \par\smallskip
    $E_{\max}$
\end{minipage}
\hfill
\begin{minipage}{0.3\textwidth}
    \centering
    \includegraphics[width=\textwidth]{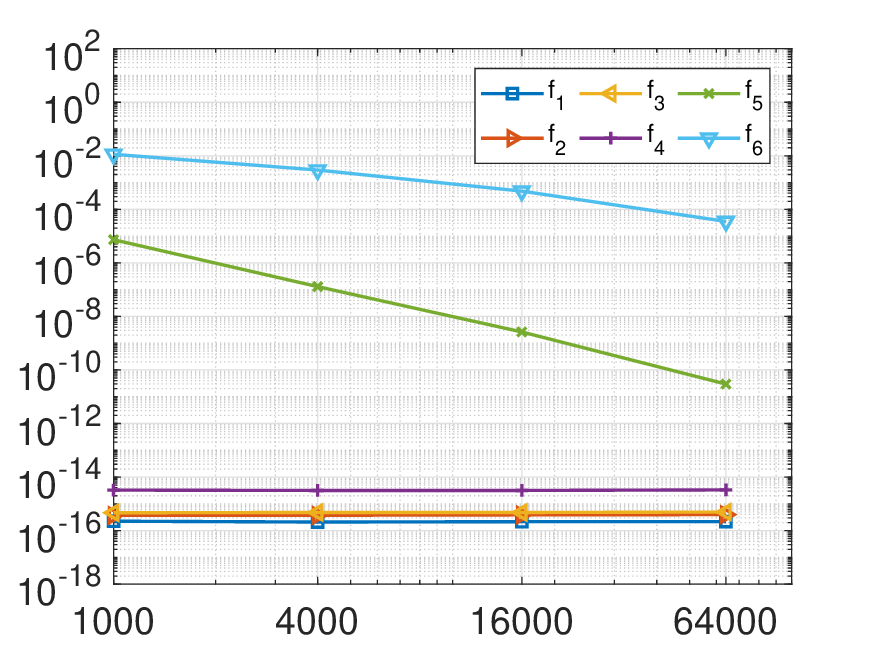}
    \par\smallskip
    $E_{\operatorname{mean}}$
\end{minipage}
\hfill
\begin{minipage}{0.3\textwidth}
    \centering
    \includegraphics[width=\textwidth]{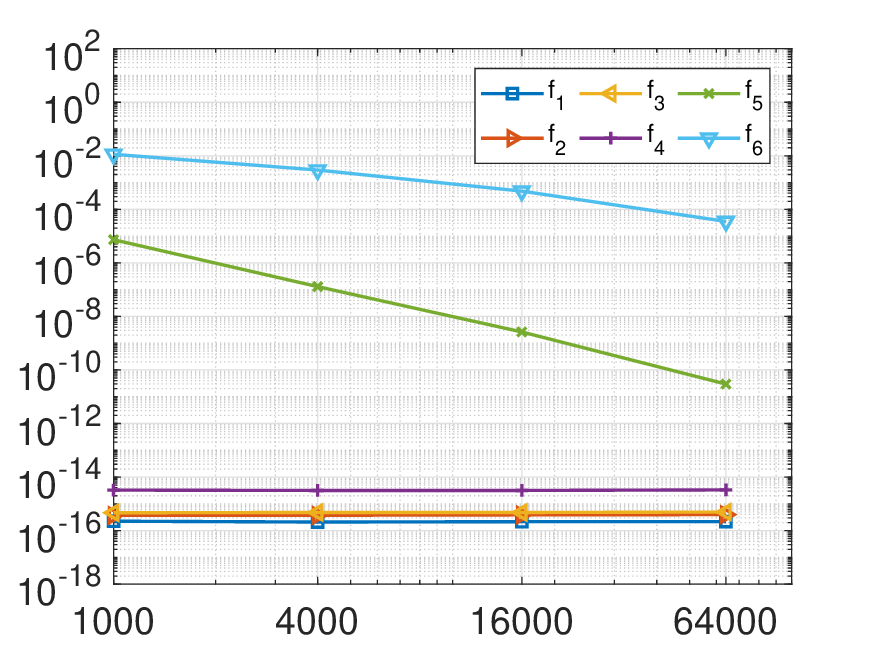}
    \par\smallskip
    $E_{\operatorname{RMS}}$
\end{minipage}

\vspace{0.7cm}

\begin{minipage}{0.3\textwidth}
    \centering
    \includegraphics[width=\textwidth]{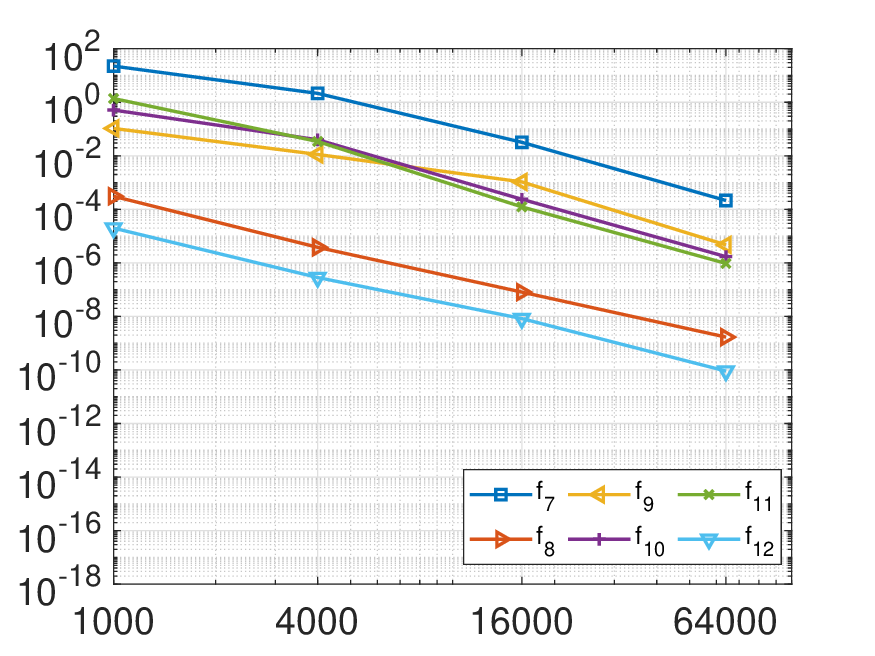}
    \par\smallskip
    $E_{\max}$
\end{minipage}
\hfill
\begin{minipage}{0.3\textwidth}
    \centering
    \includegraphics[width=\textwidth]{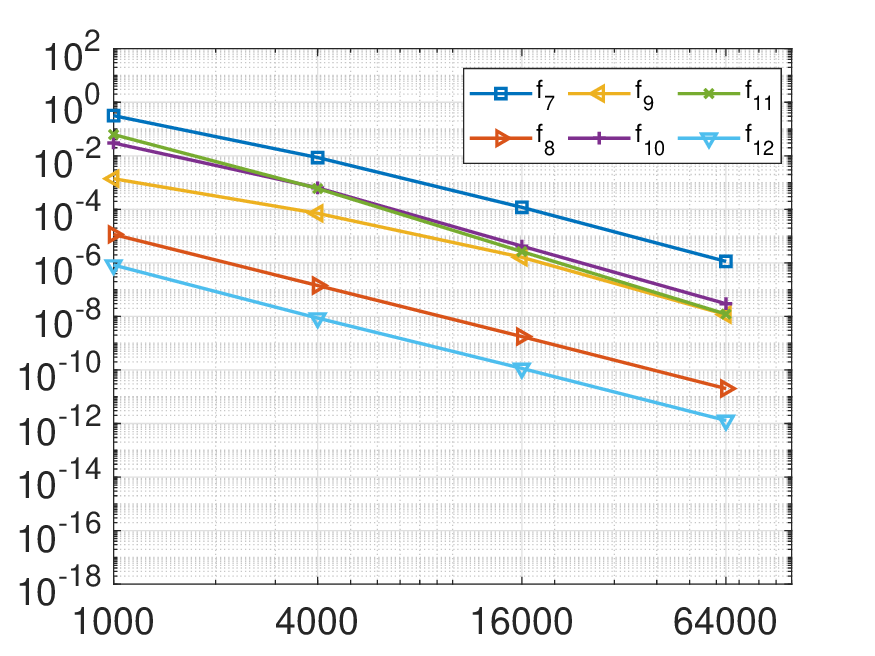}
    \par\smallskip
    $E_{\operatorname{mean}}$
\end{minipage}
\hfill
\begin{minipage}{0.3\textwidth}
    \centering
    \includegraphics[width=\textwidth]{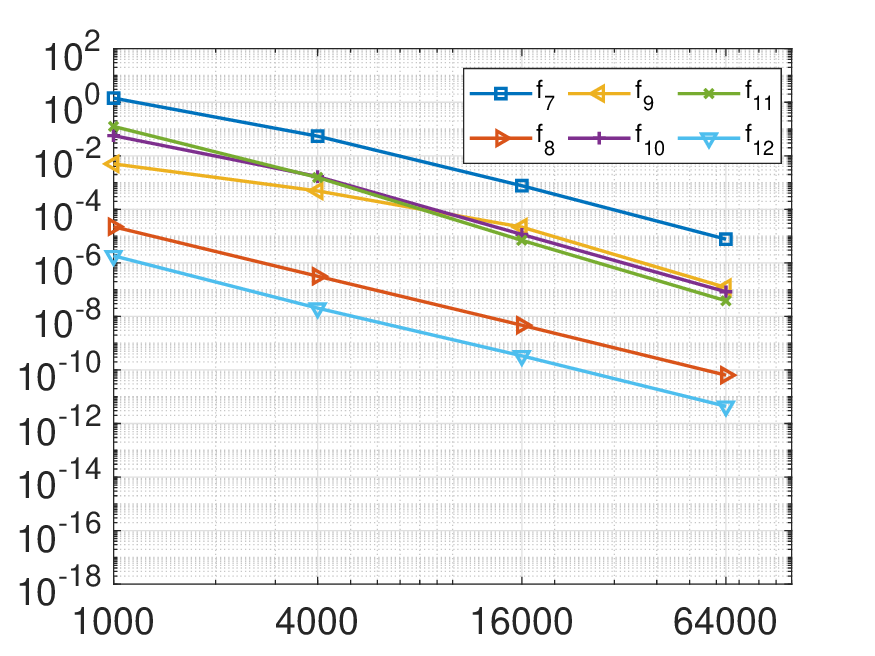}
    \par\smallskip
    $E_{\operatorname{RMS}}$
\end{minipage}

\caption{Maximum, mean and RMS errors for test functions
$f_1$--$f_{12}$ as the number of interpolation nodes increases.
The first row refers to $f_1$--$f_6$, while the second row refers to
$f_7$--$f_{12}$.}
\label{fig:error_nodes}
\end{figure}

\begin{table}[htbp]
\centering
\caption{Maximum, mean and RMS interpolation errors for the toroidal test functions using Halton nodes.}
\label{tab:toroidal_errors}
\renewcommand{\arraystretch}{1.15}
{\scriptsize
\begin{tabular}{ccccc}
\toprule
$N$ & Error & $f_{13}$ & $f_{14}$ & $f_{15}$ \\
\midrule
1000
& $E_{\max}$
& $7.1830\mathrm{e}{-03}$
& $1.5824\mathrm{e}{-06}$
& $7.8130\mathrm{e}{-03}$ \\
&
$E_{\operatorname{mean}}$
& $5.6814\mathrm{e}{-04}$
& $4.1570\mathrm{e}{-08}$
& $3.5813\mathrm{e}{-04}$ \\
&
$E_{\operatorname{RMS}}$
& $9.8138\mathrm{e}{-04}$
& $9.9274\mathrm{e}{-08}$
& $7.8055\mathrm{e}{-04}$ \\
\midrule
4000
& $E_{\max}$
& $1.9703\mathrm{e}{-04}$
& $2.6561\mathrm{e}{-08}$
& $4.7358\mathrm{e}{-04}$ \\
&
$E_{\operatorname{mean}}$
& $6.8079\mathrm{e}{-06}$
& $4.4516\mathrm{e}{-10}$
& $6.5464\mathrm{e}{-06}$ \\
&
$E_{\operatorname{RMS}}$
& $1.4134\mathrm{e}{-05}$
& $1.2150\mathrm{e}{-09}$
& $2.2069\mathrm{e}{-05}$ \\
\midrule
16000
& $E_{\max}$
& $4.2509\mathrm{e}{-06}$
& $6.2795\mathrm{e}{-10}$
& $6.6096\mathrm{e}{-06}$ \\
&
$E_{\operatorname{mean}}$
& $8.0788\mathrm{e}{-08}$
& $5.8819\mathrm{e}{-12}$
& $8.7574\mathrm{e}{-08}$ \\
&
$E_{\operatorname{RMS}}$
& $2.0435\mathrm{e}{-07}$
& $1.9284\mathrm{e}{-11}$
& $3.1834\mathrm{e}{-07}$ \\
\midrule
64000
& $E_{\max}$
& $4.9854\mathrm{e}{-08}$
& $5.6213\mathrm{e}{-12}$
& $1.9147\mathrm{e}{-07}$ \\
&
$E_{\operatorname{mean}}$
& $7.3710\mathrm{e}{-10}$
& $6.0890\mathrm{e}{-14}$
& $1.0037\mathrm{e}{-09}$ \\
&
$E_{\operatorname{RMS}}$
& $2.0609\mathrm{e}{-09}$
& $2.1468\mathrm{e}{-13}$
& $4.7687\mathrm{e}{-09}$ \\
\bottomrule
\end{tabular}
}
\end{table}

\subsection{Varying the degree of the local polynomial space}

The second experiment investigates the influence of the local polynomial degree
on the approximation accuracy of the compactly supported multinode Shepard operator. In this case, the
number of interpolation nodes is fixed at $N=64000$, while the degree of the
local polynomial space $\HH_d(\TT)$ varies from $d=1$ to $d=6$.
 The
corresponding results are shown in Figure~\ref{fig:error_degree} and Table~\ref{tab:toroidal_errors_degree}.

As expected, enriching the local approximation space generally leads to a
progressive reduction of the interpolation errors. In particular, polynomial
test functions are reproduced up to machine precision as soon as the local
polynomial space contains the target function, thereby confirming the
theoretical polynomial reproduction property also in this setting. For the remaining smooth
non-polynomial functions, increasing the polynomial degree generally
improves the approximation, although the amount of improvement depends on the
regularity and complexity of the function under consideration.

\begin{figure}[htbp]
\centering

\begin{minipage}{0.3\textwidth}
    \centering
    \includegraphics[width=\textwidth]{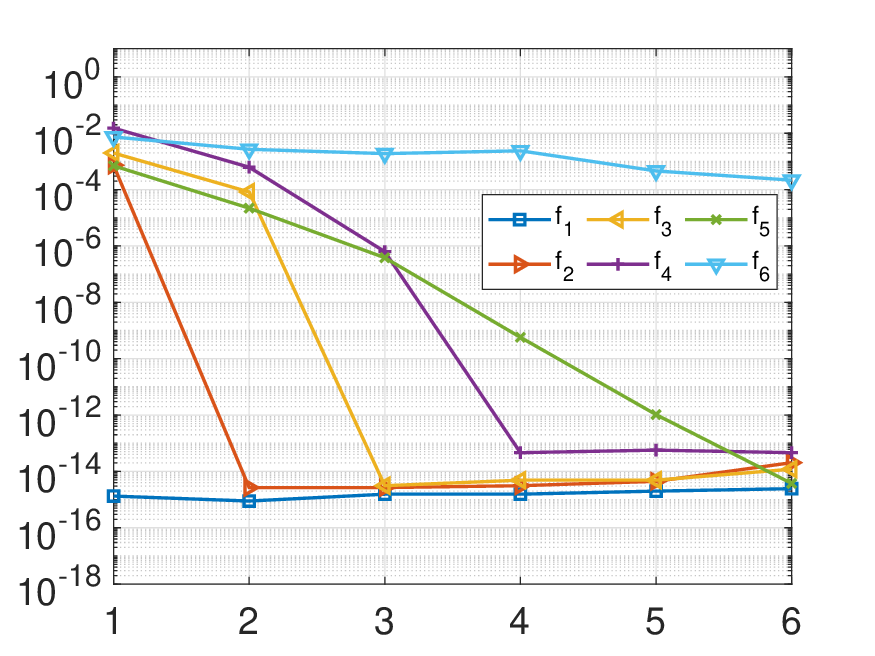}
    \par\smallskip
    $E_{\max}$
\end{minipage}
\hfill
\begin{minipage}{0.3\textwidth}
    \centering
    \includegraphics[width=\textwidth]{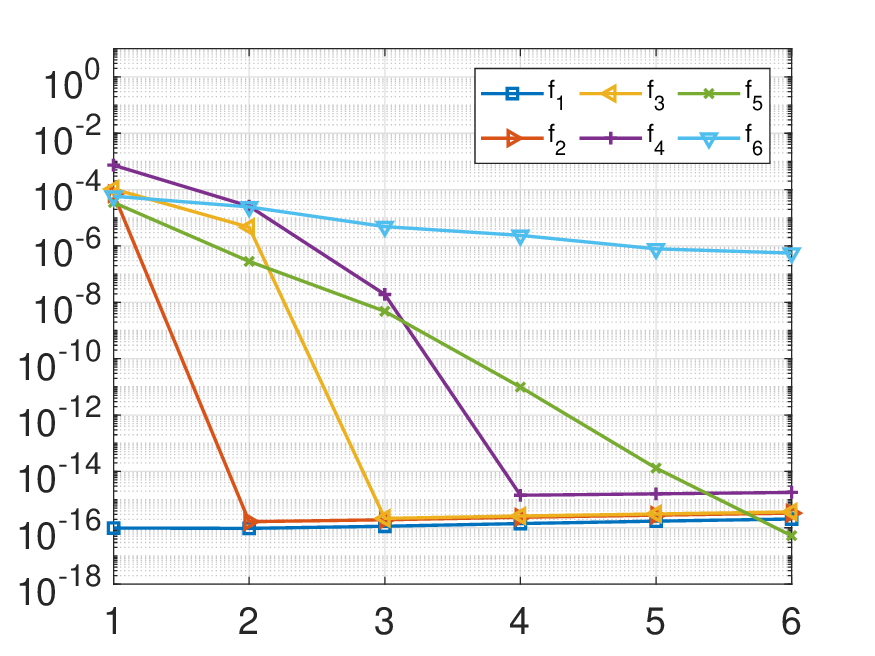}
    \par\smallskip
    $E_{\operatorname{mean}}$
\end{minipage}
\hfill
\begin{minipage}{0.3\textwidth}
    \centering
    \includegraphics[width=\textwidth]{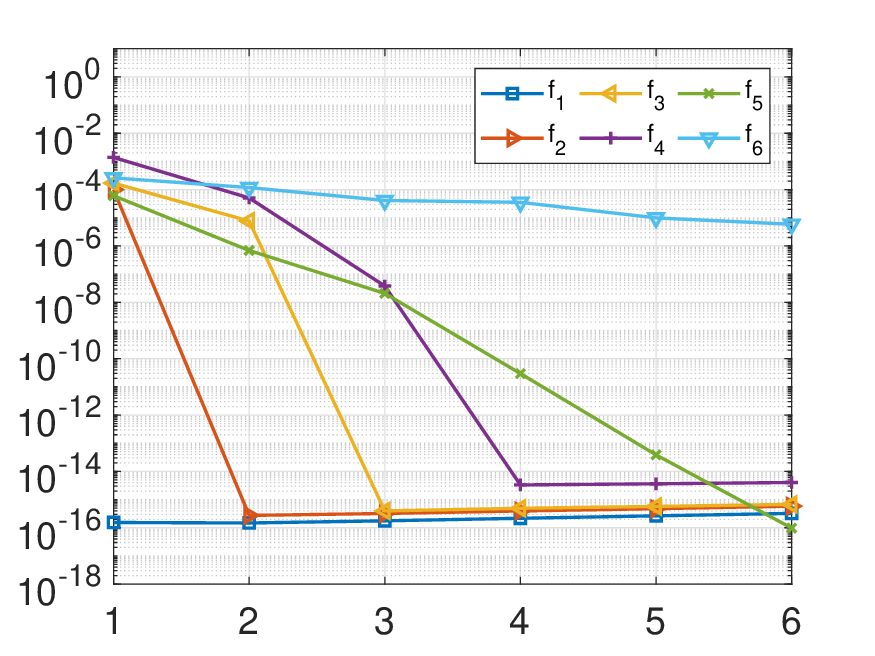}
    \par\smallskip
    $E_{\operatorname{RMS}}$
\end{minipage}

\vspace{0.7cm}

\begin{minipage}{0.3\textwidth}
    \centering
    \includegraphics[width=\textwidth]{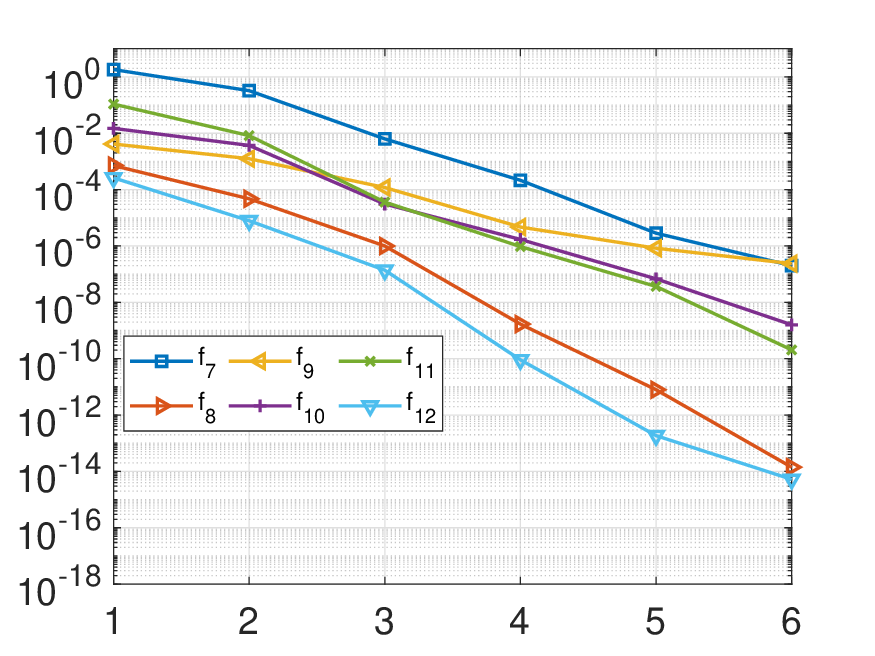}
    \par\smallskip
    $E_{\max}$
\end{minipage}
\hfill
\begin{minipage}{0.3\textwidth}
    \centering
    \includegraphics[width=\textwidth]{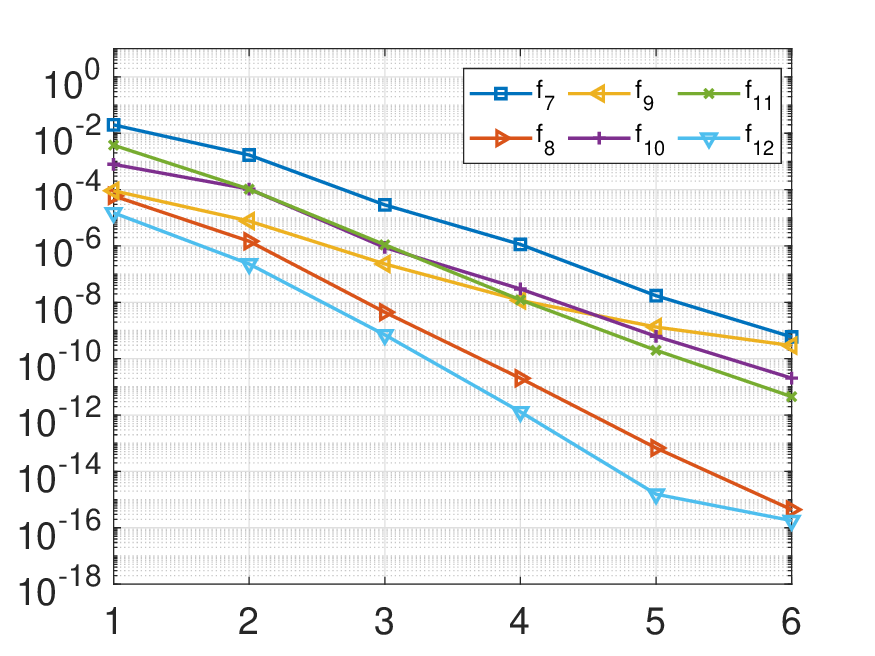}
    \par\smallskip
    $E_{\operatorname{mean}}$
\end{minipage}
\hfill
\begin{minipage}{0.3\textwidth}
    \centering
    \includegraphics[width=\textwidth]{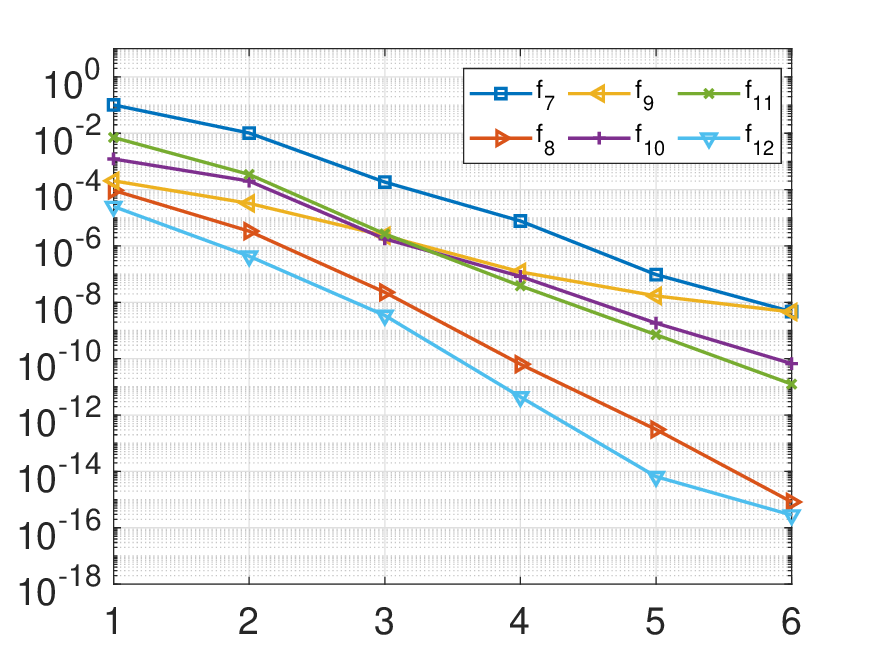}
    \par\smallskip
    $E_{\operatorname{RMS}}$
\end{minipage}

\caption{Maximum, mean and RMS errors for test functions
$f_1$--$f_{12}$ as the degree $d$ of $\HH_d(\TT)$ increases,
with fixed $N=64000$. The first row refers to $f_1$--$f_6$,
while the second row refers to $f_7$--$f_{12}$.}
\label{fig:error_degree}
\end{figure}

\begin{table}[htbp]
\centering
\caption{Maximum, mean and RMS interpolation errors for the toroidal test functions with $N=64000$ Halton nodes and varying degree.}
\label{tab:toroidal_errors_degree}
\renewcommand{\arraystretch}{1.15}
{\scriptsize
\begin{tabular}{ccccc}
\toprule
$d$ & Error & $f_{13}$ & $f_{14}$ & $f_{15}$ \\
\midrule
1
& $E_{\max}$
& $3.8013\mathrm{e}{-03}$
& $2.1418\mathrm{e}{-04}$
& $7.7137\mathrm{e}{-03}$ \\
&
$E_{\operatorname{mean}}$
& $1.6580\mathrm{e}{-04}$
& $1.2218\mathrm{e}{-05}$
& $3.1304\mathrm{e}{-04}$ \\
&
$E_{\operatorname{RMS}}$
& $2.4671\mathrm{e}{-04}$
& $1.9907\mathrm{e}{-05}$
& $5.4374\mathrm{e}{-04}$ \\
\midrule
2
& $E_{\max}$
& $1.7793\mathrm{e}{-04}$
& $2.8392\mathrm{e}{-06}$
& $7.4520\mathrm{e}{-04}$ \\
&
$E_{\operatorname{mean}}$
& $1.3554\mathrm{e}{-05}$
& $1.1558\mathrm{e}{-07}$
& $1.8794\mathrm{e}{-05}$ \\
&
$E_{\operatorname{RMS}}$
& $2.0555\mathrm{e}{-05}$
& $2.3585\mathrm{e}{-07}$
& $4.3459\mathrm{e}{-05}$ \\
\midrule
3
& $E_{\max}$
& $1.6103\mathrm{e}{-05}$
& $7.5761\mathrm{e}{-09}$
& $4.3644\mathrm{e}{-05}$ \\
&
$E_{\operatorname{mean}}$
& $1.4977\mathrm{e}{-07}$
& $7.7848\mathrm{e}{-11}$
& $1.3168\mathrm{e}{-07}$ \\
&
$E_{\operatorname{RMS}}$
& $4.8508\mathrm{e}{-07}$
& $3.5652\mathrm{e}{-10}$
& $7.7026\mathrm{e}{-07}$ \\
\midrule
4
& $E_{\max}$
& $4.9854\mathrm{e}{-08}$
& $5.6213\mathrm{e}{-12}$
& $1.9147\mathrm{e}{-07}$ \\
&
$E_{\operatorname{mean}}$
& $7.3710\mathrm{e}{-10}$
& $6.0890\mathrm{e}{-14}$
& $1.0037\mathrm{e}{-09}$ \\
&
$E_{\operatorname{RMS}}$
& $2.0609\mathrm{e}{-09}$
& $2.1468\mathrm{e}{-13}$
& $4.7687\mathrm{e}{-09}$ \\
\midrule
5
& $E_{\max}$
& $3.7783\mathrm{e}{-10}$
& $4.4409\mathrm{e}{-15}$
& $7.8785\mathrm{e}{-10}$ \\
&
$E_{\operatorname{mean}}$
& $3.7206\mathrm{e}{-12}$
& $4.1177\mathrm{e}{-16}$
& $4.6280\mathrm{e}{-12}$ \\
&
$E_{\operatorname{RMS}}$
& $1.1353\mathrm{e}{-11}$
& $5.8260\mathrm{e}{-16}$
& $2.1505\mathrm{e}{-11}$ \\
\midrule
6
& $E_{\max}$
& $2.6190\mathrm{e}{-12}$
& $5.7732\mathrm{e}{-15}$
& $3.5354\mathrm{e}{-12}$ \\
&
$E_{\operatorname{mean}}$
& $3.5166\mathrm{e}{-14}$
& $4.7784\mathrm{e}{-16}$
& $2.9345\mathrm{e}{-14}$ \\
&
$E_{\operatorname{RMS}}$
& $1.0124\mathrm{e}{-13}$
& $6.9030\mathrm{e}{-16}$
& $1.3149\mathrm{e}{-13}$ \\
\bottomrule
\end{tabular}
}
\end{table}

\FloatBarrier

\subsection{High-degree conditioning and sensitivity diagnostics}
\label{sec:high-degree-diagnostics}

The very small errors reported in Figure~\ref{fig:error_degree} and
Table~\ref{tab:toroidal_errors_degree} for $d=5$ and $d=6$ coexist with
increasingly ill-conditioned local coefficient systems. We therefore
performed a separate a posteriori diagnostic using exactly the same
$N=64000$ Halton nodes, the same candidate neighbourhoods, and the same
stencils selected by the original $PA=LU$ procedure. 

It is important to distinguish the two Vandermonde matrices occurring in
the implementation. The rectangular candidate matrix used for the LU-based
selection is locally centred but is not scaled. After a stencil
$\sigma_j=\{\bs{x}_{j_1},\ldots,\bs{x}_{j_{m_d}}\}$ has been selected, we set
\[
 \overline{\bs{x}}_j
 =\frac{1}{m_d}\sum_{i=1}^{m_d}\bs{x}_{j_i},
 \qquad
 \delta_j
 =\max_{1\leq i\leq m_d}
   \norm{\bs{x}_{j_i}-\overline{\bs{x}}_j}_2,
\]
and construct the final square matrix in the locally centred and
isotropically scaled coordinates
\[
 \widehat{\bs{x}}_{j,i}
 =\frac{\bs{x}_{j_i}-\overline{\bs{x}}_j}{\delta_j},
 \qquad
 (V_j)_{ik}=\beta_k(\widehat{\bs{x}}_{j_i}).
\]
All reciprocal condition estimates reported below refer to this final square
matrix $V_j$, not to the rectangular candidate matrix.

For each stencil, we measured the relative $\Pi V =LU$ factorization residual
\[
 r_j^{\rm LU}
 =\frac{\norm{\Pi_jV_j-L_jU_j}_{\infty}}
        {\norm{V_j}_{\infty}}
\]
and, for the test functions $f_{\nu}$, $\nu=1,\dots,15$, the largest normwise backward error
\[
 \eta_j
 =\max_{1\leq \nu\leq 15}
 \frac{\norm{V_j\bs{c}_{\nu,j}-\bs{f}_{\nu,j}}_{\infty}}
 {\norm{V_j}_{\infty}\norm{\bs{c}_{\nu,j}}_{\infty}
  +\norm{\bs{f}_{\nu,j}}_{\infty}}.
\]
The principal results are summarized in Table~\ref{tab:high-degree-diagnostics}.
Here $\varepsilon_{\rm mach}=\mathrm{eps}$ denotes the double-precision machine epsilon used by MATLAB.
The quantity $E_{\rm rep}^{\max}$ is the largest error obtained by reproducing,
on candidate points not belonging to the selected stencil, every polynomial function of
the complete local reduced basis and twelve normalized linear combinations of
that basis, after scaling each test to unit local $\ell^\infty$ amplitude. Moreover, for a local candidate set $Y_j$ we define the sampled
Lebesgue factor
\[
 \widehat\lambda_j
 =\max_{\bs{\xi}\in Y_j\setminus\sigma_j}
   \sum_{k=1}^{m_d}\abs{\ell_{j,k}(\bs{\xi})}.
\]
This is a discrete diagnostic and must not be identified with the continuous
factor $\Lambda_h$ in~\eqref{eq:Lambda-h}.

The reciprocal condition estimates show that the monomial coefficient
representation becomes particularly sensitive for $d=6$. Nevertheless, the LU
factorization residuals and the normwise backward errors remain at the level of
double-precision roundoff. Thus the linear systems are solved with small
backward errors, although their coefficient vectors may be highly sensitive to
perturbations.

To assess the local interpolants rather than only their coefficient systems, we
next tested exact reproduction away from the interpolation nodes. For every
stencil, all $m_d$ functions of the local reduced basis, together with twelve
normalized linear combinations, were interpolated on $\sigma_j$ and evaluated
on $Y_j\setminus\sigma_j$. The maximum errors in
Table~\ref{tab:high-degree-diagnostics} remain of order $10^{-15}$ for both
degrees. Hence the complete high-degree spaces used in the computation are
reproduced to essentially machine precision on these independent local test
points.

Finally, the data of $f_{13}$, $f_{14}$, and $f_{15}$ were perturbed by a vector
$\bs{\eta}$ and the observed local amplification was measured by
\[
 \mathcal A_j(\bs{\eta})
 =\frac{
 \norm{P_j[\bs{f}+\bs{\eta}]-P_j[\bs{f}]}_{\ell^\infty(Y_j\setminus\sigma_j)}}
 {\norm{\bs{\eta}}_{\ell^\infty(\sigma_j)}}.
\]
The statistics for relative perturbations of size $10^{-12}$ are reported in
Table~\ref{tab:high-degree-diagnostics}. At this level, and also at level
$10^{-10}$, all observed amplification factors remained below the corresponding
sampled Lebesgue factors. Repeating the test at level $10^{-10}$
produced essentially unchanged medians and $95$th percentiles; the corresponding
maxima were $12.08$ for $d=5$ and $21.60$ for $d=6$. Perturbations of size
$10^{-14}$ lie too close to the floating-point error floor to provide a reliable
worst-case amplification measurement and are therefore not used in the
interpretation.

Additional experiments compared the original representation with
column equilibration, a standard diagonal scaling technique for linear
systems \cite{CurtisReid1972,Higham2002}, and with the Newton-like
representation naturally associated with the LU factorization at
discrete Leja points \cite{Bos,BosNewton2011}. We also compared
centered candidate Vandermonde matrices with and without isotropic
coordinate scaling, while retaining LU row pivoting as the
stencil-selection mechanism \cite{Bos,DellAccioDiTommasoSiar2021}.
None of these variants produced a systematic improvement in the final
interpolation accuracy or in the condition distribution of the selected
square systems. We therefore retain the original $PA=LU$ construction. The
present diagnostics do not prove a uniform high-degree stability bound; rather,
they show that, for the smooth data considered here, very small basis-dependent
reciprocal condition estimates coexist with machine-precision reproduction and
moderate observed amplification. At the operator level, the sampled Lebesgue
factor is consequently the more pertinent quantity for interpreting sensitivity
to data perturbations.

\begin{table}[htbp]
\centering
\caption{High-degree diagnostics for the original $PA=LU$ implementation.
Panel (a) concerns the final square, centred and scaled Vandermonde matrices.
Panel (b) reports operator-level diagnostics on local candidate points. The
noise-amplification statistics correspond to relative perturbations of size
$10^{-12}$; the median and $95$th-percentile entries give the ranges over
$f_{13}$--$f_{15}$, while the maximum is taken over all three functions, all
stencils, and five perturbation realizations.}
\label{tab:high-degree-diagnostics}

{\scriptsize
\setlength{\tabcolsep}{3.5pt}
\renewcommand{\arraystretch}{1.15}

\textbf{(a) Local coefficient systems}
\par\smallskip

\begin{adjustbox}{max width=\textwidth}
\begin{tabular}{ccccccc}
\toprule
$d$
& $m_d$
& $L$
& $\operatorname{median}\operatorname{rcond}(V_j)$
& $\operatorname{rcond}(V_j)<\varepsilon_{\rm mach}$
& $\max_j r_j^{\rm LU}$
& $\max_j \eta_j$ \\
\midrule
5
& 52
& 4982
& $4.29\mathrm{e}{-15}$
& $1.63\%$
& $2.07\mathrm{e}{-16}$
& $3.33\mathrm{e}{-16}$ \\
6
& 74
& 3842
& $1.68\mathrm{e}{-17}$
& $90.53\%$
& $2.08\mathrm{e}{-16}$
& $3.56\mathrm{e}{-16}$ \\
\bottomrule
\end{tabular}
\end{adjustbox}

\medskip

\textbf{(b) Reproduction, sampled Lebesgue factors, and noise amplification}
\par\smallskip

\begin{adjustbox}{max width=\textwidth}
\begin{tabular}{cccccc}
\toprule
$d$
& $E_{\rm rep}^{\max}$
& $\operatorname{median}_j \widehat{\lambda}_j$
& $\max_j \widehat{\lambda}_j$
& $\operatorname{median}\mathcal{A}\,/\,q_{0.95}(\mathcal{A})$
& $\max\mathcal{A}$ \\
\midrule
5
& $5.55\mathrm{e}{-15}$
& 16.07
& 44.25
& $2.31$--$2.32\,/\,4.30$--$4.33$
& 12.23 \\
6
& $5.11\mathrm{e}{-15}$
& 21.17
& 92.63
& $2.63$--$2.66\,/\,4.86$--$4.94$
& 20.12 \\
\bottomrule
\end{tabular}
\end{adjustbox}
}

\end{table}

\section{Application to Computational Fluid Dynamics data on the torus}\label{sec:cfd}

We now test the method on data arising from Computational Fluid Dynamics simulations \cite{Jakob,gerris}. The dataset contains a two-dimensional velocity field on a periodic square domain, with components
\[
\bs{u}(x,y)=(u(x,y),v(x,y)).
\]
The field originates from the numerical simulation of an incompressible flow and is therefore governed, at the modelling level, by the incompressible Navier--Stokes equations
\begin{equation*}
\frac{\partial \bs{u}}{\partial t}+(\bs{u}\cdot\nabla)\bs{u}=-\nabla p+\nu\Delta\bs{u},
\qquad \nabla\cdot\bs{u}=0,
\end{equation*}
where $p$ is the pressure and $\nu$ is the kinematic viscosity. The data are defined on a periodic $512\times512$ grid and are mapped to the torus by setting $\alpha=2\pi x$, $\beta=2\pi y,
$ and
\begin{equation*}
\begin{aligned}
X(\alpha,\beta)=\big(& (R+r\cos\beta)\cos\alpha,\;
(R+r\cos\beta)\sin\alpha,\;
 r\sin\beta\big).
\end{aligned}
\end{equation*}
with $R=3$ and $r=1$.
We interpret the two velocity components in the orthonormal tangent frame of the toroidal parametrization. More precisely, we define
\[
\boldsymbol e_{\alpha}
=
\frac{\boldsymbol X_{\alpha}}
     {\|\boldsymbol X_{\alpha}\|_2}
=
(-\sin\alpha,\cos\alpha,0),
\]
and
\[
\boldsymbol e_{\beta}
=
\frac{\boldsymbol X_{\beta}}
     {\|\boldsymbol X_{\beta}\|_2}
=
(-\sin\beta\cos\alpha,
-\sin\beta\sin\alpha,
\cos\beta).
\]
The reference tangent velocity field is then obtained through the
orthonormal tangent lifting
\[
\boldsymbol{\mathcal U}
=
u\,\boldsymbol e_{\alpha}
+
v\,\boldsymbol e_{\beta}.
\]
After reconstructing the two scalar components separately, the
corresponding approximating tangent field is
\[
\boldsymbol{\mathcal U}_h
=
u_h\,\boldsymbol e_{\alpha}
+
v_h\,\boldsymbol e_{\beta}.
\]
Since the tangent frame is orthonormal,
\[
\|\boldsymbol{\mathcal U}\|_2
=
\sqrt{u^2+v^2},
\qquad
\|\boldsymbol{\mathcal U}
-\boldsymbol{\mathcal U}_h\|_2
=
\sqrt{(u-u_h)^2+(v-v_h)^2}.
\]
Figure~\ref{fig:cfd_field} compares the velocity field on the original
periodic computational domain with its orthonormal tangent lifting onto
the torus.

\begin{figure}[htbp]
\centering
\begin{minipage}{0.48\textwidth}\centering\includegraphics[width=\textwidth]{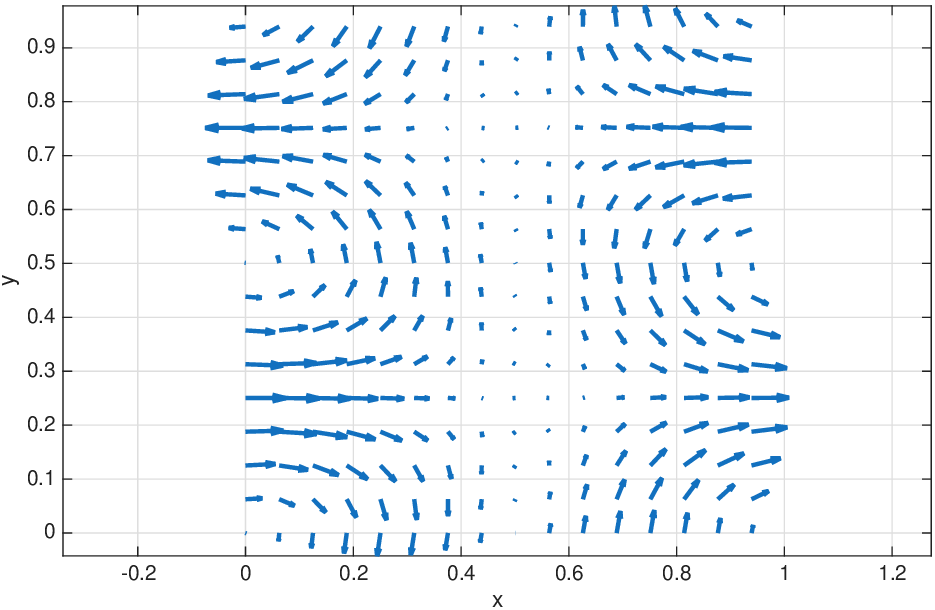}\end{minipage}\hfill
\begin{minipage}{0.48\textwidth}\centering\includegraphics[width=\textwidth]{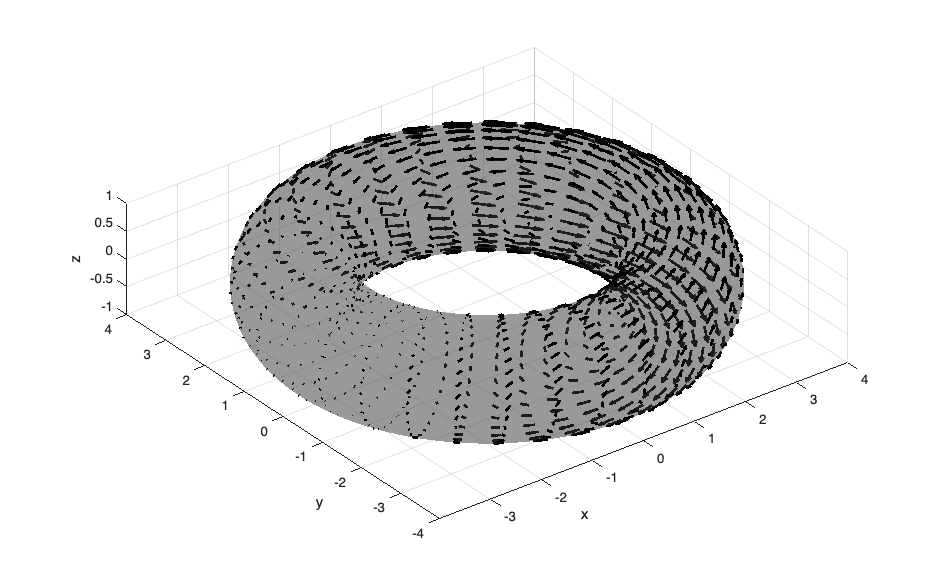}\end{minipage}
\caption{Velocity field extracted from the CFD dataset:
original periodic computational domain (left) and orthonormal tangent lifting onto the torus (right).}
\label{fig:cfd_field}
\end{figure}

The interpolation tests use $2000$ data points, split into $1800$ training nodes and $200$ validation nodes. The training nodes are used to construct the local toroidal interpolants and the compactly supported weights, while the validation nodes are used only to assess accuracy. In the scalar component experiments, the construction generated $244$ local interpolation stencils.

\subsection{Velocity components and magnitude}

The componentwise interpolation of $u$ and $v$ is highly accurate. The component $u$ has mean error $3.4561e-6$ and maximum error $5.5247e-5$, while the component $v$ has mean error $1.3859e-6$ and maximum error $1.6019e-5$. Table~\ref{tab:cfd_scalar_errors} summarizes the scalar results, including the velocity magnitude $\|\boldsymbol{\mathcal U}\|_2=(u^2+v^2)^{1/2}$. The reference and reconstructed values of the component $u$ are compared in Figure~\ref{fig:u_cfd}, while Figure~\ref{fig:mag_cfd} displays the reference and reconstructed velocity magnitudes.

The velocity magnitude is not interpolated as an independent scalar quantity. After reconstructing the two velocity components separately, we define
\[
\|\boldsymbol{\mathcal U}_h\|_2=\bigl(u_h^2+v_h^2\bigr)^{1/2}.
\]
At every validation node, the reverse triangle inequality gives
\[
\bigl|\|\boldsymbol{\mathcal U}\|_2-\|\boldsymbol{\mathcal U}_h\|_2\bigr|
\leq
\|\boldsymbol{\mathcal U}-\boldsymbol{\mathcal U}_h\|_2.
\]
Consequently, the maximum, mean, and root mean square errors of the reconstructed magnitude are bounded above by the corresponding vector-field errors. The numerical values reported in Table~\ref{tab:cfd_scalar_errors} are consistent with this estimate.

\begin{table}[htbp]
\centering
\caption{Error indicators for scalar quantities extracted from the CFD velocity field.}
\label{tab:cfd_scalar_errors}
\begin{tabular}{lccc}
\toprule
Quantity & $E_{\operatorname{\max}}$ & $E_{\operatorname{mean}}$ & $E_{\operatorname{RMS}}$ \\
\midrule
$u$ & $5.5247\mathrm{e}{-5}$ & $3.4561\mathrm{e}{-6}$ & $7.7880\mathrm{e}{-6}$ \\
$v$ &  $1.6019\mathrm{e}{-5}$ & $1.3859\mathrm{e}{-6}$& $2.5329\mathrm{e}{-6}$ \\
$\|\boldsymbol{\mathcal U}\|_2$ & $5.4015\mathrm{e}{-5}$ & $3.3540\mathrm{e}{-6}$ & $7.6888\mathrm{e}{-6}$ \\
\bottomrule
\end{tabular}
\end{table}

\begin{figure}[htbp]
\centering
\begin{minipage}{0.48\textwidth}\centering\includegraphics[width=\textwidth]{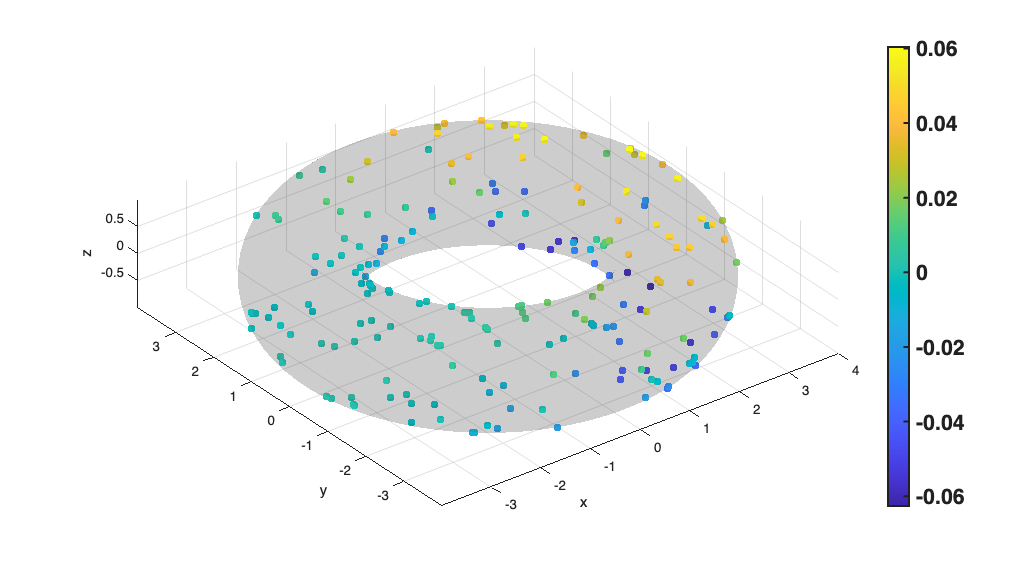}\end{minipage}\hfill
\begin{minipage}{0.48\textwidth}\centering\includegraphics[width=\textwidth]{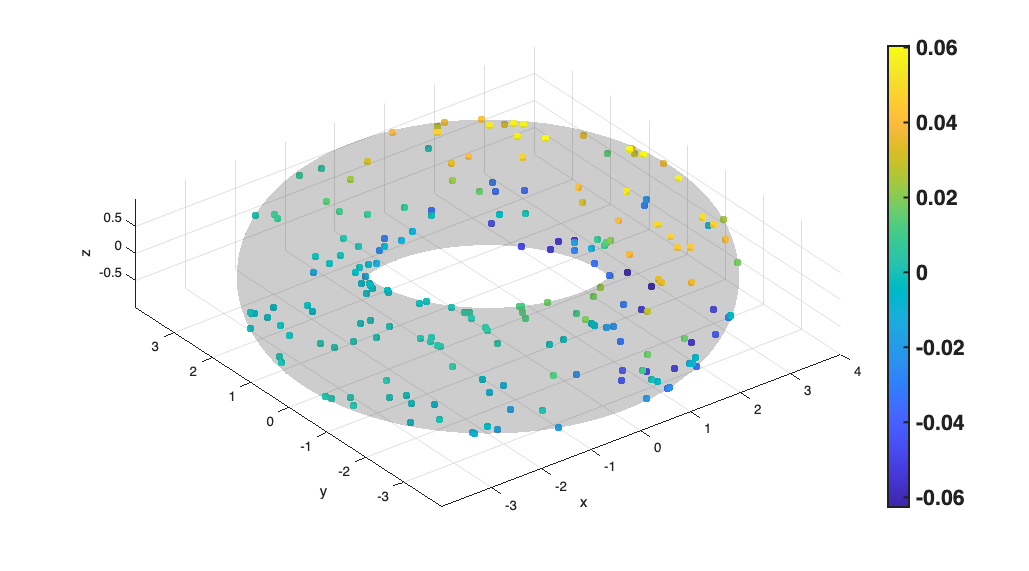}\end{minipage}
\caption{Velocity component $u$ on validation nodes: reference values (left) and compactly supported multinode Shepard reconstruction (right).}
\label{fig:u_cfd}
\end{figure}

\begin{figure}[htbp]
\centering
\begin{minipage}{0.48\textwidth}\centering\includegraphics[width=\textwidth]{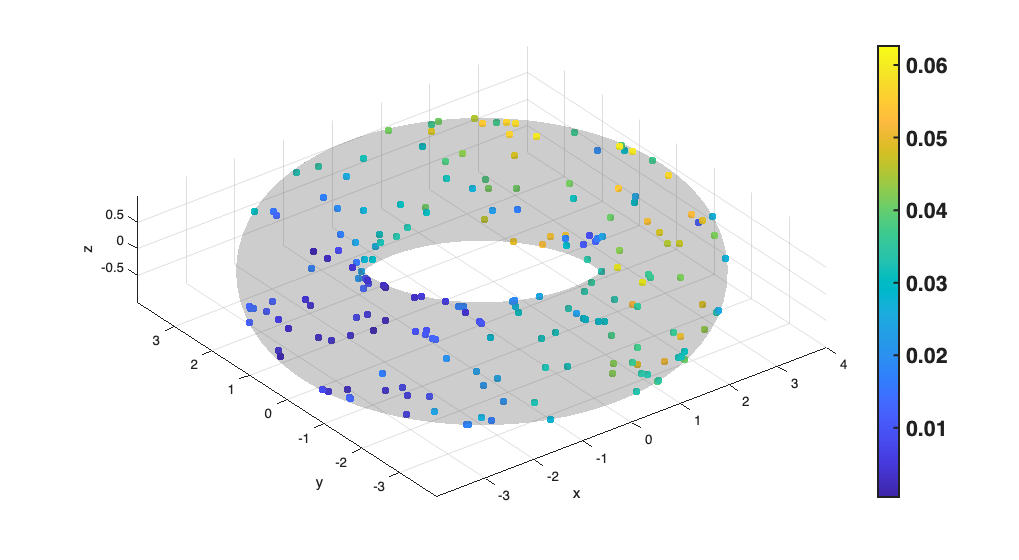}\end{minipage}\hfill
\begin{minipage}{0.48\textwidth}\centering\includegraphics[width=\textwidth]{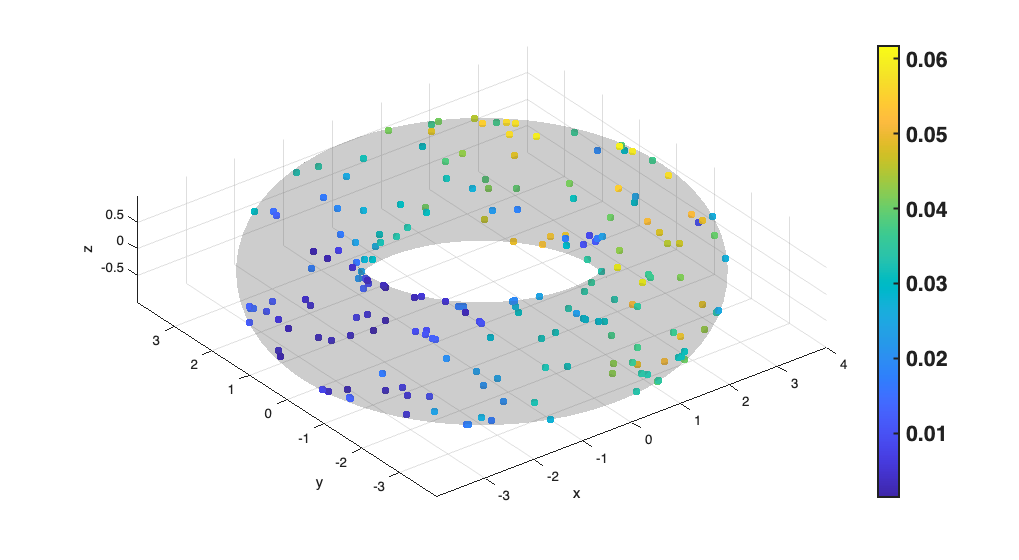}\end{minipage}
\caption{Velocity magnitude on validation nodes: reference values (left) and reconstruction obtained from the interpolated components $u_h$ and $v_h$ (right).}
\label{fig:mag_cfd}
\end{figure}

\subsection{Vector field reconstruction}

For the full tangent-field reconstruction, the two scalar compactly supported multinode Shepard interpolants are combined through the orthonormal tangent frame
\[
\boldsymbol{\mathcal U}_h
=
u_h\boldsymbol e_{\alpha}
+
v_h\boldsymbol e_{\beta}.
\]
The tangent-field error at a validation node is
\[
E(\boldsymbol x)
=
\|\boldsymbol{\mathcal U}(\boldsymbol x)
-\boldsymbol{\mathcal U}_h(\boldsymbol x)\|_2
=
\sqrt{(u-u_h)^2+(v-v_h)^2}.
\]
The numerical indicators are
\[
E_{\operatorname{mean}}=4.0557e-6,\qquad
E_{\max}=5.5268e-5,\qquad
E_{\operatorname{RMS}}=8.1895e-6.
\]
Figure~\ref{fig:vector_cfd} shows the reference and reconstructed tangent velocity fields, while Figure~\ref{fig:vector_error} reports the local tangent-field error. The agreement between the two tangent-field plots indicates that both direction and magnitude are preserved by the interpolation process.

\begin{figure}[htbp]
\centering
\begin{minipage}{0.48\textwidth}\centering\includegraphics[width=\textwidth]{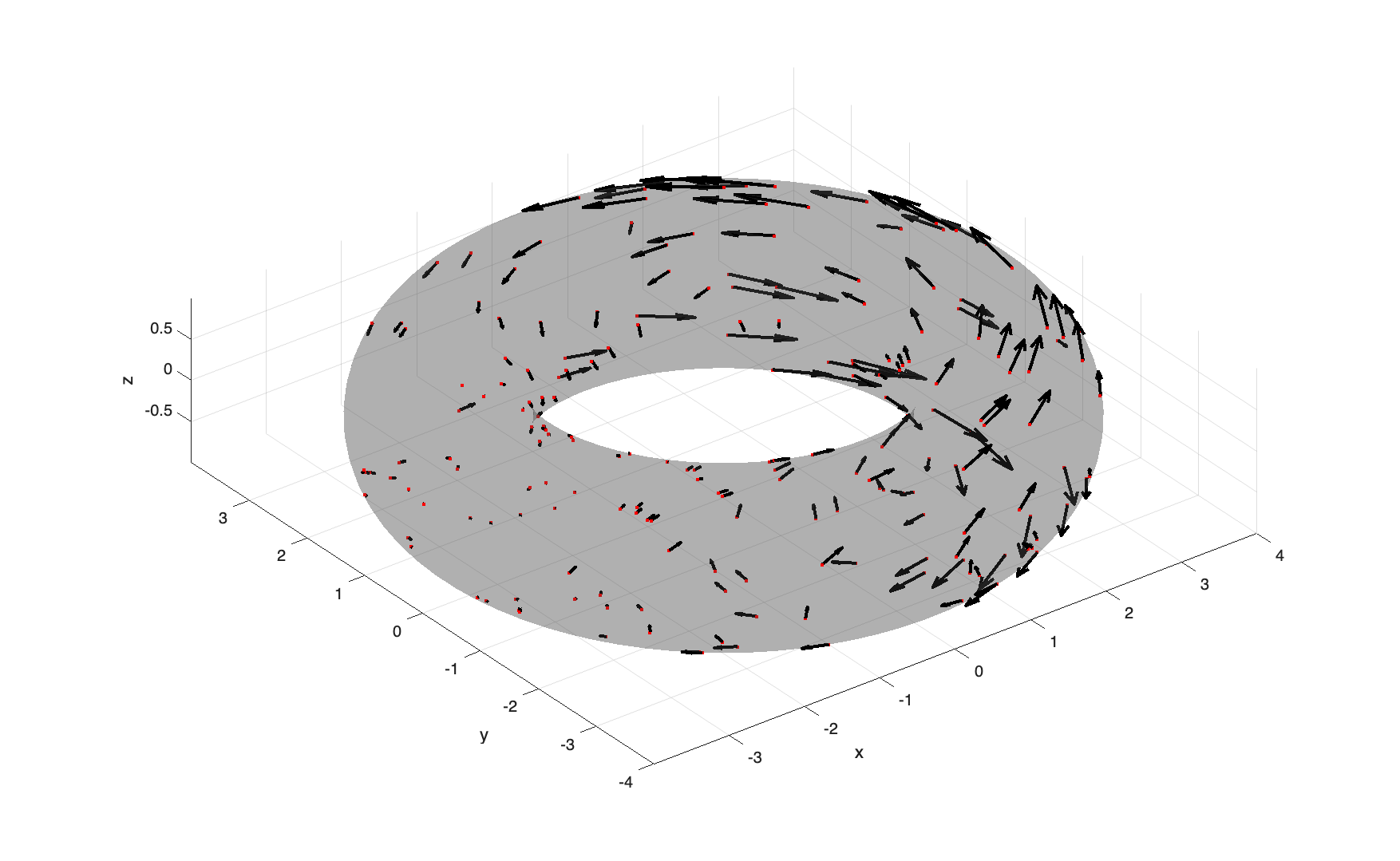}\end{minipage}\hfill
\begin{minipage}{0.48\textwidth}\centering\includegraphics[width=\textwidth]{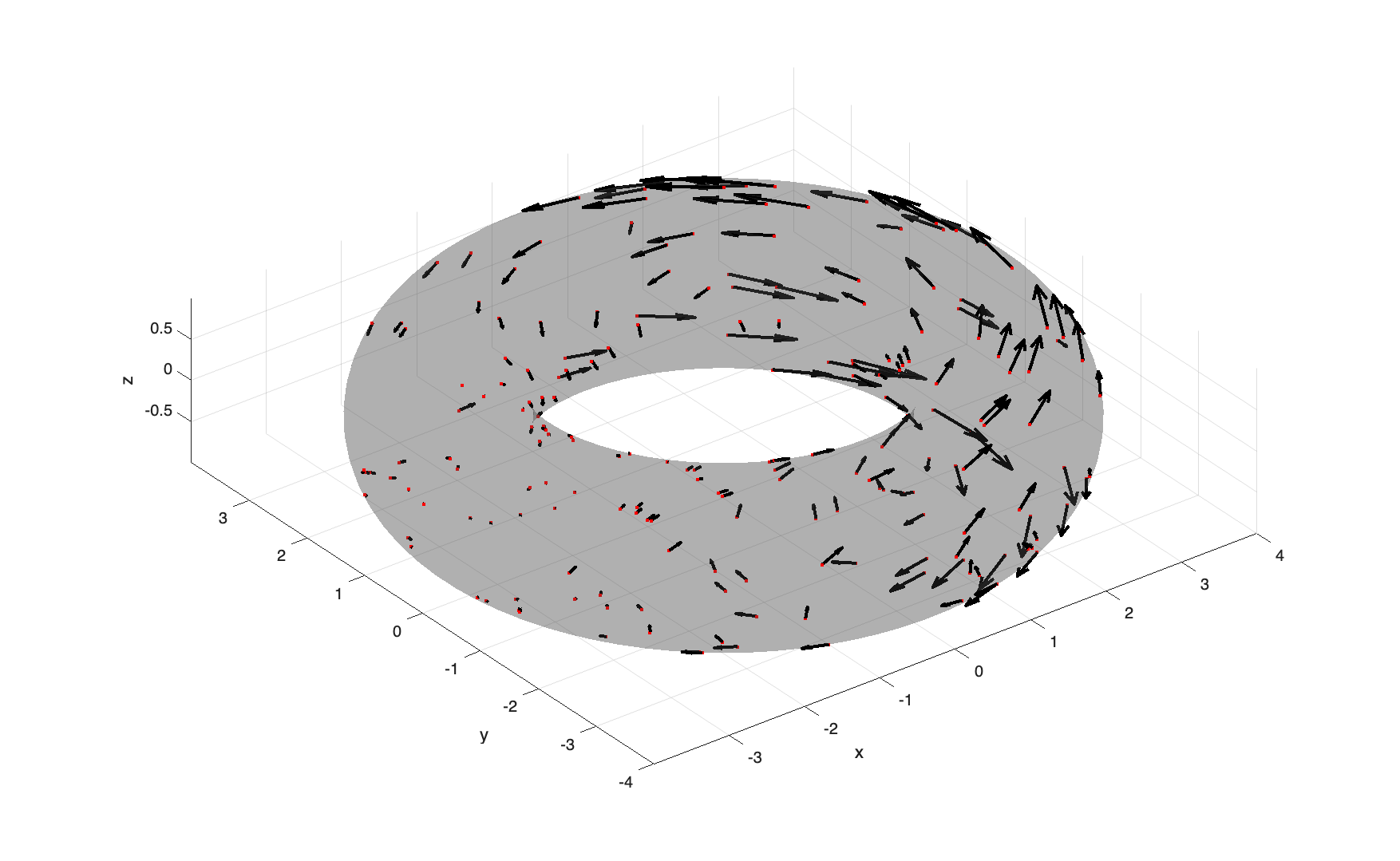}\end{minipage}
\caption{Tangent velocity field on the validation nodes:
reference field (left) and compactly supported multinode Shepard reconstruction obtained through the
orthonormal tangent frame (right).}
\label{fig:vector_cfd}
\end{figure}

\begin{figure}[htbp]
\centering
\includegraphics[width=0.5\textwidth]{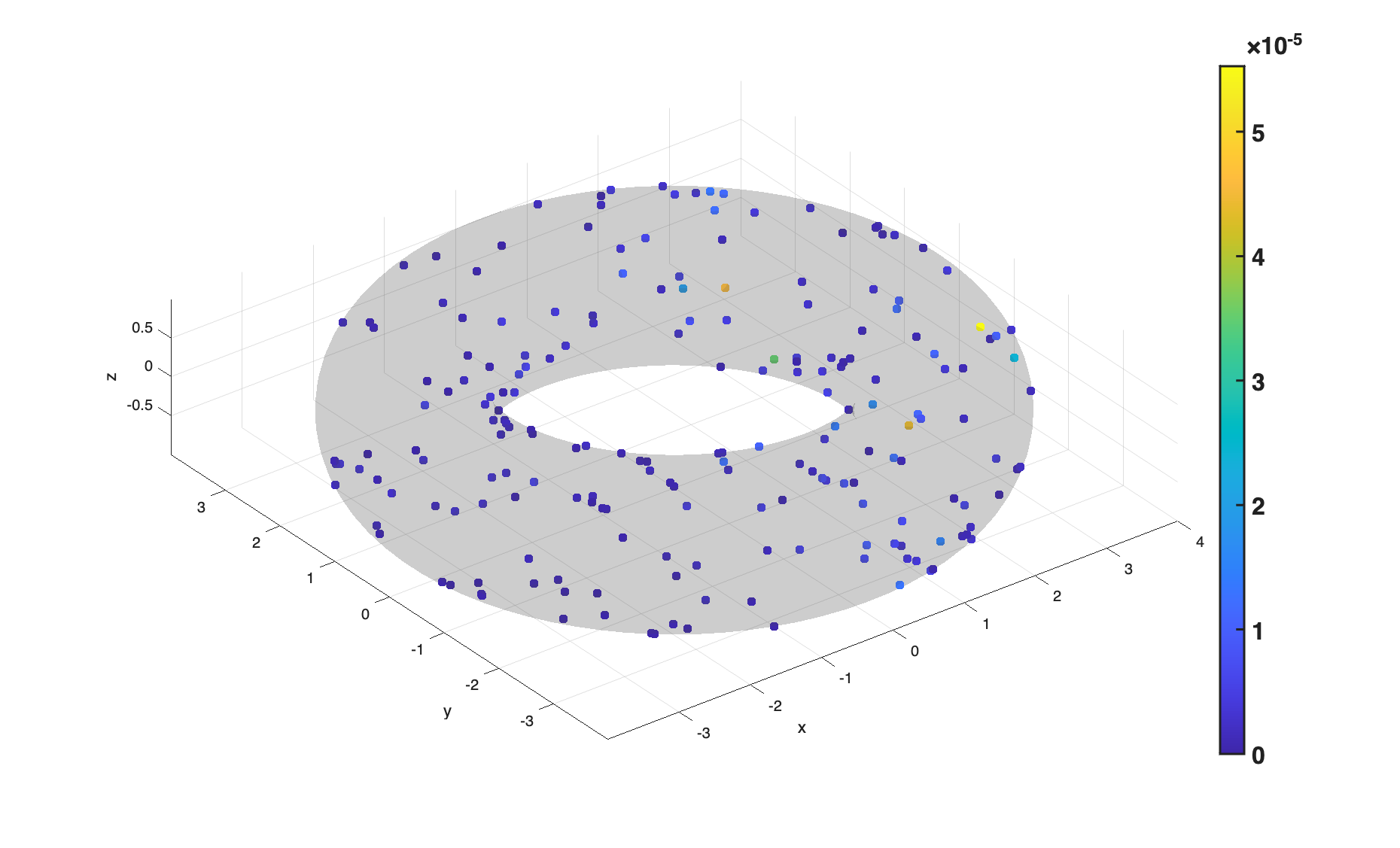}
\caption{Tangent-field interpolation error
$\|\boldsymbol{\mathcal U}-\boldsymbol{\mathcal U}_h\|_2$
on the validation nodes.}
\label{fig:vector_error}
\end{figure}

\section{Implementation and LU-based stencil construction}\label{sec:algorithms}

The implementation follows the interpolatory multinode construction described
above and consists of three stages: generation of the reduced toroidal basis,
selection of a covering family of minimal stencils, and stable evaluation of the
compactly supported blend.  The analytical experiments and stability diagnostics reported in Section~4
were carried out in MATLAB R2024a.  The underlying implementation extends the MATLAB framework for
multinode Shepard interpolation developed in \cite{dell2024multinode}.

\subsection{Reduced basis and local Vandermonde matrices}

Let $Y=\{\bs{y}_1,\ldots,\bs{y}_q\}\subset\TT$, let
$\bs{c}\in\RR^3$, and let $\delta>0$.  For the ordered reduced basis
$\mathcal B_d=\{\beta_1,\ldots,\beta_{m_d}\}$, we denote by
\begin{equation*}
 \mathcal V_d(Y;\bs{c},\delta)
 =\left[
 \beta_k\!\left(\frac{\bs{y}_i-\bs{c}}{\delta}\right)
 \right]_{\substack{1\le i\le q\\1\le k\le m_d}}
 \in\RR^{q\times m_d}.
\end{equation*}
the local toroidal Vandermonde matrix. 
The translated and scaled basis functions remain restrictions of
ambient polynomials of total degree at most $d$. Whenever the resulting
square Vandermonde matrix is nonsingular, these $m_d$ functions are
linearly independent in $\mathcal H_d(\mathbb T)$ and therefore form a
basis of the same restriction space.

\begin{algorithm}[H]
\caption{Reduced toroidal Vandermonde matrix in local coordinates}
\label{alg:local-vandermonde}
\begin{algorithmic}[1]
\Require Degree $d$, points $Y=\{\bs{y}_1,\ldots,\bs{y}_q\}$,
center $\bs{c}$, scale $\delta>0$
\Ensure Matrix $\mathcal V_d(Y;\bs{c},\delta)$
\State Generate all monomials $x^iy^jz^k$ satisfying $i+j+k\le d$.
\State Remove the monomials divisible by $\operatorname{LM}(F)=x^4$.
\State Order the remaining monomials as
$\mathcal B_d=\{\beta_1,\ldots,\beta_{m_d}\}$.
\For{$i=1,\ldots,q$}
  \State $\widehat{\bs{y}}_i\gets(\bs{y}_i-\bs{c})/\delta$.
  \For{$k=1,\ldots,m_d$}
    \State $[\mathcal V_d]_{ik}\gets\beta_k(\widehat{\bs{y}}_i)$.
  \EndFor
\EndFor
\end{algorithmic}
\end{algorithm}

\subsection{Selection of minimal stencils by row-pivoted LU}

The stencil family is constructed once the node set
$X=\{\bs{x}_1,\ldots,\bs{x}_n\}$ and the degree $d$ have been fixed.  The
algorithm maintains a list $I_{\rm rem}$ of indices not yet covered.  Its first
entry is used as an anchor, whereas the local candidate points are always drawn
from the full set $X$; consequently, different stencils may overlap.

The initial candidate neighborhood is an axis-aligned box in $\RR^3$.  In the
implementation, its density-dependent length parameter is
\begin{equation*}
 \ell_d=\left(\frac{m_{d+1}\,|\Omega_X|}{n}\right)^{1/2},
 \qquad
 |\Omega_X|=(x_{\max}-x_{\min})(y_{\max}-y_{\min})(z_{\max}-z_{\min}).
\end{equation*}
and the box is enlarged in increments of ten percent until it contains at least
$m_{d+1}=2((d+1)^2+1)$ nodes.  The candidates are then ordered by increasing
ambient Euclidean distance from the anchor.

Let $Y_j=\{\bs{y}_{j,1},\ldots,\bs{y}_{j,q_j}\}$ denote the ordered candidate
set and let
\[
 \bs{c}^{\rm cand}_j=\frac1{q_j}\sum_{i=1}^{q_j}\bs{y}_{j,i}.
\]
The rectangular candidate matrix is
\begin{equation}
 A_j=\mathcal V_d(Y_j;\bs{c}^{\rm cand}_j,1).
 \label{eq:candidate-vandermonde}
\end{equation}
Thus this matrix is centered but deliberately not scaled.  Partial row pivoting
in
\begin{equation*}
 \Pi_jA_j=L_jU_j.
\end{equation*}
produces a permutation vector $\pi_j$.  The first $m_d$ pivot rows define the
minimal stencil
\begin{equation*}
 \sigma_j=\{\bs{y}_{j,\pi_j(1)},\ldots,
             \bs{y}_{j,\pi_j(m_d)}\}.
\end{equation*}
This is the same Leja-type extraction principle already used in the spherical
multinode construction and in numerical algorithms for discrete Leja points
\cite{DellAccioSphere,Bos,DellAccioDiTommasoSiar2021}.  The indices selected at
the current step are removed from $I_{\rm rem}$, and the process is repeated
until the list is empty.

\begin{algorithm}[H]
\caption{LU-based construction of the stencil family}
\label{alg:lu-stencils}
\begin{algorithmic}[1]
\Require Nodes $X=\{\bs{x}_1,\ldots,\bs{x}_n\}$, degree $d$,
initial length $\ell_d$
\Ensure Covering family $\Sigma=\{\sigma_1,\ldots,\sigma_L\}$
\State $m\gets m_d$, $q_{\min}\gets m_{d+1}$,
$I_{\rm rem}\gets\{1,\ldots,n\}$, $\Sigma\gets\varnothing$.
\While{$I_{\rm rem}\ne\varnothing$}
  \State Let $a$ be the first index in $I_{\rm rem}$ and set $s\gets0$.
  \Repeat
    \State $Y\gets\{\bs{x}_i\in X:
    \norm{\bs{x}_i-\bs{x}_a}_{\infty}
    \le (1+0.1s)\ell_d/2\}$.
    \State $s\gets s+1$.
  \Until{$\#Y\ge q_{\min}$}
  \State Order the points of $Y$ by increasing
  $\norm{\bs{y}-\bs{x}_a}_2$.
  \State $\bs{c}_{Y}\gets(\#Y)^{-1}\sum_{\bs{y}\in Y}\bs{y}$.
  \State $A\gets\mathcal V_d(Y;\bs{c}_{Y},1)$.
  \State Compute $A(\pi,:)=LU$ by partial row pivoting.
  \State $\sigma\gets\{\bs{y}_{\pi(1)},\ldots,
  \bs{y}_{\pi(m)}\}$ and $\Sigma\gets\Sigma\cup\{\sigma\}$.
  \State Remove the indices of the points in $\sigma$ from $I_{\rm rem}$.
\EndWhile
\end{algorithmic}
\end{algorithm}

The distinction between the matrices in \eqref{eq:candidate-vandermonde} and in the local solve below
is essential: the candidate matrix is centered and unscaled, whereas the final
square matrix is both centered and isotropically scaled.  The diagnostics in
Section~\ref{sec:high-degree-diagnostics} showed that scaling the candidate
matrix changes only a small fraction of the selected nodes and does not yield a
systematic improvement in the conditioning of the final square systems.

\subsection{Local interpolation and stable compact blending}

For each selected stencil
$\sigma_j=\{\bs{x}_{j_1},\ldots,\bs{x}_{j_{m_d}}\}$, the implementation sets
\begin{equation*}
 \overline{\bs{x}}_j
 =\frac1{m_d}\sum_{k=1}^{m_d}\bs{x}_{j_k},
 \qquad
 \delta_j
 =\max_{1\le k\le m_d}
   \norm{\bs{x}_{j_k}-\overline{\bs{x}}_j}_2.
\end{equation*}
and forms the final square matrix
\begin{equation*}
 V_j=\mathcal V_d(\sigma_j;\overline{\bs{x}}_j,\delta_j).
\end{equation*}
The coefficient vector is obtained from
\begin{equation*}
 V_j\bs{c}_j=\bs{f}_j.
\end{equation*}
by Gaussian elimination with partial pivoting, implemented by MATLAB's direct
backslash solver.  The local polynomial is evaluated in the same centered and
scaled coordinates,
\begin{equation}
 P_j[f](\bs{x})=
 \sum_{k=1}^{m_d}c_{j,k}\,
 \beta_k\!\left(\frac{\bs{x}-\overline{\bs{x}}_j}{\delta_j}\right).
 \label{eq:local-polynomial-evaluation}
\end{equation}

The support radius is then set to
$R_j=\operatorname{diam}_E(\sigma_j)+H$.  A stencil is declared active at
$\bs{x}$ only after the exact geometric test
\begin{equation*}
 d_E(\bs{x},\bs{x}_{j_k})<R_j,
 \qquad k=1,\ldots,m_d.
\end{equation*}
has been satisfied.  A preliminary barycentric test is used only as a cheap
filter and does not alter the active set.

At data nodes, the prescribed value is returned directly, thereby avoiding the
singular inverse-distance expression.  At all other active points the
unnormalized weight is evaluated logarithmically:
\begin{equation}
 \log\omega_{\mu,j}(\bs{x})
 =\mu\sum_{k=1}^{m_d}
 \left[
 \log\!\left(1-\frac{d_E(\bs{x},\bs{x}_{j_k})}{R_j}\right)
 -\log d_E(\bs{x},\bs{x}_{j_k})
 \right].
 \label{eq:log-unnormalized-weight}
\end{equation}
The normalized blend is computed by a log-sum-exp normalization.  More
precisely, with
$a(\bs{x})=\max_{j\in J_{\bs{x}}}\log\omega_{\mu,j}(\bs{x})$, the returned
value is
\begin{equation}
 \widetilde{\mathcal M}_{\mu}[f](\bs{x})
 =\frac{\displaystyle
   \sum_{j\in J_{\bs{x}}}
   e^{\log\omega_{\mu,j}(\bs{x})-a(\bs{x})}P_j[f](\bs{x})}
  {\displaystyle
   \sum_{j\in J_{\bs{x}}}
   e^{\log\omega_{\mu,j}(\bs{x})-a(\bs{x})}}.
 \label{eq:log-sum-exp-blend}
\end{equation}
For vectorized evaluation, the maximum and the two scaled sums in \eqref{eq:log-sum-exp-blend}
are updated incrementally.  This avoids overflow and underflow without
modifying the mathematical operator.  If no active stencil is found, the code
returns an undefined value and issues a warning; under the hypotheses of
Lemma~\ref{lem:coverage}, this situation cannot occur when $H>h_{X,\TT}$.

\begin{algorithm}[H]
\caption{Stable evaluation of the CSMS operator}
\label{alg:stable-csms}
\begin{algorithmic}[1]
\Require Nodes $X$, data $f$, stencils $\Sigma$, degree $d$, parameter $\mu$,
upper fill-distance bound $H$, evaluation point $\bs{x}$
\Ensure $\widetilde{\mathcal M}_{\mu}[f](\bs{x})$
\If{$\bs{x}=\bs{x}_i$ for some data node}
  \State \Return $f(\bs{x}_i)$.
\EndIf
\State $\mathcal J\gets\varnothing$.
\For{$\sigma_j\in\Sigma$}
  \State Compute $\overline{\bs{x}}_j$, $\delta_j$, $V_j$, and solve
  $V_j\bs{c}_j=\bs{f}_j$.
  \State $R_j\gets\operatorname{diam}_E(\sigma_j)+H$.
  \If{$d_E(\bs{x},\bs{x}_{j_k})<R_j$ for every $k=1,\ldots,m_d$}
    \State Evaluate $P_j[f](\bs{x})$ and $\log\omega_{\mu,j}(\bs{x})$ by
    \eqref{eq:local-polynomial-evaluation} and \eqref{eq:log-unnormalized-weight}.
    \State $\mathcal J\gets\mathcal J\cup\{j\}$.
  \EndIf
\EndFor
\State $a\gets\max_{j\in\mathcal J}\log\omega_{\mu,j}(\bs{x})$.
\State \Return the normalized value in \eqref{eq:log-sum-exp-blend}.
\end{algorithmic}
\end{algorithm}

\section{Conclusion}

We have presented a compactly supported multinode Shepard interpolation
method for scattered data on the torus embedded in $\mathbb{R}^3$.
The construction combines minimal local polynomial interpolation,
compactly supported partition-of-unity weights, and a polynomial basis
adapted to the algebraic structure of the torus. The redundancies
induced by the quartic defining equation are removed through the
quotient by the corresponding ideal and the construction of a
Gröbner-reduced monomial basis. The local equivalence between the
periodic parameter distance and the Euclidean distance inherited from
the ambient space justifies the use of Euclidean distances within
sufficiently small supports.

Using a smooth normal extension in a tubular neighbourhood of the
torus and a local ambient Taylor argument, we established the error
estimate
\[
 \bigl\|f-\widetilde {\mathcal{M}}_{\mu,h}[f]\bigr\|_{L^\infty(\mathbb T)}
 \le
 C_{\mathbb T,d}\,
 \Lambda_h\,\rho_h^{d+1}
 \|f\|_{C^{d+1}(\mathbb T)}.
\]
Under uniform locality and stability assumptions on the interpolation
stencils, this yields convergence of order $d+1$ with respect to the
fill distance.

The numerical experiments confirm polynomial reproduction and show a
marked reduction of the interpolation errors under node refinement and
degree enrichment. The additional diagnostics for $d=5$ and $d=6$
show that, although the final square Vandermonde matrices become
increasingly ill-conditioned in their monomial representation, the
local factorizations and solves retain backward errors close to machine
precision. Moreover, all modes of the reduced polynomial spaces are
reproduced to nearly machine precision, while the observed
amplification of small perturbations remains moderate for the tests
considered. These results indicate that the reciprocal condition
estimate of the coefficient matrix should be distinguished from the
actual sensitivity of the local interpolation operator, which is more
directly reflected by the corresponding Lebesgue factors. They provide
evidence of high numerical accuracy for the present smooth data, but
do not constitute a proof of uniform high-degree stability.

Finally, the application to Computational Fluid Dynamics data shows
that the method accurately reconstructs the velocity components, the
velocity magnitude obtained from the interpolated components, and the
associated tangent velocity field through an orthonormal lifting. The
results demonstrate that compactly supported multinode Shepard
operators provide an effective interpolatory meshfree tool for
scattered data on toroidal geometries.

\appendix
\section{Generation of the reduced toroidal basis}
\label{app:buchberger}
\begin{lstlisting}[basicstyle=\ttfamily\small,breaklines=true]
baseHd[d_] := 
 Module[{count = 0, parts, str}, 
  For[i = 0, i <= Min[3, d], i++, 
   For[j = 0, j <= d - i, j++, 
    For[k = 0, k <= d - i - j, k++, parts = {};
     
     (* x *)
     If[i == 1, AppendTo[parts, "x"]];
     If[i > 1, AppendTo[parts, "x^" <> ToString[i]]];
     
     (* y *)
     If[j == 1, AppendTo[parts, "y"]];
     If[j > 1, AppendTo[parts, "y^" <> ToString[j]]];
     
     (* z *)
     If[k == 1, AppendTo[parts, "z"]];
     If[k > 1, AppendTo[parts, "z^" <> ToString[k]]];
     
     (* construction of the monomial *)
     If[parts == {}, str = "1", str = StringJoin[parts]];
     
     Print[str];
     count++;
     ]]];
  
  Print["Dimension = ", count];
  ]
\end{lstlisting}

\section*{Declarations}

\textbf{Conflict of interest.}
The authors declare that they have no conflict of interest.

\medskip
\noindent\textbf{Funding.}
This research was supported by the GNCS-INdAM 2026 project
``Metodi polinomiali e kernel per l'approssimazione da dati discreti e
integrali con software OS''.

\medskip
\noindent\textbf{Author contributions.}
Francesco Dell'Accio, Filomena Di Tommaso, Rossana Lammirato, and
Francesco Larosa contributed equally to the conception and scientific
development of this work. The theoretical analysis, methodological
development, software implementation, numerical experimentation,
validation, interpretation of the results, and preparation and revision
of the manuscript were carried out collaboratively. All authors read
and approved the final manuscript and agree to be accountable for all
aspects of the work.

\medskip
\noindent\textbf{Acknowledgements.}
This research was carried out as part of RITA ``Research ITalian Network on
Approximation'' and as part of the UMI group ``Teoria dell'Approssimazione e
Applicazioni''. The authors are members of the INdAM-GNCS Research Group.

\medskip
\noindent\textbf{Data availability.}
The CFD data used in the numerical experiments are cited in the manuscript.
Implementation files and derived data can be made available upon reasonable
request.

\bibliographystyle{spmpsci}
\bibliography{references}

@book{Cox,
  author    = {Cox, David A. and Little, John and O'Shea, Donal},
  title     = {{I}deals, {V}arieties, and {A}lgorithms},
  publisher = {Springer},
  year      = {2007}
}

@book{BochnakCosteRoy,
  author    = {Bochnak, Jacek and Coste, Michel and Roy, Marie-Fran{\c{c}}oise},
  title     = {Real Algebraic Geometry},
  publisher = {Springer},
  year      = {1998}
}

@book{Wendland,
  author    = {Wendland, Holger},
  title     = {{S}cattered {D}ata {A}pproximation},
  series    = {Cambridge Monographs on Applied and Computational Mathematics},
  publisher = {Cambridge University Press},
  year      = {2005}
}

@article{Shepard,
  author  = {Shepard, Donald},
  title   = {{A} {T}wo-{D}imensional {I}nterpolation {F}unction for {I}rregularly-{S}paced Data},
  journal = {Proceedings of the 23rd ACM National Conference},
  pages   = {517--524},
  year    = {1968}
}

@article{DellAccioSphere,
  author  = {Dell'Accio, Francesco and Di Tommaso, Filomena},
  title   = {{I}nterpolation of {S}cattered {D}ata on the {S}phere by {M}ultinode {S}hepard {O}perators},
  journal = {J. Sci. Comput.},
  volume  = {104},
  pages   = {96},
   number={3},
  year    = {2025}
}

@article{Farwig,
  author  = {Farwig, R{"u}diger},
  title   = {{R}ate of {C}onvergence of {S}hepard's {G}lobal {I}nterpolation {F}ormula},
  journal = {Math. Comput.},
  volume  = {46},
  number  = {174},
  pages   = {577--590},
  year    = {1986}
}

@article{DiTommaso,
  author  = {Dell'Accio, Francesco and Di Tommaso, Filomena},
  title   = {{S}cattered {D}ata {I}nterpolation by {S}hepard's {L}ike {M}ethods: {C}lassical {R}esults and {R}ecent {A}dvances},
  journal = {Dolomites Res. Notes Approx.},
  volume  = {10},
  pages   = {32--39},
  year    = {2017}
}

@article{Bos,
  author  = {Bos, Len and De Marchi, Stefano and Sommariva, Alvise and Vianello, Marco},
  title   = {{C}omputing {M}ultivariate {F}ekete and {L}eja {P}oints by {N}umerical {L}inear {A}lgebra},
  journal = {SIAM J. Numer. Anal.},
  volume  = {48},
  number  = {5},
  pages   = {1984--1999},
  year    = {2010}
}

@article{Tien1997,
  author = {Wong, Tien-Tsin and Luk, Wai-Shing and Heng, Pheng-Ann},
  title = {Sampling with {H}ammersley and {H}alton {P}oints},
  journal = {J. Graph. Tools},
  volume = {2},
  number = {2},
  pages = {9--24},
  year  = {1997}
}

@article{gerris,
  author = {Popinet, S.},
  title = {{F}ree {C}omputational {F}luid {D}ynamics},
  journal = {ClusterWorld},
  year = {2004},
  volume = {2},
  number = {6},
  url = {http://gfs.sf.net/}
}

@article{Jakob,
  author = {Jakob, Jakob and Gross, Markus and G{\"u}nther, Tobias},
  title = {A {F}luid {F}low {D}ata {S}et for {M}achine {L}earning and its {A}pplication to {N}eural {F}low {M}ap {I}nterpolation},
  journal = { IEEE Trans. Vis. Comput. Graph.},
  year = {2020}
}

@article{dell2024multinode,
  title={The {M}ultinode {S}hepard {M}ethod: {MATLAB} {I}mplementation},
  author={Dell'Accio, Francesco and Di Tommaso, Filomena and Larosa, Francesco},
  journal={J. approx. softw.},
  volume={1},
  number={2},
  year={2024}
}

@article{dellaccio2019rate,
  title={Rate of convergence of multinode {S}hepard operators},
  author={Dell'Accio, Francesco and Di Tommaso, Filomena},
  journal={Dolomites Res. Notes Approx.},
  volume={12},
  pages={1--6},
  year={2019}
}

@book{Lee,
  author    = {Lee, John M.},
  title     = {{I}ntroduction to {S}mooth {M}anifolds},
  publisher = {Springer},
  year      = {2012}
}

@article{hangelbroek2026,
  title={Generalized local polynomial reproductions},
  author={Hangelbroek, Thomas and Rieger, Christian and Wright, Grady B},
  journal={Found. Comput. Math.},
  pages={1--45},
  year={2026},
  doi={10.1007/s10208-026-09751-z},
  publisher={Springer}
}

@article{CurtisReid1972,
   author = {Curtis, A. R. and Reid, J. K.},
    title = {On the {A}utomatic {S}caling of {M}atrices for {G}aussian {E}limination},
    journal = {IMA J. Appl. Math.},
    volume = {10},
    number = {1},
    pages = {118-124},
    year = {1972}
}

@book{Higham2002,
  author= {Higham, N. J.},
  title = {Accuracy and Stability of Numerical Algorithms},
  edition = {2},
  publisher= {SIAM, Philadelphia},
  year = {2002}
}

@article{BosNewton2011,
  author  = {Bos, Len and De Marchi, Stefano and
             Sommariva, Alvise and Vianello, Marco},
  title   = {On {M}ultivariate {N}ewton {I}nterpolation at {D}iscrete {L}eja {P}oints},
  journal = {Dolomites Res. Notes Approx.},
  volume  = {4(Special Issue)},
  pages   = {15--20},
  year    = {2011}
}

@article{DellAccioDiTommasoSiar2021,
  author= {Dell'Accio, F. and Di Tommaso, F. and Siar, N.},
  title = {On the numerical computation of bivariate {L}agrange polynomials},
  journal= {Appl. Math. Lett.},
  volume= {112},
  pages = {106845},
  year = {2021}
}

\end{document}